\documentclass[11pt]{amsart}

\usepackage[OT2, T1]{fontenc}
\usepackage{url}
\usepackage{amsmath}
\usepackage{array}
\usepackage{graphicx}
\usepackage{amsfonts}
\usepackage{amssymb}
\usepackage{amstext}
\usepackage{amsthm}
\usepackage{hyperref}
\usepackage{colonequals}
\usepackage{enumitem}
\usepackage[alphabetic,lite]{amsrefs}
\usepackage{cleveref}
\usepackage[all,cmtip]{xy}
\usepackage{fullpage}
\usepackage{appendix}
\usepackage{verbatim}

\numberwithin{equation}{subsection}

\newtheorem{theorem}{Theorem}
\newtheorem{lemma}[theorem]{Lemma}
\newtheorem{proposition}[theorem]{Proposition}
\newtheorem{corollary}[theorem]{Corollary}

\newtheorem{conj}[theorem]{Conjecture}

\theoremstyle{definition}
\newtheorem{defn}[theorem]{Definition}

\theoremstyle{remark}
\newtheorem{remark}[theorem]{Remark}
\newtheorem*{Remark}{Remark}

\newtheorem{para}[theorem]{}

\newcommand{\bC}{\mathbb{C}}
\newcommand{\bD}{\mathbb{D}}

\newcommand{\bF}{\mathbb{F}}
\newcommand{\bG}{\mathbb{G}}

\newcommand{\bL}{\mathbb{L}}

\newcommand{\bP}{\mathbb{P}}
\newcommand{\bQ}{\mathbb{Q}}

\newcommand{\bZ}{\mathbb{Z}}

\newcommand{\bbA}{\mathbf{A}}
\newcommand{\bbB}{\mathbf{B}}
\newcommand{\bbC}{\mathbf{C}}
\newcommand{\bbD}{\mathbf{D}}
\newcommand{\bbE}{\mathbf{E}}

\newcommand{\bbG}{\mathbf{G}}

\newcommand{\cA}{\mathcal{A}}

\newcommand{\cD}{\mathcal{D}}

\newcommand{\cM}{\mathcal{M}}

\newcommand{\cO}{\mathcal{O}}

\newcommand{\cT}{\mathcal{T}}

\newcommand{\cZ}{\mathcal{Z}}

\newcommand{\fa}{\mathfrak{a}}
\newcommand{\fb}{\mathfrak{b}}
\newcommand{\fc}{\mathfrak{c}}
\newcommand{\fd}{\mathfrak{d}}
\newcommand{\fe}{\mathfrak{e}}

\newcommand{\cris}{\mathrm{cris}}

\newcommand{\KS}{{\mathrm{KS}}}

\newcommand{\Finf}{F_{\infty}}

\newcommand{\el}{l(n)}

\DeclareMathOperator{\GL}{GL}
\DeclareMathOperator{\SL}{SL}
\DeclareMathOperator{\GSp}{GSp}

\DeclareMathOperator{\SO}{SO}
\DeclareMathOperator{\GSpin}{GSpin}
\DeclareMathOperator{\spin}{GSpin_{3,2}}
\DeclareMathOperator{\Spf}{Spf}
\DeclareMathOperator{\Mp}{Mp}

\DeclareMathOperator{\End}{End}

\DeclareMathOperator{\Spec}{Spec}

\DeclareMathOperator{\Span}{Span}

\DeclareMathOperator{\Fil}{Fil}
\DeclareMathOperator{\disc}{disc}
\DeclareMathOperator{\diag}{diag}
\DeclareMathOperator{\Nm}{Nm}

\DeclareMathOperator{\Id}{Id}

\DeclareMathOperator{\Tr}{Tr}

\DeclareMathOperator{\rk}{rk}
\DeclareMathOperator{\Frob}{Frob}
\DeclareMathOperator{\gr}{gr}
\DeclareMathOperator{\Pic}{Pic}

\usepackage[usenames,dvipsnames]{color}  %color comments
\newcommand{\yunqing}[1]{{\color{Blue} \sf  Yunqing: [#1]}}

\begin{document}

\title{Reductions of abelian surfaces over global function fields}

\author{Davesh Maulik}
\author{Ananth N. Shankar}
\author {Yunqing Tang}

\maketitle

\begin{abstract}
    Let $A$ be a non-isotrivial ordinary abelian surface over a global function field with good reduction everywhere. Suppose that $A$ does not have real multiplication by any real quadratic field with discriminant a multiple of $p$. We prove that there are infinitely many places modulo which $A$ is isogenous to the product of two elliptic curves.
\end{abstract}
\setcounter{tocdepth}{1}
\tableofcontents

\section{Introduction}

\subsection{The main results}

Let $p$ be an odd prime and let $\cA_2$ denote the moduli stack of principally polarized abelian surfaces over $\bF_p$.
We view $\cA_2$ as (the special fiber of the canonical integral model of) a GSpin Shimura variety and let $Z(m)$ denote the Heegner divisors in $\cA_2$; more precisely, $Z(m)$ parametrizes abelian surfaces with a special endomorphism $s$ such that $s\circ s$ is the endomorphism given by multiplication by $m$ (see \S\ref{sec_sp_end}). 
\begin{theorem}\label{thm_main}
Assume $p\geq 5$.
Let $C$ be an irreducible smooth quasi-projective curve with a finite morphism  $C\rightarrow \cA_{2,\overline{\bF}_p}$. Assume that the generic point of $C$ corresponds to an ordinary abelian surface. 
\begin{enumerate}
    \item\label{item_S} 
    If the image of $C$ is not contained in any Heegner divisor $Z(m)$, and if $C$ is projective, then there exist infinitely many $\bar{\bF}_p$-points on $C$ which correspond to non-simple abelian surfaces.
    \item\label{item_H}
    If the image of $C$ is contained in some $Z(m)$ such that $p\nmid m$, then there exist infinitely many $\bar{\bF}_p$-points on $C$ which correspond to abelian surfaces isogenous to self-products of elliptic curves
\end{enumerate}
\end{theorem}
%Old version: Assume $p\geq 5$. Let $C$ be an irreducible quasi-projective curve in $\cA_{2,\overline{\bF}_p}$. Assume that the generic point of $C$ corresponds to an ordinary abelian surface. 
%\begin{enumerate}
%    \item\label{item_S} 
%    If $C$ is not contained in any Heegner divisor $Z(m)$, and if $C$ is projective, then there exist infinitely many $\bar{\bF}_p$-points on $C$ which correspond to non-simple abelian surfaces.
%    \item\label{item_H}
%    If $C$ is contained in some $Z(m)$ such that $p\nmid m$, then there exist infinitely many $\bar{\bF}_p$-points on $C$ which correspond to abelian surfaces isogenous to self-products of elliptic curves
%\end{enumerate}

%SAVE:Note that the proof that we gave in this paper doesn't apply directly to the product of modular curve case because we are using Borcherds theory on the open moduli space and for product of j-lines, the Chow group is trivial--this is consistent with that Borcherds theory indicates that the generating series is a weight $2$ modular form with level $1$ and hence must be the zero form; the local estimates are fine, so maybe we could make a remark in the proof sections.
In \Cref{thm_main}(\ref{item_H}), note that the elliptic curve may vary for these points.  An equivalent statement is that there exist infinitely many $\bar{\bF}_p$-points on $C$ which correspond to abelian surfaces whose N\'eron--Severi ranks are strictly larger than that of the generic point of $C$. Note that in the case (\ref{item_H}), any irreducible component of $Z(m)\subset \cA_2$ is an irreducible component of a Hecke translate of some Hilbert modular surface associated to the real quadratic field $F=\bQ(\sqrt{m})$ (if $m$ is a square number, then we obtain a Hecke translate of the self-product of the modular curve). 

\begin{Remark}
The assumption that the generic point is ordinary is necessary (especially if we formulate the theorem in terms of the N\'eron--Severi rank). For instance, in the case (\ref{item_H}), we may take $C$ to be an irreducible component of the non-ordinary locus. If $p$ is inert in $F$, then all the points on $C$ are supersingular and the N\'eron--Severi rank does not jump. If $p$ is split in $F$, then the only points where the N\'eron--Severi rank jumps are the finitely many supersingular points.
\end{Remark}

\begin{Remark}\label{whynotcompact} We make the (technical) assumption that $C$ is projective in (\ref{item_S}) because the Heegner divisors $Z(m)$ are all non-compact and we plan to remove this assumption in future work. On the other hand, the Hilbert modular surfaces considered in (\ref{item_H}) do contain compact special divisors (see the second half of \S\ref{sec_sp_end} for the definitions of special divisors in the Hilbert case, and \S\ref{def_setT} for a criterion of when these special divisors are compact) whose $\overline{\bF}_p$ points parameterize abelian surfaces isogenous to a self-product of elliptic curves. By working exclusively with these compact special divisors, we no longer need assume that $C$ is projective. 
\end{Remark}

\begin{remark}\label{rmk_realquad}
A modification of our argument shows that with the same assumption in (\ref{item_S}), for a fixed real quadratic number field $F$, there are infinitely many ordinary $\bar{\bF}_p$-points on $C$ such that the corresponding abelian surfaces admit real multiplication by $F$. %$F=\bQ(\sqrt{D})$.
Here we need to assume $p\geq 7$ if $p$ is ramified in $F$. Otherwise, $p\geq 5$ is enough.
\end{remark}

To prove \Cref{thm_main}(\ref{item_S}), we consider the intersection number of $C$ and $Z(\ell^2)$, where $\ell$ is a varying prime number. If we consider $Z(\ell)$ with $\ell \equiv 3 \bmod 4$ instead, we prove
\begin{theorem}\label{thm_max}
Suppose we have the same assumptions as in \Cref{thm_main}(\ref{item_S}). Then there are infinitely many ordinary $\bar{\bF}_p$-points on $C$ such that, for each of these points, the corresponding abelian surface admits real multiplication by the ring of integers of some real quadratic field (note that the quadratic fields may vary for these points).
\end{theorem}

It would be interesting to find $\bar{\bF}_p$-points of complex multiplication by maximal orders, but our current method only asserts real multiplication by  maximal orders.

\subsection{Previous work and heuristics}
\Cref{thm_main} is a generalization of \cite[Proposition 7.3]{CO06}, where Chai and Oort proved \Cref{thm_main}(\ref{item_H}) with $\mathcal{A}_1 \times \mathcal{A}_1$ taking the place of a Hilbert modular surface. Their proof crucially uses the product structure of the Shimura variety, as well as the product structure of the Frobenius morphism. Following the discussion in \S 7 of \cite{CO06}, \Cref{thm_main} is related to a bi-algebraicity conjecture. See \S 1.4 for more details. 

We offer the following heuristic for \Cref{thm_main}(\ref{item_S}).  
Using Honda and Tate's classification of $\bF_q^n$-isogeny classes of abelian varieties in terms of Weil-$q^n$ numbers, the number of $\bF_q$-isogeny classes of abelian varieties is seen to equal $q^{n(3/2 + o(1))}$. %Following results of Achter and Howe in \cite{AH17}, the number of non-simple principally polarized abelian surfaces over $\bF_{q^n}$ is roughly $q^{n(5/2 + o(1))}$, and the number of non-simple isogeny classes is roughly $q^{n(1 + o(1))}$\yunqing{Ananth will check; maybe by Weil number count}. 
Similarly, the number of \emph{split} $\bF_{q^n}$-isogeny classes in $\cA_2$ is seen to equal $q^{n(1 + o(1))}$. If we treat the map from $C(\bF_{q^n})$ to the set of $\bF_{q^n}$-isogeny classes as a random map, we expect that the number of $\bF_{q^n}$ points of $C$ which are not simple is around $q^{n/2(1 + o(1))}$. Letting $n$ approach infinity, this heuristic suggests that infinitely many points of $C(\overline{\bF}_q)$ that are split.
There are analogous questions in other settings.  For the case of equicharacteristic $0$, these results are well known (for instance, the density of Noether--Lefschetz loci is discussed in \cite{Voisin}).  In mixed characteristic, the analogue of \Cref{thm_main}(\ref{item_H}) is treated in \cite{Ch}, \cite{Ananth17}. The major difference between \Cref{thm_main} and these other cases is that the ordinary generic point assumption is crucial since the result is simply false otherwise (as remarked in \S 1.1). 

Indeed, this difference hints at the key difficulty in our setting, which is that the local intersection number at a supersingular point is of the same magnitude as the total intersection number, which makes the approach more complicated than that of \cite{Ananth17}; we discuss this in more detail in \S 1.3.

\subsection{Proof of the main results}\label{sec_intro_pf}
We view both Hilbert modular surfaces and the Siegel three-fold as GSpin Shimura varieties attached to a quadratic space $(V,Q)$.  In each setting, we have a notion of special endomorphisms and special divisors and, for simplicity, we use the same notation $Z(m)$.

The main idea of the proof is to compare the global and local intersection numbers of $C . Z(m)$\footnote{Although $C$ is not a substack of $\cA_2$, we may define $C.Z(m)$ as the degree of the pull back of $Z(m)$ via $C\rightarrow \cA_{2,\bar{\bF}_p}$ when $C$ is projective.} for appropriate sequences of $m$ and show it is not possible for finitely many points to account for the total global intersection as $m$ increases. 

More precisely,
\begin{enumerate}
    \item The global intersection number $I(m) \colonequals C . Z(m)$ is controlled by Borcherds theory \cite{Bor98} (see also \cite{Davesh} and \cite{HMP}).
    \item We prove that as $m\rightarrow \infty$, the total local contribution from supersingular points is at most $\frac{11}{12}I(m)$ by studying special endomorphisms.\footnote{Indeed, the ratio depends on $p$ and it goes to $1/2$ as $p\rightarrow \infty$.}
    \item We prove that the local contribution from a non-supersingular point is $o(I(m))$ as $m\rightarrow \infty$. 
\end{enumerate}
This allows us to conclude that, as $m\rightarrow \infty$, more and more points of $C$ contribute to the intersection $C . Z(m)$. In order to prove \Cref{thm_main}(\ref{item_S}), the sequence of $m$ will consist only of squares, and in order to prove \Cref{thm_max}, the sequence will consist only of primes. Note that in $\cA_2$, the Heegner divisor $Z(m)$ for square $m$  parametrizes abelian surfaces which are not geometrically simple, thereby allowing us to deduce \Cref{thm_main}(\ref{item_S}). Similar arguments allow us to deduce part \Cref{thm_main}(\ref{item_H}), and also \Cref{thm_max}. 

Compared to the number field situation, the main difficulty of the positive characteristic function field case is that the local contributions at supersingular points are of the same magnitude as the global contribution. More precisely, taking the Hilbert case as an example, Borcherds theory implies that the generating series of $Z(m)$ is a non-cuspidal modular form of weight $2$; 
on the other hand, the theta series attached to the special endomorphism lattice at a supersingular point is also a non-cuspidal weight $2$ modular form since the lattice is of rank $4$.
Therefore, even without considering higher intersection multiplicities, the local intersection number of $C.Z(m)$ at a supersingular point is also of the same magnitude as the growth rate of Fourier coefficients of an Eisenstein series of weight $2$. %$O(m)$. 

\subsubsection*{Bounding the local contribution from a supersingular point}
Let $A \rightarrow C$ denote the family of principally polarized abelian surfaces induced from a morphism  $C\rightarrow \cA_{2,\overline{\bF}_p}$, 
and let $\Spf \overline{\bF}_p[[t]] \rightarrow C$ denote the formal neighborhood of a supersingular point. For a special endomorphism $s$ such that $s\circ s=m$, we say that $s$ is of norm $m$.

The local contribution to $C . Z(m)$ from this supersingular point equals $\sum_{n=0}^{\infty} r_n(m)$, where $r_n(m)$ is the number of special endomorphisms of $A \mod t^{n+1}$ with norm $m$. Therefore, in order to bound the local contribution, it suffices to prove that, as 
$n\rightarrow \infty$, 
there are many special endomorphisms of $A \mod t^n$ which decay rapidly enough (see \Cref{def_decay}, \Cref{thm_decay}, \Cref{thm_decay_sg} for precise statements). 

A similar decay result appears in the mixed characteristic setting (see \cite{Ananth17}), by a straightforward application of Grothendieck--Messing theory.  In the equicharacteristic case, however, proving our decay results is much more involved. In particular, we need to use Kisin's description \cite[\S 1.4, 1.5]{Kisin} of the $F$-crystal associated to a certain automorphic vector bundle $\bL_\cris$, whose $F$-invariant part is the lattice of special endomorphisms, in order to prove the required decay.  See \S\ref{summary_Kisin} and the proof of \Cref{thm_decay} for more details.  

We will focus on the Siegel case from now on. Let $L_0$ denote the lattice of special endomorphisms of $A \mod t$, and let $L_n \subset L_0$ be the lattice of special endomorphisms of $A \mod t^{n+1}$. These lattices are of rank $5$ and are equipped with natural quadratic forms such that $A \mod t^{n+1}$ admits a special endomorphism of norm $m$ if and only if $m$ is represented by $L_n$. Broadly speaking, we can bound the local contribution by using geometry-of-numbers techniques. To obtain the desired estimate, we choose the sequence $m$ as follows. We first prove the existence of a rank $2$ sublattice $P_n \subset L_n$ that has the following property: for all $m$ bounded by an appropriate function of $n$, the abelian surface $A \mod t^{n+1}$ has a special endomorphism of norm $m$ only if the quadratic form restricted to $P_n$ represents $m$. This fact follows from the existence of a rank $3$ submodule of special endomorphisms which decay rapidly (\Cref{thm_decay,thm_decay_sg}). Furthermore, the discriminant of $P_n$ goes to infinity as $n\rightarrow\infty$. 
Therefore, the density of numbers (or primes, or prime-squares) represented by the \emph{binary} quadratic form $P_n$ approaches zero, as
$n\rightarrow \infty$.
We now pick a sequence of prime-squares $m$ none of which are represented by $P_n$ defined by the finitely many supersingular points on $C$.

The non-ordinary locus is singular at superspecial points. This allows us to prove the existence of a special endomorphism that decays ``more rapidly than expected'' (see \Cref{def_decay}(3)). Consequently, by the explicit formula of Eisenstein series in these cases by \cite{BK01}, we prove that the sum of local contributions at supersingular points is at most $11/12$ of the global contribution. 

We remark that our proof is more involved than the proof of \cite[Proposition 7.3]{CO06} because the intersection theory on Hilbert modular surfaces and Siegel three-folds is more complicated than that on the product of $j$-lines.

\subsection{Additional remarks} The key difference between the number field and function field situation is the following. Let $A$ be an abelian surface over $\mathcal{O}_K$, where $K$ is a local field. The $\bZ_p$-module of special endomorphisms of $A[p^{\infty}]$ has rank $\leq 3$. This rank equals three if and only if $A$ can be realized as the limit point (in the analytic topology) of a sequence of CM points. This can happen in the mixed characteristic case, but not in the equicharacteristic $p$ case unless $A$ is defined over a finite field.\footnote{Ordinary abelian varieties which have CM are defined over finite fields.} Thus, we have a rank $3$ decay in the Decay Lemmas (\Cref{thm_decay,thm_decay_sg}).

In the setting of higher dimensional GSpin Shimura varieties, for the same reason, we expect that generalizations of the Decay Lemma will only yield a rank $3$ $\bZ_p$-module that decays rapidly. This  has the consequence of the existence of formal curves, such that the module of special endomorphisms of the $p$-divisible group over these formal curves have large rank. An interesting bi-algebraicity question is whether such formal curves can be algebraic without being special. In the ordinary case, Chai has the following conjecture:

%has  even in  generalizations of the Decay Lemma will only yield decay of a rank-3 $\bZ_p$-submodule of special endomorphisms,  Towards the higher dimensional generalization of \Cref{thm_main}, there are extra difficulties related to conjectures along the lines of a mod $p$ Andr\'e--Oort conjecture. More precisely, in analogy with the bi-algebraicity theorem of special subvarieties used in the proof of the Andr\'e--Oort conjecture,
%Chai conjectured the following for the mod $p$ case. 
\begin{conj}[\cite{Ch03}*{Conj. 7.2, Remark 7.2.1, Prop. 5.3, Remark 5.3.1}]
Let $X$ be a subvariety in a mod $p$ Shimura variety passing through an ordinary point $P$. Assume that the formal germ of $X$ at $P$ is a formal torus in the Serre--Tate coordinates. Then $X$ is a Shimura subvariety.
\end{conj}
%This conjecture pertains to potential higher dimensional generalizations of \Cref{thm_main} as follows. Let $X$ be an algebraic family of abelian varieties inside a GSpin Shimura variety, and let $x\in X$ be a closed ordinary point. Consider the $p$-divisible group associated to the abelian scheme in a formal neighborhood of $x$. The conjecture provides a non-trivial upper bound of the $\bZ_p$-rank of the module of special endomorphisms of this $p$-divisible group.\footnote{Unfortunately, this non-trivial upper bound is still not enough for us to apply our strategy to prove \Cref{thm_main} for all orthogonal Shimura varieties.}

\subsection{Organization of paper}
In \S\ref{sec_lat+sp-end}, we recall the notion of special endomorphisms, special divisors and crystalline realization $\bL_\cris$ of the automorphic vector bundle of special endomorphisms.
In \S\ref{sec_lattices}, we recall the lattices of special endomorphisms of a supersingular point and compute $\bL_{\cris}$ on its deformation space.  
In \S\ref{sec_global}, we recall Borcherds theory and the explicit formula for the Fourier coefficients of vector-valued Eisenstein series due to Bruinier--Kuss; we use them to compare the global intersection number and the $\bmod\, t$ local intersection number at a supersingular point. Sections \S\ref{sec_decay_Hil} and \S\ref{sec_decay_Sie} are the key technical part of the paper. We prove the decay theorems for special endomorphisms, which we will use to bound the higher local intersection multiplicities at supersingular points. Section \S\ref{sec_outline} provides the outline of the main proofs and by geometry-of-numbers arguments, we prove \Cref{thm_main}(\ref{item_H}) in \S\ref{sec_pf_Hil} and prove \Cref{thm_main}(\ref{item_S}) and \Cref{thm_max} in \S\ref{sec_pf_Sie}. 

In order to get the main idea of the proof, the reader may focus on \Cref{thm_main}(\ref{item_H}) and start from \S\S\ref{sec_outline},\ref{sec_pf_Hil} and refer back to \S\S\ref{sec_lattices}-\ref{sec_decay_Hil} when necessary.

\subsection{Notation}
We write $f\asymp g$ if $f=O(g), g=O(f)$. Throughout the paper, $p$ is an odd prime.

\subsection*{Acknowledgement} We thank Johan de Jong, Keerthi Madapusi Pera, Arul Shankar, Salim Tayou, and Jacob Tsimerman for helpful discussions.
D.M. is partially supported by NSF FRG grant DMS-1159265.  Y.T. is partially supported by the NSF grant DMS-1801237.
We would like to thank the referee for valuable suggestions on the exposition of the paper.

\numberwithin{theorem}{subsection}

\section{Special endomorphisms}\label{sec_lat+sp-end}
In this section, we first introduce quadratic lattices $(L,Q)$ such that the associated GSpin Shimura varieties will be $\cA_2$ and certain Hilbert modular surfaces related to the Heegner divisors $Z(m)$. The definition of special endomorphisms and Heegner divisors are given in \S\ref{sec_sp_end}.

\subsection{The global lattice $L$}\label{sec_lat_glo}
For a quadratic $\bZ$-lattice $(L,Q)$, let $C(L)$ (resp.~$C^+(L)$) denote the (resp.~even) Clifford algebra of $L$. Let $(-)'$ denote the standard involution on $C(L)$ fixing all elements in $L$ given by $(v_1\cdots v_n)'=v_n\cdots v_1$ for $v_i\in L$. Let $V$ denote $L\otimes \bQ$ endowed with the quadratic form $Q$. There is a bilinear form $[-,-]$ on $V$ given by $[x,y]:=Q(x+y)-Q(x)-Q(y)$.

Let $L_S$ be the rank $5$ $\bZ$-lattice endowed with the quadratic form $Q(x)=x_0^2+x_1x_2-x_3x_4$ for $x=(x_0, \cdots, x_4)\in \bZ^5$. This quadratic form has signature $(3,2)$ and $L_S$ is an even lattice, maximal among $\bZ$-valued sublattices in $L_S\otimes \bQ$. For $p>2$, $L_S$ is self-dual at $p$. A direct computation shows that $C^+(L_S)\cong M_4(\bZ)$.
Let \[v_0=(1,0,0,0,0), v_1=(0,1,1,0,0), v_2=(0,1,-1,0,0), v_3=(0,0,0,1,1), v_4=(0,0,0,1,-1).\]
Then $\delta:=v_0\cdots v_4 \in C(L_S)$ lies in the center of $C(L_S)$ and $\delta'=\delta, \delta^2=1$. Therefore, there is an isomorphism between quadratic spaces given by $L_S\xrightarrow[]{\simeq} \delta L_S \subset C^+(L_S)$. (See for instance \cite[App.~A]{KR3}.)
%SAVE: the fact that $\delta$ lies in the center can be checked by that $\delta$ commutes with all $v_i$; indeed, $Q$ is diagonal with respect to $v_i$ and hence each $v_i$ anti-commutes with $v_j, j\neq i$. The quadratic form on $\delta L$ is given by $\delta v \delta v$, where the product is taken in the Clifford algebra. Since $\delta$ lies in the center, we have $\delta v \delta v= \delta^2 v^2 =Q(v)$.

Given a vector $x\in L_S$ such that $Q(x)=m, m\in \bZ_{>0}$, the orthogonal complement $x^\perp\subset L_S$ endowed with the restriction of $Q$ on $x^\perp$ is a quadratic $\bZ$-lattice of signature $(2,2)$ and let $L_H\subset x^\perp \otimes \bQ$ be a maximal lattice containing $x^\perp$. If $m$ is not a perfect square, let $F$ denote the real quadratic field $\bQ(\sqrt{m})$. A direct computation shows that there is an isomorphism $L_H\otimes \bQ\cong \bQ^2\oplus F$ such that $Q((a,b,\gamma))=ab+\Nm_{F/\bQ}\gamma$. The assumption $p\nmid m$ and $p>2$ implies that $x^\perp$ and hence $L_H$ are self-dual at $p$.  
%SAVE: write $x=(a_0,\dots, a_4)$ with $a_0^2+a_1a_2-a_3a_4=m$; if $m$ is not a square, then WLOG, WMA $a_1\neq 0$ (other cases admit similar computations). Then an element $x'\in x^\perp$ satisfies $x'_2=(a_3x'_4+a_4x'_3-a_2x'_1-2a_0x'_0)/a_1$ and we have $Q(x')=y_0^2-(m/a_1^2)y_1^2-y_3y_4$, where $y_0=x'_0-a_0x'_1/a_1, y_1=x'_1, y_3=x'_3-a_3x'_1/a_1, y_4=x'_4-a_4x'_1/a_1$. This quadratic form over $\bQ$ is compatible with the one on p.15 in \cite{HY}.

Now let $(L,Q)$ have signature $(n,2)$, and let $p$ be a prime such $(L,Q)$ is self-dual at $p$. As in \cite[\S 4.1 \S 4.2]{AGHMP}, \cite[\S 1]{KR3}, there is a GSpin Shimura variety $M$ attached to $(L,Q)$ and this Shimura variety also admits a smooth integral model $\cM$ over $\bZ_{(p)}$ since $L$ is self-dual at $p$; the Shimura variety (and its integral model) recovers the moduli space of principally polarized abelian surfaces when $L=L_S$ (see \Cref{rmk-compareKR} for details) and it is a Hilbert modular surface when $L=L_H$ (see for instance \cite[\S 2.2, \S 3.1]{HY}). We may write $M_L$ and $\cM_L$ to emphasis on the dependence on $L$.

To prove \Cref{thm_main}(1) and \Cref{thm_max}, we will take $L=L_S$ and to prove \Cref{thm_main}(2), we will take $L=L_H$. 
%SAVE: since we work with generic ordinary curve $C$, hence the special endomorphism also lift to characteristic $0$ and hence $C$ lies in the reduction of Hilbert modular surface.
\subsection{Special endomorphisms and special divisors}\label{sec_sp_end}

We first introduce the notion of special endomorphisms when $L=L_S$ and $\cM$ is the moduli space of principally polarized abelian surfaces. Given an $\cM$-scheme $S$, let $A_S$ denote the pullback of the universal principally polarized abelian surface on $\cM$ via $S\rightarrow \cM$; let $\dagger$ denote the Rosati involution on $A_S$.
\begin{defn}\label{def_sp-end}
A \emph{special endomorphism} of $A_S$ is an element $s\in \End(A_S)$ such that $s^\dagger=s$ and $\Tr s=0$, where $\Tr$ is the reduced trace on the semisimple algebra $\End(A_S)\otimes \bQ$.
\end{defn}

\begin{remark}\label{rmk-compareKR}
Our definition of special endomorphisms is essentially the same as the one given by Kudla--Rapoport (\cite[Def.~2.1, Eqn.~(2.21)]{KR3}). Indeed, as in \cite[\S\S 1-2]{KR3}, the moduli problem indicates that every $\cM$-scheme $S$ gives rise to a principally polarized abelian scheme $B_S$ over $S$ with $\iota: C^+(L)\hookrightarrow \End(B_S)$ and a polarization such that the induced Rosati involution $\dagger$ satisfies $\iota(c)^\dagger=\iota(c^T)$, where $(-)^T$ is the transpose on $C^+(L)\simeq M_4(\bZ)$ (see condition (iii) and the first paragraph of \cite[p.701]{KR3}); 
moreover, for each $\ell\neq p$, there is an isomorphism $C^+(L)\otimes \bZ_\ell \simeq T_\ell(B_S)$, where $T_\ell$ denotes the $\ell$-adic Tate module, compatible with the $C^+(L)$-action (it acts on itself via left multiplication; see \cite[p.703]{KR3}).\footnote{Although Kudla--Rapoport uses abelian schemes up to isogeny to give the moduli interpretation, one may translate it into abelian schemes up to isomorphism; see also \cite[\S 2.2]{AGHMP17}.} Therefore, via $\iota$, we have $B_S\cong A_S^4$, where $A_S$ is an abelian surface and by the compatibility of the polarization with $\iota$ (see also \cite[Eqn.~(1.9), (1.10)]{KR3}), and the polarization on $B_S$ is induced by the self-product of a principal polarization on $A_S$. Hence $\cM$ parameterizes principally polarized abelian surfaces.
%SAVE: more precisely, let $M$ be the matrix on $H^_B(B_S)$ of an element in $C^+(L)\cong M_4(\bZ)$ and let $\Phi$ denote the matrix of polarization. Then the compatibility condition says (LHS is the definition of Rosati involution) $(\Phi^{-1})^T M^T \Phi^T=M^T$, i.e., $\Phi^T$ commutes with any matrix in $M_4(\bZ)$ so it must be a diagonal block matrix.
Moreover, an element $s_B$ in $\End(B_S)\cong M_4(\End(A_S))$ commuting with $\iota(C^+(L))$ is of form $\diag(s,s,s,s)$, where an endomorphism $s$ of $A_S$. In the sense of Kudla--Rapoport, such $s_B$ is special if and only if it is traceless and fixed by the Rosati involution on $B_S$; this is equivalent to that $s$ is traceless and fixed by the Rosati involution on $A_S$. Therefore, our definition is the same as that of Kudla--Rapoport.
\end{remark}

\begin{defn}\label{def_Lcris_S}
Let $\bD$ denote the Dieudonn\'e crystal over $\cM_{\bF_p}$ (i.e., the first relative crystalline homology of the universal family of principally polarized abelian surface over $\cM_{\bF_p}$). Let $\bL_\cris\subset \End(\bD)$ denote the sub-crystal of trace $0$ elements fixed by the Rosati involution.\footnote{Note that Frobenius is not an endomorphism on $\End(\bD)$, due to the existence of negative slopes. However, we will abuse terminology, and still treat $\End(\bD)$ and $\bL_{\cris}$ as $F$-crystals in the sense that we will view Frobenius as a map from $\End(\bD)$ to $\End(\bD)[1/p]$, while remembering the integral structure, and similarly for $\bL_{\cris}$.} 
\end{defn}

By definition, when $S$ is a $\cM_{\bF_p}$-scheme, an element $s\in \End(A_S)$ is a special endomorphism if and only if the crystalline realization of $s\in \End(\bD_S)$ lies in $\bL_{\cris,S}$.

\begin{defn}\label{def_spend_pdiv}
For the $p$-divisible group $A_S[p^\infty]$, we say $s\in \End(A_S[p^\infty])$ is a \emph{special endomorphism} if the image of $s$ in $\End(\bD_S)$ lies in $\bL_{\cris,S}$.
\end{defn}

\begin{remark}\label{compare_S}
In \cite[\S 4.14]{MP16}, there is a definition of $\bL_\cris$ as a direct summand of the endomorphism of the first relative crystalline cohomology of the Kuga--Satake abelian scheme over $\cM_{\bF_p}$. More precisely, the left multiplication of $\GSpin(V,Q)\subset C^+(V)^\times$ acting on $C(V)$ induces a variation of Hodge structures on $C(V)$ over $M$; this gives rise to the Kuga--Satake abelian scheme $A^\KS$ over $M$ and the Kuga--Satake abelian scheme extends over $\cM$. The $8$-dimensional abelian scheme considered by Kudla--Rapoport is a sub abelian scheme of $A^\KS$ via the natural embedding $C^+(V)\subset C(V)$. (Note that in \cite{KR3}, $\gamma\in\GSpin(V,Q)$ acts on $C^+(V)$ by the right multiplication by $\gamma$ and $C^+(V)$ acts on $C^+(V)$ by left multiplication, which is opposite to the convention in \cite{MP16}. This difference is due to the different choices of the symplectic pairing on $C^+(V)$ and $C(V)$ in \cite[(1.9)]{KR3} and \cite[\S 1.6]{MP16}. If we use the symplectic pairing in \cite{MP16} for the discussion in \cite{KR3}, then we obtain similar results as in \cite{KR3} but with the convention consistent with that in \cite{MP16}.)

Let $\bD^\KS$ denote the Dieudonn\'e crystal of $A^\KS$ over $\cM_{\bF_p}$; Madapusi Pera defined $\bL_\cris \subset \End (\bD^\KS)$ by the crystalline realization of the absolute Hodge cycle induced by the $\GSpin(V,Q)$-invariant idempotent which realizes $V\subset \End(C(V))$ as a direct summand. Since the element $\delta$ given in \S\ref{sec_lat_glo} lies in the center of $C(L)$, then it induces an isomorphism $\End(C(L))\supset L \cong \delta L \subset \End(C^+(L))$ compatible with $\GSpin(V,Q)$-action. 
Therefore, $\delta$ induces an isomorphism between the crystals $\bL_\cris$ in our sense and the one in the sense of Madapusi Pera; in particular, the notions of special endomorphisms coincide under the identification via $\delta$. Also, for a special endomorphism $s$ in both cases, $s\circ s$ is a scalar multiple $Q(s)$ on the suitable abelian scheme; since $\delta^2=1$, hence $Q(s)$ remains the same for images of $s$ under various identification of special endomorphisms. By \cite[Lem.~5.2]{MP16}, $Q(s)>0$ for all nonzero special endomorphism $s$.
\end{remark}

\begin{defn}\label{def_spdiv_S}
For $m\in \bZ_{>0}$, the \emph{special divisor} $\cZ(m)$ is the Deligne--Mumford stack over
$\cM$ with functor of points $\cZ(m)(S) = \{s\in \End(A_S) \text{ special } | Q(s) = m\}$ for any $\cM$-scheme $S$. We use the same notation for the image of $\cZ(m)$ in $\cM$. By for instance \cite[Prop.~4.5.8]{AGHMP}, $\cZ(m)$ is flat over $\bZ_{(p)}$ and hence $\cZ(m)_{\bF_p}$ is still a divisor of $\cM_{\bF_p}$; we denote $\cZ(m)_{\bF_p}$ by $Z(m)$.
\end{defn}

\begin{lemma}
Every $\bar{\bF}_p$-point of $Z(m^2)$ corresponds to a geometrically non-simple abelian surface.
\end{lemma}
\begin{proof}
Let $s$ be a special endomorphism of an abelian surface $A$ such that $s\circ s=[m^2]$. Then $(s-[m])\circ (s+[m])=0$. Since $\Tr s=0$, then $s\pm [m]\neq 0$ and hence $s\pm [m]$ are not invertible. Then $\ker (s-[m])$ defines a non-trivial sub abelian scheme of $A$.
\end{proof}

We now discuss the case when $L=L_H$. We keep the same notation as in \S\ref{sec_lat_glo}. For simplicity, we first discuss the case when $L_H=x^\perp$, where $x\in L_S$ and $Q(x)=m$ with $p\nmid m$; for the general case, the following discussion still holds true when replacing endomorphisms with suitable elements in $\End \otimes \bQ$ (see the end of this subsection). When $L_H=x^\perp\subset L_S$, the Shimura variety (and its integral model) $\cM_{L_H}$ defined by $L_H$ is naturally a sub-Shimura variety of $\cM_{L_S}$, the moduli space of principally polarized abelian surfaces, and hence a point on $\cM_{L_H}$ corresponds to a polarized abelian surface with real multiplication by $\cO:=\bZ[x]/(x^2-m)$. Let $\sigma$ denote the ring automorphism on $\cO$ satisfying $x^{\sigma} = -x$. As before, let $S$ be a $\cM_{L_H}$-scheme, and let $A_S$ denote the abelian surface over $S$ with real multiplication by $\cO$.
\begin{defn}[{\cite[\S 3.1 p.26]{HY}}]
A \emph{special endomorphism} (resp. \emph{special quasi-endomorphism}) of $A_S$ is an element $s\in \End(A_S)$ (resp. $s\in \End(A_S)\otimes \bQ$) such that $s^\dagger=s$ and $s\circ f =f^\sigma \circ s$ for all $f\in \cO$.
\end{defn}

We still use $\bD$ to denote the pullback to $\cM_{L_H, \bF_p}$ the Dieudonn\'e crystal over $\cM_{L_S, \bF_p}$ in \Cref{def_Lcris_S}; since the abelian surfaces over $\cM_{L_H}$ admit an $\cO$-action, the Dieudonn\'e crystal $\bD$ is also endowed with an $\cO$-action.
\begin{defn}\label{def_Lcris_H}
Let $\bL_\cris\subset \End(\bD)$ denote the sub-crystal of elements $v$ fixed by Rosati involution and $s\circ f=f^\sigma \circ s$ for all $f\in \cO$. For the $p$-divisible group $A_S[p^\infty]$, we say $s\in \End(A_S[p^\infty])$ is a \emph{special endomorphism} if the image of $s$ in $\End(\bD_S)$ lies in $\bL_{\cris,S}$.
\end{defn}

\begin{remark}
By \Cref{compare_S} and \cite[Prop.~2.5.1, Prop.~2.6.4]{AGHMP17}, in order to show that the above definitions of special endomorphisms and $\bL_\cris$ can be identified with those by Madapusi Pera, we only need to show that for an endomorphism $s$ (of either the abelian surface or of its Dieudonn\'e crystal $\bD$) fixed by the Rosati involution is traceless and orthogonal to $x$ if and only if $s\circ x=-x\circ s$. To see this, note that if $\Tr s=0$, then $s\perp x$ if and only if $Q(s+x)-Q(s)-Q(x)=s\circ x+x\circ s=0$; on the other hand, if $s\circ x=-x\circ s$, then $x^{-1}\circ s \circ x =-s$ and hence $\Tr s=0$.
\end{remark}

\begin{para}\label{def_spdiv_H}
In general (i.e. when $x^\perp \subsetneq L_H$), we may still use the same definition for $\bL_\cris$ and special endomorphisms of $p$-divisible groups, as $x^\perp$ is self-dual at $p$ and hence $x^\perp \otimes \bZ_p = L_H \otimes \bZ_p$. On the other hand,
we consider \emph{special quasi-endomorphisms} $s\in \End(A_S)\otimes \bQ$ which satisfy the following integrality condition: the $\ell$-adic realizations of $s$ lie in $L_H\otimes \bZ_\ell \subset \End(T_\ell(A_S)\otimes \bQ_\ell)$ for all $\ell\neq p$ and the crystalline realizations of $s$ lie in $\bL_{\cris, S}$. As in \Cref{def_spdiv_S}, the \emph{special divisor} $\cZ(m)$ is the Deligne--Mumford stack over $\cM_{L_H}$ with $\cZ(m)(S)$ given by \[\{s\in \End(A_S)\otimes\bQ \text{ special quasi-endomorphism satisfying the integrality condition above } \mid Q(s)=m\}\]
for any $\cM$-scheme $S$. By the proof of \cite[Prop.~4.5.8]{AGHMP}, where they used \cite[Prop.~5.21]{MP16}, $\cZ(m)$ is flat over $\bZ_p$. We use $Z(m)$ to denote the image of $\cZ(m)_{\bF_p}$ in $\cM_{L_H,\bF_p}$, which is a divisor in $\cM_{L_H,\bF_p}$. 
\end{para}
%For the Hilbert case, cite Howard--Yang's book and cite ST for non-simpleness. To identify the lattice of special endomorphisms, we use Siegel case and compare divisors.

\subsection{Lattices of special endomorphisms of supersingular points}
For a fixed supersingular point, let $A$ denote the abelian surface attached to this point. 

\begin{defn}\label{def_L''}
Let $L''$ denote the $\bZ$-lattice of special endomorphisms of $A$ (resp. special quasi-endomorphisms when $L=L_H$). Let $L''\subset L'\subset L''\otimes\bQ$ be a $\bZ$-lattice which is maximal at all $\ell\neq p$ and $L''\otimes \bZ_p=L'\otimes \bZ_p$. Let $Q'$ denote the natural quadratic form on $L'$ given by $s\circ s=[Q'(s)]\in \End(A)\otimes \bQ$. By the positivity of the Rosati involution, $Q'$ is positive definite (see for instance \cite[Lem.~5.12]{MP16}).
\end{defn}
Even though there seem to be choices involved here, we will see that for our computation, these choices do not matter and the result will only depend on the Ekedahl--Oort stratum where the supersingular point lies in. The information of $L'\otimes \bZ_p$ will be provided in \S\ref{sec_lattices}. 
%SAVE: more precisely, similar to the global lattice, all we need are all localizations of $L'$; this is all we need to give an asymptotic of the Fourier coefficients in \S\ref{sec_global}.

\begin{lemma}\label{lem_lat_ell}
$(L'\otimes \bZ_\ell, Q')\cong (L\otimes \bZ_\ell, Q)$ for $\ell\neq p$.
\end{lemma}

\begin{proof}
Both lattices shall be maximal at $\ell$ and by \cite[Rem.~7.2.5]{HP}, $(L'\otimes \bQ_\ell, Q')\cong (L\otimes \bQ_\ell, Q)$. Then we conclude by the fact that there is a unique isometry class of $\bZ_\ell$-maximal sublattices of a given $\bQ_\ell$-quadratic space (see for instance, \cite[Thm.~A.1.2]{HP}).
\end{proof}

\begin{remark}
Actually, for the case of Hilbert modular surfaces, the essential part of the above lemma is \cite[Prop.~3.1.3]{HY}. For the $\cA_2$ case, we can explicitly compute $L''$ as follows and it is maximal. By \cite[Prop.~5.2]{Eke}, for any $\ell\neq p$, there is a unique class (up to $\GL_4(\bZ_\ell)$-conjugation) of principal polarizations on the Tate module $T_\ell(A)$. Therefore, to compute $L''\otimes \bZ_\ell$, we may assume that $A=E^2$ and endowed with the product principal polarization, where $E$ is a supersingular elliptic curve. Hence the quadratic form on the lattice $L''$, which is the trace $0$ part of $H^2(A)$, is given by $x_0^2+\Nm$, where $\Nm$ is the quadratic form given by the reduced norm on the quaternion algebra $\End(E)$. %SAVE: here we use the fact that the endomorphism ring of a supersingular elliptic curve is a MAXIMAL order in the quaternion algebra.
\end{remark}

%%%%%%%%%%%%%%%%%%%%%%%%%%%%%%%%%%%%%%%%%%%%%%%%%%%%%%%%%%%%%%%%%%%%%%%%%%%%%%%%%%%%%%%%%%%%%%%
\section{The $F$-crystals $\bL_{\cris}$ on local deformation spaces of supersingular points}\label{sec_lattices}
In this section, we compute the lattices ($L''\otimes \bZ_p$ in \Cref{def_L''}) of special endomorphisms of supersingular points with the natural quadratic forms following Howard--Pappas \cite[\S\S 5-6]{HP}.\footnote{One may also carry out the computation following Ogus \cite[\S 3]{Ogus79}.} In conjunction with \cite[\S 1]{Kisin}, we then obtain $\bL_{\cris}$ (see \Cref{def_Lcris_S} and \Cref{def_Lcris_H}) on the formal neighborhoods of supersingular points in the Shimura variety $\cM$. 
As a direct consequence, we obtain the local equation of the non-ordinary locus in \S\ref{sec_eqn-non-ord}. These are the key inputs to \S\S\ref{sec_decay_Hil}-\ref{sec_decay_Sie}; in particular, we use the explicit descriptions of this section to prove our decay results.

\subsection{A brief review of the work of Howard--Pappas and Kisin}\label{sec_HP+Kisin}
Since both \cite{HP} and \cite{Kisin} apply to GSpin Shimura varieties of any dimension, we first recall their results in the general setting. 

Let $(V,Q)$ denote a quadratic $\bQ$-vector space of signature $(n,2)$ and let $L\subset V$ be a maximal even lattice which is self-dual at $p$. Let $\cM$ denote the smooth canonical integral model over $\bZ_p$ of the GSpin Shimura variety attached to $(L,Q)$ in \cite{Kisin}. 

Set $k=\overline{\bF}_p, W=W(k), K=W[1/p]$. 
In this section, we consider a fixed supersingular point $P\in \cM(k)$. In the case of abelian surfaces considered in \S\ref{sec_lat+sp-end} (with $L=L_S$ or $L_H$), $P$ supersingular means the corresponding abelian surface over $P$ is supersingular. This in turn is equivalent to the action of the crystalline Frobenius $\varphi$ on $\bL_{\cris,P}(W)$  being pure, with slope $0$. In the general setting, let $\bD$ denote the Dieudonn\'e crystal of the universal Kuga--Satake abelian variety over $\cM_{\bF_p}$ and let $\bL_{\cris}\subset \End(\bD)$ denote the sub crystal corresponding to $L\subset C(L)$ defined in \cite[\S 4.14]{MP16}.\footnote{Note that in the cases $L=L_H, L_S$ in \S 2, we still take $\bD$ to be the Dieudonn\'e crystal of the universal abelian surfaces, not that of the Kuga--Satake abelian varieties.} Let $\varphi$ denote the crystalline Frobenius on $\bD_P(W)$ and $\bL_{\cris,P}(W)$. Then we say $P$ is \emph{supersingular} if $\varphi$ acts on $\bL_{\cris,P}(W)$ with pure slope $0$ (see for instance \cite[Lem.~4.2.4, \S 7.2.1]{HP}).

By Dieudonn\'e theory, we have $L''\otimes \bZ_p=\bL_{\cris,P}(W)^{\varphi=1}$. In order to compute $L''\otimes \bZ_p$ and the $\varphi$-action on $\bL_{\cris,P}(W)$, we introduce another free $W$-module $\bL^\#_P(W)$ following \cite[\S 6.2.1]{HP}.\footnote{Note that in \cite{HP}, they use $y$ to denote a point in $\cM(k)$ and $\bL^\#_P(W)$ is denoted by $L_y$ while $\bL_{\cris, P}(W)$ is denoted by $L_y^\#$.}
\begin{defn}
The filtration on $\bD_P(W)$ is given by $\Fil^1 \bD_P(W):=\varphi^{-1}(p\bD_P(W))$. We define $\bL^\#_P(W):=\{v\in \bL_{\cris,P}(W)\otimes_W K \mid v\Fil^1 \bD_P(W)\subset \Fil^1 \bD_P(W)\}$.
\end{defn}

\begin{para}\label{summary_HP}
By \cite[Thm.~7.2.4]{HP}, studying supersingular points and their formal neighborhood in $\cM$ reduces to study the points and their formal neighborhood in the associated Rapoport--Zink spaces and hence we use results in \cite[\S\S 5-6]{HP}.

By \cite[Prop.~6.2.2]{HP}, $\varphi(\bL^\#_P(W))=\bL_{\cris,P}(W)$. In particular, \[L''\otimes \bZ_p=\bL_{\cris,P}(W)^{\varphi=1}=\bL^\#_P(W)^{\varphi=1}.\] 
Recall that in \Cref{def_L''}, we endow $V':=L''\otimes \bQ_p$ with a quadratic form $Q'$; let $[-,-]'$ denote the bilinear form on $V'$ given by $[x,y]'=Q'(x+y)-Q'(x)-Q'(y)$. Hence \[V'=(\bL_{\cris,P}(W)\otimes_W K)^{\varphi=1}.\] Since $P$ is supersingular, we have $n=\rk_W \bL_{\cris,P}(W)=\rk_{\bZ_p} L''=\dim_{\bQ_p}V'$.

Let $\Lambda_P\subset V'$ denote the dual of $L''\otimes \bZ_p$ with respect to $[-,-]'$. Then by \cite[Prop.~6.2.2]{HP}, $\Lambda_P$ is a \emph{vertex lattice}, i.e., $\Lambda_P$ is a $\bZ_p$-lattice in $V'$ such that $p\Lambda_P\subset \Lambda_P^\vee \subset \Lambda_P$. The \emph{type} $t_P$ of $\Lambda_P$ is defined to be  $\dim_{\bF_p}(\Lambda_P/\Lambda_P^\vee)$.
By \cite[Prop.~5.1.2, (1.2.3.1)]{HP}, there is $t_{\max}\in 2 \bZ$ which only depends on $n$ and $\det (V')=\det (V_{\bQ_p})$\footnote{The determinant $\det(V')$ is the determinant of the Gram matrix $([x_i,x_j]')_{i,j=1,\dots, n+2}$, where $\{x_i\}_{i=1}^{n+2}$ is a $\bQ_p$-basis of $V'$; we view $\det(V')$ as an element in $\bQ_p^\times/(\bQ_p^\times)^2$.} such that $t_P\in 2\bZ$ and $2\leq t_P \leq t_{\max}$. Moreover, there exists a vertex lattice $\Lambda\subset V'$ of type $t_{\max}$ such that $\Lambda_P\subset \Lambda$. Indeed, the proof of \cite[Prop.~5.1.2]{HP} constructs all possible isometry classes of $\Lambda$ (with the quadratic form) for all $(V,Q)$ (note that in \emph{loc.~cit.}, they proved that for given $(V,Q)$, the isometry class of $\Lambda$ is unique). 

Therefore, given $(V,Q)$, we first obtain the isometry class of $\Lambda$ of type $t_{\max}$ and then all isometry classes of the lattices of special endomorphisms $L''\otimes \bZ_p$ attached to all supersingular points are given by the duals of the vertex lattices contained in $\Lambda$.
%SAVE: from the proof of [HP, Prop 5.1.2], for each $(V,Q)$, there is exactly one isometry class of maximal vertex lattice $\Lambda$ because they showed that maximal vertex lattice is a maximal lattice and by Eichler's theorem, see A.1.2 of HP, maximal lattices are in one isometry class.

From $\Lambda$, we may compute all possible isomorphism classes of $\bL_{\cris,P}(W)$ and $\bL_P^\#(W)$ as rank $n$ free $W$-modules endowed with a quadratic form/bilinear form and a $\sigma$-linear Frobenius $\varphi$ (here we use $\sigma$ to denote the Frobenius action on $W$) following \cite[Prop.~6.2.2, \S 5.3.1]{HP}. Indeed, $\bL^\#_P(W)\subset \Lambda\otimes_{\bZ_p}W=:\Lambda_W$ is the preimage of a Lagrangian $\overline{L}^\#_P\subset \Lambda_W/\Lambda_W^\vee$ with respect to the quadratic form $pQ'\bmod p$ such that 
\begin{equation}\label{dim_lag}
    \dim (\overline{L}^\#_P+\overline{\varphi}(\overline{L}^\#_P))=t_{\max}/2+1,
\end{equation} where we use $\varphi$ to denote the $\sigma$-linear map on $\Lambda_W$ given by $\Id \otimes \sigma$ and $\overline{\varphi}(\overline{v}):=\overline{\varphi(v)}$ is well-defined for $\overline{v}\in \Lambda_W/\Lambda_W^\vee$ with a lift $v\in \Lambda_W$. The quadratic form and $\varphi$-action on $\bL^\#_P(W)$ are the restrictions of the quadratic forms and $\varphi$-action on $\Lambda_W$. We then obtain $\bL_{\cris,P}(W)=\varphi(\bL^\#_P(W))$. Note that by \cite[Prop.~5.1.2]{HP}, the even dimensional $\bF_p$-quadratic space $(\Lambda/\Lambda^\vee, pQ'\bmod p)$ does not have a Lagrangian defined over $\bF_p$ and hence is nonsplit; see \cite[\S\S 3.2-3.3]{HP14} for a discussion on how to find all such $\overline{L}^\#_P$.
\end{para}

\begin{defn}
For a supersingular point $P$, we say $P$ is \emph{superspecial} if $t_P=2$;\footnote{In the settings in \S\ref{sec_lat+sp-end}, $P$ is superspecial if and only if the corresponding abelian surface is isomorphic to the product of two supersingular elliptic curves, which is the usual definition for an abelian surface to be superspecial.} 
%SAVE: The Hilbert case is reduced to the Siegel case since for $x\in L''_S$ with $p\nmid Q(x)>0$, then $x^\perp \otimes \bZ_p + \bZ_p x = L''_S\otimes \bZ_p$ and this decomposition is compatible after taking dual. Hence a point $P$ has the same type independent of viewing it in $\cM_H$ or $\cM_S$. Such points are studied in details in \cite{KR3}. The discussion in \cite{LO} classifies all PPAS and we can directly compute the type and check. In particular, in the principally polarized case, the lattice $L''\otimes \bZ_p$, or more precisely the polarized $p$-divisible group only depends on being superspecial or not (see \cite[\S 4.1, \S 6.1]{LO}. 
we say $P$ is \emph{supergeneric} if $t_P=t_{\max}\neq 2$.
\end{defn}
By \cite[Prop.~5.2.2]{HP}, $P$ is superspecial if and only if \begin{equation}\label{cri-ssp}
    \varphi^2(\bL_P^\#(W))\subset \bL_P^\#(W) + \varphi (\bL_P^\#(W)).
\end{equation}

By \cite[(1.2.3.1)]{HP}, in the setting of \S\ref{sec_lat+sp-end}, we have $t_{\max}\leq 4$ and hence the supersingular points in question are either superspecial or supergeneric.

%SAVE:  When the supersingular locus is $1$-dimensional (for our applications, the supersingular locus is of dimension $0$ or $1$), then the superspecial points corresponding to those Lagrangians contained in a smaller vertex lattice. Indeed, by \cite{HP}*{Thm.~D(ii),(iii), Prop.~5.3.2, Thm.~6.3.1}, the dimension of the Rapoport--Zink space $RZ_{\Lambda}$ is $t/2-1$; for supergeneric points, $RZ_{\Lambda}$ is (some irreducible components of) the entire Rapoport--Zink space $RZ$ of the basic locus and hence it is $1$-dimensional. For superspecial points, they lie on the intersection distinct irreducible components of $RZ$ and the corresponding $RZ_{\Lambda}$ is $0$-dimensional and hence $t=2$ and this vertex lattice is contained in some maximal vertex lattice.

\begin{remark}\label{rmk-QQ'}
By \cite[Prop.~4.7 (iii) (iv)]{MP16}, $\GSpin(L,Q)_W$ acts on $\bD_P(W)$ and $\bL_{\cris,P}(W)$; moreover, as $W$-quadratic spaces, $\bL_{\cris,P}(W)\cong L\otimes W$ (we use $Q_W$ to denote the quadratic form on $L''\otimes \bZ_p$) and for $x\in \bL_{\cris,P}(W), x\circ x = Q_W(x)\cdot \Id \in \End(\bD_P(W))$. Therefore $Q'$ on $L''\otimes \bZ_p$ is the restriction of $Q$ on $\bL_{\cris,P}(W)$ to $L''\otimes \bZ_p$. We introduce the notation $Q'$ to emphasize that $Q'$ and $Q$ (as $\bZ_p$-quadratic forms) are restrictions of $Q_W$ to $\bZ_p$-lattices in different $\bQ_p$-subspaces. Hence $\GSpin(L,Q)_W=\GSpin (\bL_{\cris,P}(W),Q')$.
\end{remark}

\begin{para}\label{summary_Kisin}
We now describe the $F$-crystal $\bL_{\cris}$ over the formal completion $\widehat{\cM}_{P}$ along the supersingular point $P$ following \cite[\S\S 1.4-1.5]{Kisin} and \cite[\S 4.5]{Moonen98}; see also \cite[\S\S 3.1.4, 3.1.6]{HP}. 

The Hodge filtration $\Fil^1 \bD_P(W) \bmod p \subset \bD_P(k)$ corresponds to a cocharacter $\overline{\mu}: \bG_{m,k}\rightarrow \GSpin (L,Q)_k$ and we pick a cocharacter $\mu: \bG_{m,W}\rightarrow \GSpin (L,Q)_W$ which lifts $\overline{\mu}$. Let $U_P\subset \GSpin(L,Q)_W$ denote the opposite unipotent of the parabolic subgroup defined by $\mu$; and let $\widehat{U_P}$ denote the formal completion of $U_P$ along the identity. Pick coordinates and write $\widehat{U_P}=\Spf W[[x_1,\dots, x_d]]$ such that $x_1=\cdots =x_d=0$ defines the identity element in $U_P$. Let $\sigma$ denote the Frobenius action on $W[[x_1,\dots, x_d]]$ which lifts the $\sigma$-action on $W$ and for which $\sigma(x_i)=x_i^p$.

Let $R$ denote $\widehat{\cO}_{\cM,P}$, the complete local ring of $\cM$ at $P$. Then there exists an isomorphism from $\Spf R$ to $\widehat{U_P}$ (and we still use $\sigma$ to denote the Frobenius action on $R$ via the identification to $W[[x_1,\dots, x_d]]$) such that 
\begin{enumerate}
    \item $\bD(R)=\bD_P(W)\otimes_W R$ and $\bL_{\cris}(R)=\bL_{\cris,P}(W)\otimes_W R$ as $R$-modules;
    \item and under the above identifications, the $\sigma$-linear Frobenius action, denoted by $\Frob$, on $\bD(R)$ and $\bL_{\cris}(R)$ is given by $u\cdot (\varphi \otimes \sigma)$, where $u$ denotes the universal $W[[x_1,\dots, x_d]]$-point in $\widehat{U_P}$ and $\varphi$ is the crystalline Frobenius on $\bD_P(W)$ or $\bL_{\cris,P}(W)$ given in \S\ref{summary_HP}.
\end{enumerate}
%SAVE: The choice of the lift of filtration amounts to choose a lift of $P$ to be the center of $R$ as in Kisin and HP. We do not need to work with $R$ directly and hence we do not need to talk about the lift of $P$. As in Moonen, we just need the coordinates on $U_P$ such that the identity is the zero. 
On $\bL_{\cris}$, the $\GSpin(L,Q)_W$ action factors through the quotient $\SO(L,Q)_W$. So from now on, since we will only care about $\Frob$ on $\bL_{\cris}$, then by \Cref{rmk-QQ'}, we will work with $\mu: \bG_{m,W}\rightarrow \SO(\bL_{\cris,P}(W), Q')$ and $U_P$ the opposite unipotent of $\mu$ in $\SO(\bL_{\cris,P}(W), Q')$.
%SAVE: here is a subtle question on whether the cocharacter $\mu$ we choose lifts to a cocharacter of GSpin. In our particular case, it is fine; see for instance \cite[\S 4.2.1]{HP}.
\end{para}

In the rest of this section, we will apply \S\S \ref{summary_HP}, \ref{summary_Kisin} to the setting in \S\ref{sec_lat+sp-end} and we will work with the coordinates on $\widehat{U_P}$. When $L=L_H$, we write $\widehat{U_P}=\Spf W[[x,y]]$ and when $L=L_S$, we write $\widehat{U_P}=\Spf W[[x,y,z]]$. We will use $\epsilon\in \bZ_p^\times$ to denote an element which is not a perfect square in $\bZ_p$. Let $\bZ_{p^2}$ (resp. $\bQ_{p^2}$) denote $W(\bF_{p^2})$ (resp. $W(\bF_{p^2})[1/p]$) and let $\lambda \in \bZ_{p^2}^\times$ be an element such that $\sigma(\lambda)=-\lambda$ (for instance, we can take $\lambda$ to be a root in $\bZ_{p^2}$ of $x^2-\epsilon=0$). We will use $\{v_i\}_{i=1}^{n+2}$ to denote a $W$-basis of $\bL_{\cris,P}(W)$ and $\{w_i\}_{i=1}^{n+2}$ to denote a $\bZ_p$-basis of $\Lambda_P^\vee=\bL_{\cris,P}(W)^{\varphi=1}$; note that $\Span_W\{w_i\}$ is a $W$-sublattice of $\bL_{\cris,P}(W)$.

\subsection{The Hilbert case $L=L_H$}\label{sec_F_Hil}
Recall that as in \Cref{thm_main}(\ref{item_H}), we have $p\nmid m\in \bZ_{>0}$.
 \begin{para}\label{Frob-H-inert}
Assume that $p$ is inert in $\bQ(\sqrt{m})$;\footnote{If $m\in \bZ$ is a perfect square, then by convention, we view every prime $p$ to be split in $\bQ[x]/(x^2-m)$ and the discussion of the split case still holds.} then we have $t_{\max}=4$.

The vertex lattice with type $t_{\max}$ is $\Lambda=\Span_{\bZ_p}\{e_1,f_1\}\oplus Z$, where
 \[
 [Z, e_1]'=[Z,f_1]'=[e_1,e_1]'=[f_1,f_1]'=0,\, [e_1,f_1]'=1/p,\, Z\cong \bZ_{p^2},\, Q'(x)=x\sigma(x)/p,\, \forall x\in Z.
 \]
 %SAVE: in general, a quadratic form on $Z$ is given by $cx\sigma(x)$ for $c\in \bQ_p^\times/\Nm \bQ_{p^2}^\times$. Since $\Nm: \bF_{p^2}\rightarrow \bF_p$ is surjective and hence every element in $\bZ_p^\times$ is a norm. Therefore, there are two classes of $c$, namely $1$ and $p$.
 Hence $\Lambda^\vee=p\Lambda$. Set $e_2=(1\otimes 1+ (1/\lambda) \otimes \lambda)/2, f_2=(1\otimes 1+ (-1/\lambda) \otimes \lambda)/2 \in \bZ_{p^2}\otimes_{\bZ_p} Z$. Then, as elements in $\Lambda_W$,
 \[\varphi(e_1)=e_1, \varphi(f_1)=f_1,
 \varphi(e_2)=f_2, \varphi(f_2)=e_2, [e_2,e_2]'=[f_2,f_2]'=0, [e_2,f_2]'=1/p.
 \]

 All possible $\overline{L}^\#_P$ are given by two families of Lagrangians in $k$-quadratic space spanned by $\bar{e}_1,\bar{e}_2,\bar{f}_1,\bar{f}_2\in \Lambda_W/\Lambda_W^\vee$ with quadratic form $pQ$ satisfying \eqref{dim_lag}: \[\overline{L}^\#_P=\Span_k\{\bar{e}_1+\sigma^{-1}(\bar{c})\bar{f}_2,\bar{e}_2-\sigma^{-1}(\bar{c})\bar{f}_1\}, \text{ or } \overline{L}^\#_P=\Span_k\{\bar{e}_1+\sigma^{-1}(\bar{c})\bar{e}_2, \sigma^{-1}(\bar{c})\bar{f}_1-\bar{f}_2\},\] where $\bar{c}\in k$.\footnote{Indeed, as $\dim_k \Lambda_W/\Lambda_w^\vee$ is small, in this case, all Lagrangians satisfy \eqref{dim_lag}. There are two families and each is parametrized by $\bP^1(k)$ so more accurately, we shall view $\bar{c}\in \bP^1(k)$, i.e., there are two more Lagrangians $\Span_k\{\bar{f}_1,\bar{f}_2\}$ and $\Span_k\{\bar{e}_2, \bar{f}_1\}$; however, since the role of $e_i$ and $f_i$ are symmetric, the computation for these two cases are equivalent to $\Span_k\{\bar{e}_1,\bar{e}_2\}$ and $\Span_k\{\bar{e}_1, \bar{f}_2\}$ so we may safely omit them and only take $\bar{c}\in k$. Moreover, we use $\sigma^{-1}(\bar{c})$ to be the parameter here because eventually we want to work with $\bL_{\cris,P}(W)=\varphi(\bL^\#_P(W))$.} Therefore, we have that
 \[\bL_{\cris,P}(W)=\Span_W\{e_1+ce_2, cf_1-f_2, pe_2, pf_1\}, \text{ or } \bL_{\cris,P}(W)=\Span_W\{e_1+cf_2,e_2-cf_1, pf_1, pf_2\},\]
 where $c\in W$.\footnote{Here we notice that $\varphi$ swaps two families of $\bL^\#_P(W)$; in particular, the general formula for $\bL_{\cris,P}(W)$ is the same as that for $\bL^\#_P(W)$ (other than swapping between the two families). This observation holds true in general by \cite[Rmk.~3.5]{HP14}.}
 Moreover, by \eqref{cri-ssp}, $P$ is superspecial if and only if $\sigma^{-1}(c)-\sigma(c)\in pW$, which is equivalent to the Teichmuller lift of $\bar{c}$ lying in $\bZ_{p^2}$. Note that if $c-c'\in pW$, then $c,c'$ define the same $\bL_{\cris,P}(W)$. Therefore, without loss of generality, from now on, we will only work with $c\in W$ which is the Teichmuller lifting of $\bar{c}\in k$. Hence $P$ is superspecial if and only if there exists $c\in \bZ_{p^2}$ such that $\bL_{\cris,P}(W)$ is given by the above form.

 In order to compute the $F$-crystal $\bL_{\cris}$, we pick the following $W$-basis $\{v_1,\dots, v_4\}$ of $\bL_{\cris,P}(W)$ such that the Gram matrix of $[-,-]'$ with respect to this basis is $\begin{bmatrix}0 & I\\ I &0\end{bmatrix}$, where $I$ denotes the $2\times 2$ identity matrix.
 For the first family, take
\[
v_1=f_2-cf_1, v_2=e_1+ce_2-\sigma^{-1}(c)cf_1+\sigma^{-1}(c)f_2, v_3=pe_2-p\sigma^{-1}(c)f_1, v_4=pf_1;
\]
for the second family, take
 \[
 v_1=e_2-cf_1, v_2=\sigma^{-1}(c)e_2-\sigma^{-1}(c)cf_1+e_1+cf_2, v_3=pf_2-p\sigma^{-1}(c)f_1, v_4=pf_1.
 \]
Then on $\bL_{\cris,P}(W)$, with respect to $\{v_1,\dots,v_4\}$, we have
 \[
\varphi=b\sigma, \text{ with }b= \begin{bmatrix}
0 & \sigma(c)-\sigma^{-1}(c) &  p & 0\\
0 & 1                                     & 0     & 0\\
1/p & 0                                  & 0     & 0\\
(\sigma^{-1}(c)-\sigma(c))/p & 0 & 0 & 1
\end{bmatrix}.
 \]
The filtration on $\bL_{\cris,P}(k)$ is given by \[\Fil^1\bL_{\cris,P}(k)=\Span_k\{\bar{v}_3\},\, \Fil^0\bL_{\cris,P}(k)=\Span_k\{\bar{v}_2,\bar{v}_3, \bar{v}_4\},\, \Fil^{-1}\bL_{\cris,P}(k)=\bL_{\cris,P}(k),\] 
so we may choose $\mu: \bG_{m,W}\rightarrow \SO(\bL_{\cris,P}(W),Q')$ to be $t\mapsto \diag(t^{-1}, 1, t, 1)$. Then $\widehat{U_P}=\Spf W[[x,y]]$ with the universal point 
 \[
u= \begin{bmatrix}
1 & x &  -xy & y\\
0 & 1  & -y    & 0\\
0 & 0   & 1     & 0\\
0 & 0 & -x & 1
\end{bmatrix} \text{ and } \Frob=ub\sigma, \text{ with }ub= \begin{bmatrix}
-xy/p-ay/p & a+x & p & y\\
-y/p & 1  & 0    & 0\\
1/p & 0   & 0     & 0\\
-x/p-a/p & 0 & 0 & 1
\end{bmatrix},
 \]
where $a=\sigma(c)-\sigma^{-1}(c)$; we have $a=0$ if $P$ is superspecial and $a\in W^\times$ if $P$ is supergeneric.

When $P$ is superspecial,  $\{w_1=pv_1+v_3,w_2=\lambda(pv_1-v_3),w_3=v_2,w_4=v_4 \}$ is a $\bZ_p$-basis of $L''\otimes \bZ_p$. Using $\{w_1,\dots, w_4\}$ as a $K$-basis of $\bL_{\cris,P}(W)[1/p]$, we have 

\begin{equation}\label{FrobHinert}
    \Frob=\left(I+\begin{bmatrix}
-\frac{xy}{2p} & \frac{\lambda xy}{2p} & \frac{x}{2p} & \frac{y}{2p}\\
-\frac{xy}{2\lambda p} & \frac{xy}{2p} & \frac{x}{2\lambda p}& \frac{y}{2\lambda p} \\
-y&\lambda y &0 &0 \\
-x&\lambda x & 0 & 0\end{bmatrix}\right) \circ \sigma.
\end{equation}

When $P$ is supergeneric, 
$\{w_1=v_4, w_2=pv_1+v_3+(c+\sigma^{-1}(c))v_4, w_3=\lambda(pv_1-v_3+(c-\sigma^{-1}(c))v_4), w_4=pv_2-cv_3-p\sigma^{-1}(c)v_1-c\sigma^{-1}(c)v_4\}$ is a $\bZ_p$-basis of $L''\otimes \bZ_p$ and with respect to this basis,
 $\Frob=(I+\frac{y}{p}A+xB)\circ \sigma$, where
\[
A=\begin{bmatrix}
-c & -c^2 & -\lambda c^2 & 0\\
1/2 & 0 & \lambda c & c^2/2\\
1/(2\lambda) & c/\lambda & 0 & -c^2/(2\lambda) \\
0 & -1 & \lambda & c
\end{bmatrix},
B=\begin{bmatrix}
0 & -1+cy/p & \lambda cy/p +\lambda & -c^2y/p\\
0 & -y/(2p) & \lambda y/(2p) & 1/2+cy/(2p)\\
0 & -y/(2\lambda p) & y/(2p) & 1/(2\lambda) +cy/(2p\lambda)\\
0 & 0 & 0 & 0
\end{bmatrix}.
\]
\end{para}

\begin{para}\label{parsplit}
Assume that $p$ is split in $\bQ(\sqrt{m})$; then we have
$t_{\max}=2$ and hence every $P$ is superspecial.

The vertex lattice with type $t_{\max}$ is $\Lambda=\{(x_1,x_2,x_3,x_4)\in \bZ_p^4\}$ with \[Q'((x_1,x_2,x_3,x_4))=x_1^2-\epsilon x_2^2-p^{-1} x_3^2 + \epsilon p^{-1} x_4^2;\] we have $\Lambda^\vee=\Span_{\bZ_p}\{e_1,e_2, pe_3, pe_4\}$, where $e_i$ is the vector with $x_i=1$ and $x_j=0$ for $j\neq i$. Recall that we take $\epsilon=\lambda^2$; we then have\footnote{There are exactly two Lagrangians and the other one is given by replacing $\lambda$ by $-\lambda$. Since $\lambda$ and $-\lambda$ play the same role in our later computation, there is no loss of generality here.} that \[\bL_{\cris,P}(W)=\Span_W\{v_1=\frac{1}{2}(e_3+\lambda^{-1}e_4), v_2=\frac{1}{2}(e_1+\lambda^{-1}e_2), v_3=-\frac{1}{2}(pe_3-\lambda^{-1}pe_4), v_4=\frac{1}{2}(e_1-\lambda^{-1}e_2)\}.\]
The Gram matrix is $\begin{bmatrix}0 & I\\ I &0\end{bmatrix}$ and on $\bL_{\cris,P}(W)$, the Frobenius $\varphi=b\sigma$ with 
\[b=\begin{bmatrix}0 & 0 & -p & 0\\
0 &0 &0 &1 \\
-1/p & 0 & 0 &0\\
0 & 1 &0 &0\end{bmatrix}.\] The filtration on $\bL_{\cris,P}(k)$ given by $\varphi$ is the same as in \S\ref{Frob-H-inert} and hence we may use the same $\mu$ and $u$ there. Therefore, on $\bL_{\cris}(W[[x,y]])$, we have
\[\Frob=ub\sigma, \text{ with } ub=\begin{bmatrix}xy/p & y & -p & x\\
y/p &0 &0 &1 \\
-1/p & 0 & 0 &0\\
x/p & 1 &0 &0\end{bmatrix}.\]
Moreover, $\{w_1=pv_1-v_3, w_2=\lambda(pv_1+v_3), w_3=v_2+v_4, w_4=\lambda(v_4-v_2)\}$ is a $\bZ_p$-basis of $L''\otimes \bZ_p$ and with respect to this basis,
\begin{equation}\label{F-Hil-sp}
    \Frob=\left(I+\begin{bmatrix}
\frac{xy}{2p}&-\frac{\lambda xy}{2p}&\frac{x+y}{2p}&\frac{-\lambda(x -y)}{2p}\\
\frac{xy}{2\lambda p}&- \frac{xy}{2p}&\frac{x+y}{2\lambda p}&\frac{-(x-y)}{2p}\\
\frac{x+y}{2}&\frac{-\lambda (x + y)}{2}&0&0\\
\frac{x-y}{2\lambda}&\frac{-(x - y)}{2}&0&0\\
\end{bmatrix}\right) \circ \sigma.
\end{equation}
\end{para}

%SAVE: Old version: the difference is replacing $y$ with $-y$. There is a basis $\{v_1,v_2,v_3,v_4\}$ of $V$ such that the Gram matrix of the bilinear form is 
\begin{comment}
\[\begin{bmatrix}0 & -1 & 0 & 0\\
-1 &0 &0 &0 \\
0& 0 & 0 &1\\
0 & 0 &1 &0\end{bmatrix}\]
and a $\varphi$-invariant basis of $V$ is $\{w_1 = pv_1 + v_2, w_2 = \lambda pv_1 - \lambda v_2, w_3 = v_3 + v_4, w_4 = \lambda v_3 - \lambda v_4  \}$. The bilinear form restricted to the $\bQ_p$-span of the $w_i$ has no isotropic vectors.
Then by the same argument, the Frobenius on the $F$-crystal associated to $V$ in the above basis has the form
\[\begin{bmatrix}
-xy/p&p&x&-y\\
1/p&0&0&0\\
x/p&0&0&1\\
-y/p&0&1&0\\
\end{bmatrix}.\]
The Frobenius on the $F$-crystal with respect to the the $w_i$ is:  
\[\begin{bmatrix}
1-\frac{xy}{2p}&\frac{\lambda xy}{2p}&\frac{x-y}{2p}&\frac{-\lambda(x +y)}{2p}\\
\frac{-xy}{2\lambda p}&1 + \frac{xy}{2p}&\frac{x-y}{2\lambda p}&\frac{-(x+y)}{2p}\\
\frac{x-y}{2}&\frac{-\lambda (x - y)}{2}&1&0\\
\frac{x+y}{2\lambda}&\frac{-x - y}{2}&0&1\\
\end{bmatrix}.\]
\end{comment}

\subsection{The Siegel case $L=L_S$}\label{sec_F_Sie}
We now compute $\bL_{\cris}$ for \Cref{thm_main}(\ref{item_S}) and \Cref{thm_max}. In this case, we have
$t_{\max}=4$. 

The vertex lattice with type $t_{\max}$ is  $\Lambda = \Span_{\bZ_p}\{e_1,f_1 \} \oplus Z_S$, where $Z_S=\{(x_1,x_2,x_3)\in \bZ_p^3\}$
\[ [Z_S, e_1]'=[Z_S,f_1]'=[e_1,e_1]'=[f_1,f_1]'=0,\, [e_1,f_1]'=1/p, Q'((x_1,x_2,x_3))=c(-\epsilon x_1^2-p^{-1} x_2^2 +\epsilon p^{-1} x_3^2),\]
for some $c\in \bZ_p^\times$. Since $\det \Lambda =\det L \in \bQ_p^\times/(\bQ_p^\times)^2$ and $\det L=2$, we have $c=-1$. Let $g=(1,0,0)\in Z_S$ and $Z=\Span_{\bZ_p}\{(0,1,0), (0,0,1)\}\subset Z_S$. Then $\Lambda/\Lambda^\vee=\Span_{\bF_p}\{\overline{e}_1,\overline{f}_1\}\oplus Z/Z^\vee$.
Note that $\Span_{\bZ_p}\{e_1,f_1\}\oplus Z$ is exactly the same quadratic $\bZ_p$-lattice which is denoted by $\Lambda$ in \S\ref{Frob-H-inert}; hence the same computation there applies to find $\bL_{\cris,P}(W)\subset \Lambda \otimes W$. More precisely, there exist $v_1,\dots, v_4\in \Span_W\{e_1,f_1\}\oplus Z\otimes W$ and $c\in W$ which is the Teichmuller lift of $\overline{c}\in k$ such that 
\begin{enumerate}
    \item $\bL_{\cris,P}(W)=\Span_W\{v_1,\dots,v_4,v_5\}$, where $v_5=g$;
    \item the Gram matrix of $[-,-]'$ with respect to $\{v_1,\dots,v_5\}$ is $\begin{bmatrix}0& I & 0\\
    I & 0 & 0\\
    0 & 0 & 2\epsilon\end{bmatrix}$, where $I$ is the $2\times 2$ identity matrix;
    \item The Frobenius $\varphi$ on $\bL_{\cris,P}(W)$ with respect to the basis $\{v_i\}$ is 
     \[
\varphi=b\sigma, \text{ with }b= \begin{bmatrix}
0 & \sigma(c)-\sigma^{-1}(c) &  p & 0 & 0\\
0 & 1    & 0     & 0 &0\\
1/p & 0                   & 0     & 0 &0\\
(\sigma^{-1}(c)-\sigma(c))/p & 0 & 0 & 1 & 0\\
0 & 0 & 0 &0 &1
\end{bmatrix};
 \]
    \item $P$ is superspecial if and only if $\sigma^2(c)=c$.
\end{enumerate}
We may choose $\mu: \bG_{m,W}\rightarrow \SO(\bL_{\cris,P}(W),Q')$ to be $t\mapsto \diag(t^{-1}, 1, t, 1, 1)$. Then $\widehat{U_P}=\Spf W[[x,y,z]]$ with the universal point 

 \[
u= \begin{bmatrix}
1 & x &  -xy-\frac{z^2}{4\epsilon} & y & z\\
0 & 1  & -y    & 0 &0\\
0 & 0   & 1     & 0 &0\\
0 & 0 & -x & 1 &0\\
0 & 0 & -\frac{z}{2\epsilon} & 0 & 1
\end{bmatrix} \text{ and } \Frob=ub\sigma, \text{ with }ub= \begin{bmatrix}
-\frac{1}{p}(xy+\frac{z^2}{4\epsilon})-\frac{ay}{p} & a+x & p & y & z\\
-\frac{y}{p} & 1  & 0    & 0 &0\\
\frac{1}{p} & 0   & 0     & 0 &0\\
-\frac{x+a}{p} & 0 & 0 & 1 &0\\
-\frac{z}{2\epsilon p} & 0 & 0 & 0 & 1
\end{bmatrix}
 \]
 
 acting on $\bL_{\cris}(W[[x,y,z]])$, 
where $a=\sigma(c)-\sigma^{-1}(c)$; note that $a=0$ if and only if $P$ is superspecial. 

For the proofs of \Cref{thm_main}(\ref{item_S}) and \Cref{thm_max}, we only need to study superspecial points so we only give the matrix of $\Frob$ with respect to a basis of $\bL_{\cris} \otimes_W K$ consisting of elements in $L''\otimes \bZ_p$ when $P$ is superspecial; we refer the reader to the appendix for the discussion when $P$ is supergeneric. 

We now assume that $P$ is superspecial. Let $w_1=\lambda(pv_1-v_3), w_2=pv_1+v_3, w_3=v_2, w_4=v_4, w_5=v_5$. Then $L''\otimes \bZ_p=\Span_{\bZ_p}\{w_1,\dots, w_5\}$. We view $\{w_i\}_{i=1}^5$ as a $K$-basis of $\bL_{\cris,P}(W)\otimes K$, then the Frobenius on $\bL_{\cris}(W[[x,y,z]])$ is given by 
\begin{equation}\label{F-Sie}
    \Frob=\left(I+ \begin{bmatrix}
\frac{1}{2p}(xy+\frac{z^2}{4\epsilon}) & -\frac{1}{2\lambda p}(xy+\frac{z^2}{4\epsilon}) & \frac{x}{2\lambda p}& \frac{y}{2\lambda p} & \frac{z}{2\lambda p}\\
\frac{\lambda}{2p}(xy+\frac{z^2}{4\epsilon}) & -\frac{1}{2p}(xy+\frac{z^2}{4\epsilon}) & \frac{x}{2p} & \frac{y}{2p} & \frac{z}{2p} \\
\lambda y & -y & 0 & 0 & 0\\
\lambda x & -x& 0 & 0 & 0\\
\frac{\lambda z}{2\epsilon}  & -\frac{ z}{2\epsilon} & 0 & 0 & 0

\end{bmatrix}\right ) \circ \sigma.
\end{equation}

\subsection{Equation of non-ordinary locus}\label{sec_eqn-non-ord}

We now use the computation in \S\S\ref{sec_F_Hil}, \ref{sec_F_Sie} to obtain the local equation of the non-ordinary locus in a formal neighborhood of a supersingular point $P$ using results in \cite{Ogus01}. Although \cite{Ogus01} only focuses on the case of K3 surfaces, the results that we recall here apply to any GSpin Shimura varieties. We follow the notation in \S\ref{sec_HP+Kisin}. For a perfect field $k'$ of characteristic $p$, for $P'\in \cM(k')$, we say $P$ is \emph{ordinary} if the slopes of the crystalline Frobenius $\varphi$ on $\bL_{\cris,P'}(W)$ are $-1,1$ with multiplicity $1$ and $0$ with multiplicity $n$.\footnote{When $L=L_H,L_S$, the point $P'$ is ordinary if and only if the corresponding abelian surface over $k'$ is ordinary by the definition of $\bL_{\cris}$.}
%SAVE: different definitions of being ordinary for general GSpin Shimura varieties. I haven't verified that this notion of ordinary coincides with Kottwitz's notion in $X_*(T)_{\Gamma}$ ---- surely, this is expected to be true (in particular, I haven't checked that the map on dominating co-characters from $\GSpin(V)$ to $\GSp(C(V))$ is injective). On the other hand, it is known that this notion of ordinary is equivalent to that the Kuga--Satake abelian variety at $P'$ is ordinary. When $\dim V$ is odd (in particular, K3 surfaces), the proof of Prop 2.5 in Nygaard "Tate conjecture for ordinary K3 surfaces" applies. When $\dim V$ is even, to apply the same proof, we use $C(V)$ instead of $C^+(V)$.

The cocharacter $\overline{\mu}$ defines a filtration $\Fil^i, i=-1,0,1$ on $\bL_{\cris,P}(k)$, which is the Hodge filtration in \cite{Ogus01} and in particular, $\dim \Fil^1 \bL_{\cris,P}(k)=1, \dim \Fil^0 \bL_{\cris,P}(k)=n+1, \dim \Fil^{-1} \bL_{\cris,P}(k)=n+2$ and the annihilator of $\Fil^1 \bL_{\cris,P}(k)$ in $\bL_{\cris,P}(k)$ with respect to $Q$ is $\Fil^0 \bL_{\cris,P}(k)$.\footnote{See also \cite[p.411]{Ogus} for the definition. Note that here we directly work on the crystalline cohomology without using the canonical isomorphism to the de Rham cohomology. Note that our filtration is shifted by $1$ when comparing to the filtration in \cite{Ogus01} because his Frobenius is $p$ times our Frobenius.} The Hodge filtration over the mod $p$ complete local ring $R\otimes_W k$ at $P$ is given by $\Fil^i \bL_{\cris}(R\otimes_W k):=\Fil^i \bL_{\cris,P}(k)\otimes_k (R\otimes k)$. Note that $\Frob (\Fil^0 \bL_{\cris}(R\otimes_W k) )\subset \Fil^0 \bL_{\cris}(R\otimes_W k)$, so we have a well-defined map $p\Frob: \gr_{-1}\bL_{\cris}(R\otimes_W k)\rightarrow \gr_{-1}\bL_{\cris}(R\otimes_W k)$, where $\gr_{-1}\bL_{\cris}(R\otimes_W k):=\Fil^{-1} \bL_{\cris}(R\otimes_W k)/\Fil^0 \bL_{\cris}(R\otimes_W k)$.

\begin{lemma}[Ogus]
For a supersingular point $P$,
The non-ordinary locus (over $k$) in the formal neighborhood of $P$ is given by the equation \[p\Frob|_{\gr_{-1}\bL_{\cris}(R\otimes_W k)}=0.\]
\end{lemma}
\begin{proof}
By \cite[Prop.~11]{Ogus01}, the discussion of the conjugate filtration on \cite[p.333-334]{Ogus01}, and the fact that the annihilator of $\Fil^1 \bL_{\cris}(R\otimes k)$ in $\bL_{\cris}(R\otimes k)$ with respect to $Q$ is $\Fil^0 \bL_{\cris}(R\otimes k)$, we have that the equation defining the non-ordinary locus is the projection of the conjugate filtration (denoted by $F^2_{con}$ in \emph{loc.~cit.}) to $\gr_{-1}\bL_{\cris}(R\otimes k)$. By definition, $F^2_{con}=p\Frob \bL_{\cris}(R\otimes k)$ and then the lemma follows.
\end{proof}
%SAVE: Ogus01, Prop.3 is also relevant. 

\begin{corollary}\label{eqn-non-ord}
 When $L=L_H$, the local equation of the non-ordinary locus in a formal neighborhood of a supersingular point $P$ is $xy=0$ if $P$ is superspecial and is $y=0$ if $P$ is supergeneric; when $L=L_S$, the local equation is $xy+z^2/(4\epsilon)=0$ if $P$ is a superspecial point and $(x+a)y+z^2/(4\epsilon)=0$ if $P$ supergeneric, where $a\in W(k)^\times$ depends on $P$.
\end{corollary}
\begin{proof}
We will prove this corollary in the case $L = L_S$, since the other case is handled the same way. Recall we have the basis $v_1\hdots v_5$ of $\bL_{\cris} $, with $\Fil^-{1} = \bL_{\cris}$ and $\Fil^0$ being spanned by $v_2,v_3,v_4$ and $v_5$. Therefore, using the explicit formulas from the previous section, we see the map 
$p\Frob: \gr_{-1}\bL_{\cris}(R\otimes_W k) \rightarrow \gr_{-1}\bL_{\cris}(R\otimes_W k)$ is given by $p\Frob(v_1) = -(xy + \frac{z^2}{4\epsilon} + ay)v_1$. Our result now follows from Ogus's description of the non-ordinary locus. 
\end{proof}

%%%%%%%%%%%%%%%%%%%%%%%%%%%%%%%%%%%%%%%%%%%%%%%%%%%%%%%%%%%%%%%%%%%%%%%%%%%%%%%%%%%%%%%%%%%%%%%
\section{Arithmetic Borcherds Theory, Siegel mass formula, and Eisenstein series}\label{sec_global}

We use arithmetic Borcherds theory \cite{HMP} to control the global intersection number of a curve $C$ with special divisors. More precisely, we use the work of Bruinier and Kuss in \cite{BK03} to study the Fourier coefficients of the Eisenstein part of the (vector-valued) modular form arising from Borcherds theory. In order to compare the global intersection number with the local contribution later in the paper, we also apply the computations in \cite{BK03} and the Siegel mass formula to the Eisenstein part of the theta series attached to a supersingular point and reduce the question to a computation of local densities and determinants of the lattices $L$ and $L'$ introduced in \S\ref{sec_lat_glo} and \Cref{def_L''} (in \S\ref{sec_Siegel-mass}, we will summarize the properties of $L'$). We use Hanke's method in \cite{Han04} to compute the local densities. Throughout this section, $p$ is an odd prime such that $L$ is self-dual at $p$. For a prime $\ell$, we use $v_\ell: \bZ_\ell\backslash\{0\}\rightarrow \bZ_{\geq 0}$ to denote the $\ell$-adic valuation.

\subsection{Arithmetic Borcherds theory and the explicit formula for the Eisenstein series}
Recall the special divisors $Z(m)$ from \Cref{def_spdiv_S} and \S\ref{def_spdiv_H}. The following modularity result is the key input to the estimate of the intersection number $Z(m).C$. 

In order to state the result using vector-valued modular forms, for $\mu\in L^\vee/L, m\in \bQ_{>0}$, let $\cZ(m,\mu)$ denote the special divisors over $\bZ$ in $\cM$ defined in \cite[\S 4.5, Def.~4.5.6]{AGHMP}. By definition, $\cZ(m,0)$ is the divisor $\cZ(m)$ defined in \S\ref{sec_sp_end}; and roughly speaking, $\cZ(m,\mu)$ parametrizes abelian surfaces $A$ with a special quasi-endomorphism $s$ such that $Q(s)=m$ and the $\ell$-adic and crystalline realizations of $s$ lie in the image of $(\mu+L)\otimes \bZ_\ell$ and $(\mu+L)\otimes \bZ_p$ in $\End(T_\ell(A)\otimes \bQ_\ell)$ and $\End(\bD\otimes_W W[1/p])$ respectively, where $\bD$ is the Dieudonn\'e module of $A$ . By the proof of \cite[Prop.~4.5.8]{AGHMP} and \cite[Prop.~5.21]{MP16}, the assumption that $L$ is self-dual at $p$ implies that $\cZ(m,\mu)$ is flat over $\bZ_p$. Let $Z(m,\mu)$ denote $\cZ(m,\mu)_{\bF_p}$. Let $(\fe_\mu)_{\mu\in L^\vee/L}$ denote the standard basis of $\bC[L^\vee/L]$. Let $\omega \in \Pic(\cM_{\bF_p})_{\bQ}$ denote the Hodge line bundle in the $\bQ$-Picard group of $\cM_{\bF_p}$; in other words, $\omega$ is the line bundle of weight $1$ modular forms (see for instance \cite[Thm.~4.4.6]{AGHMP} for a definition of $\omega$). 
%SAVE: we work with $\Pic_{\bQ}$ because it's easier to define when we work with Deligne--Mumford stacks.

\begin{theorem}[Borcherds, Howard--Madapusi-Pera]\label{thm_HMP}
Assume $(L,Q)$ is a maximal quadratic lattice of signature $(n,2)$ such that $L$ is self-dual at $p$.
%SAVE: we put the self-dual condition here for flatness of $\cZ(m)$. In particular, essentially by \cite[Prop.~5.21]{MP16}, $\cZ(m,\mu)$ is flat over $\bZ_p$ for $p$ odd and $L$ self-dual at $p$. In general without the self-dual assumption, we then need $n\geq 3$ for $p$ odd by \cite[Prop.~4.5.8]{AGHMP}. For $p=2$, not self-dual, we need $n\geq 4$ by Prop.~7.2.2 in the revised version of HMP. On the other hand, the arithmetic/classical modularity result holds true for all $n$.
The generating series \[\omega^{-1}\fe_0+\sum_{m> 0, \mu\in L^\vee/L}Z(m,\mu) q^m \fe_\mu, \text{ where } q=e^{2\pi i \tau},\] lies in $M_{1+\frac{n}{2}}(\rho_L)\otimes \Pic(\cM_{\bF_p})_{\bQ}$.  Here, $\rho_L$ denotes the Weil representation on $\bC[L^\vee/L]$ and $M_{1+\frac{n}{2}}(\rho_L)$ denotes the space of  vector-valued modular forms of $\Mp_2(\bZ)$ with respect to $\rho_L$ of weight $1+\frac{n}{2}$.\footnote{In \cite{Bor99}, \cite{BK01}, \cite{BK03}, they work with $(L,-Q)$ and the modular form is with respect to the dual of the Weil representation of $(L,-Q)$, which is the Weil representation of $(L,Q)$. Our convention is the same as the one in \cite{HMP} and \cite{Br17}.} In particular, for any $\bQ$-linear functional $\alpha: \Pic(\cM_{\bF_p})_{\bQ}\rightarrow \bC$, the vector-valued power series \[\alpha(\omega^{-1})\fe_0+\sum_{m> 0, \mu\in L^\vee/L}\alpha(Z(m,\mu)) q^m \fe_\mu\] is the Fourier expansion of an element of $M_{1+\frac{n}{2}}(\rho_L)$.
\end{theorem}

\begin{proof}
By abuse of notation, we also use $\omega$ to denote the Hodge line bundle over $\cM$. By \cite[Thm.~B]{HMP}, the generating series $\omega^{-1}\fe_0+\sum_{m> 0, \mu\in L^\vee/L}\cZ(m,\mu) q^m \fe_\mu \in M_{1+\frac{n}{2}}(\rho_L)\otimes \Pic(\cM)_{\bQ}$. Since $\cZ(m,\mu)$ are flat over $\bZ_p$, then the desired assertion follows from intersecting with $\cM_{\bF_p}$. 
  
Alternatively, Borcherds \cite[Thm.~4.5]{Bor99} proved the modularity assertion for $\cZ(m,\mu)_\bC$. By \cite[Lemma 5.12]{Davesh} and the flatness of $\cZ(m,\mu)$, the proof of Borcherds implies the desired modularity for $Z(m,\mu)$.
\end{proof}

\begin{para}\label{modularityH}
In the setting of Theorem \ref{thm_main}(\ref{item_H}) (i.e. the case when $L=L_H$), we work with curves $C$ that are not necessarily proper. We therefore need a version of the above modularity result that holds for the special fiber a toroidal compactification of $\cM$. To that end, let $\cM^{\textrm{tor}}$ denote a toroidal compactification of $\cM$, and let $D_1,\dots, D_k$ denote irreducible components of the boundary $\cM^{\textrm{tor}}_{\bF_p}\setminus \cM_{\bF_p}$. In \cite[Theorem 6.2]{BBK}, the authors prove the modularity result for $\cM^{\textrm{tor}}$, which will directly imply the modularity result for $\cM^{\textrm{tor}}_{\bF_p}$. The constant term is still given by the Hodge line bundle, still denoted by $\omega$, on $\cM^{\rm{tor}}_{\bF_p}$ and the special divisors $Z(m,\mu)$ are replaced by\footnote{Our notation $Z' + E$ is different from the notation used in \emph{loc. cit.}} $Z'(m,\mu) + E(m,\mu)$, where $Z'(m,\mu)$ is the Zariski-closure of $Z(m,\mu)$ in $\cM^{\textrm{tor}}_{\bF_p}$, and $E(m,\mu)$ is a ``correction term'', and has as its irreducible components the $D_i$ with appropriate multiplicity. Crucially, when $Z(m,\mu)$ is proper (see \S\ref{def_setT} for when this happens), the multiplicities of the $D_i$ in correction term $E(m,\mu)$ are all zero and hence $E(m,\mu)$ is trivial. %SAVE: the definition is given on p.36 and Rmk 2.16 and Def 5.6 in BBK
Therefore, compact special divisors stay as they are in the modularity theorem for $\cM^{\textrm{tor}}_{\bF_p}$.\footnote{We note that in \cite{BBK}, the authors work with Hilbert modular surfaces attached to real quadratic fields with prime discriminant $D$ and state the modularity result using modular forms with level $\Gamma_0(D)$. However, their proof, which uses Borcherds product for the Fourier expansion and the flatness of $\cZ(m,\mu)$, applies for all Hilbert modular surfaces in the setting of vector-valued modular forms by using the original work of Borcherds \cite{Bor98}. We hence deduce modularity for $\cM^{\rm{tor}}_{\bF_p}$. Although the integral special divisors (denoted by $\cT(n)$ in \cite{BBK}) are defined by taking Zariski closure in $\cM^{\rm{tor}}$ of the special divisors on the generic fiber $\cM^{\rm{tor}}_{\bQ}$, this notion coincides with our definition by the flatness of the integral special divisors in both definitions.}
\end{para}
%SAVE: Here I only claim that the modularity hold on the special fiber of a toroidal compactification, not the integral version as we only have flatness of special divisors and the Borcherds product as in Thm 4.4 (iii), but not necessarily Thm 4.4 (iv) --- the work of Bruinier cited here needs discriminant $D$ not divisible by $2$. Thus we have the char 0 and mod p result, but not the integral version.

%SAVE: HMP gives more than Borcherds+Davesh's trick. Borcherds+Davesh only gives modularity in char $p$ while HMP gives the integral version. In the integral version, in addition to the flatness of $\cZ(m,\mu)$, we also need the flatness of the divisor defined by a modular form $f\in M^!_{1-n/2}(\rho_L^\vee)$ to also be flat, which they deduce from computing $q$-expansions, which is the key part in HMP.

%SAVE: Note that this theorem holds true for any $n$ and any maximal $L$. In particular, if we apply to the product of modular curve case, we have $n=2$ and $L^\vee/L$ trivial so the generating series is a holomorphic modular form of weight $2$ level $1$, which must be zero. This is reasonable because this theorem only concerns $Y_0(1)$, not $X_0(1)$, and $Y_0(1)$ has trivial Picard group. Hence in order to recover Chai--Oort from our proof, we need the Borcherds result for product of $X_0(1)$; or in this particular case, we could just do a degree count for modular polynomials. Of course, we also need to bound the contribution from boundary since we cannot use the compact special divisor trick as in the Hilbert case.

\begin{para}\label{def_int}
Recall that we have a finite morphism $\pi: C\rightarrow \cM_{\bar{\bF}_p}$.
When $C$ is proper, for $Z\in \Pic(\cM_{\bF_p})_{\bQ}$, we define $C.Z$ as the degree of $\pi^* Z \in \Pic(C)_{\bQ}$. For \Cref{thm_main}(2), we pick a toroidal compactification $\cM^{\rm{tor}}$ of the Hilbert modular surface $\cM$ and let $C'$ denote the smooth compactification of $C$ and the finite morphism $\pi$ extends to a finite morphism $\pi': C'\rightarrow \cM^{\rm{tor}}_{\bar{\bF}_p}$. Then for a proper divisor $Z$ in $\cM_{\bF_p}$, we use $C.Z$ to denote $\deg_{C'}(\pi^*Z)$; since $Z$ is proper, $C'\cap Z=C\cap Z$ so we only need to consider points in $\cM_{\bar{\bF}_p}$.
\end{para}

\begin{para}\label{def_Eisenstein}
We apply \Cref{thm_HMP} and \S\ref{modularityH} to $\alpha(Z):=C.Z$ defined n \S\ref{def_int} for $Z\in \Pic(\cM_{\bF_p})_{\bQ}$ (and we further assume that $Z$ is proper when $L=L_H$). %SAVE: in our assumption, at least one of $C,Z$ is compact and hence this intersection number is well-defined.
We decompose the modular form $-(\omega.C)\fe_0+\sum_{m> 0, \mu\in L^\vee/L}Z(m,\mu).C q^m \fe_\mu$ as $E(q)+G(q)$, where $E(q)\in M_{1+\frac{n}{2}}(\rho_L)$ is an Eisenstein series and $G(q)\in M_{1+\frac{n}{2}}(\rho_L)$ is a cusp form. Note that the constant term of $E(q)$ is $-(\omega.C)\fe_0$. 

We now recall the vector-valued Eisenstein series $E_0(\tau)\in M_{1+\frac{n}{2}}(\rho_L)$ which has constant term $\fe_0$. This Eisenstein series has been studied in \cite[\S 1.2.3]{Br02}, \cite[\S 4]{BK01}, and \cite[\S 3]{BK03}. Here we follow \cite[\S 2.1]{Br17} as we use the same convention of quadratic forms. We denote an element in $\Mp_2(\bZ)$ by $(g,\sigma)$, where $g=\begin{bmatrix}a & b \\ c & d\end{bmatrix}\in \SL_2(\bZ)$ and $\sigma$ is a choice of the square root of $c\tau+d$. Let $\Gamma'_\infty\subset \Mp_2(\bZ)$ denote the stabilizer of $\infty$. Then for $n\geq 3$, the following summation converges on the upper half plane and we define
\[E_0(\tau)=\sum_{(g,\sigma)\in \Gamma'_\infty\backslash \Mp_2(\bZ)}\sigma(\tau)^{-(2+n)}(\rho_L(g,\sigma)^{-1}\fe_0).\]

%SAVE: compare with the definition in \cite{BK01}: they use $\widetilde{\Gamma}_\infty=\{(\begin{bmatrix} 1 & * \\ 0 &1\end{bmatrix}, 1)\}$, which is a index $4$ subgroup of $\Gamma_\infty'$, hence the $E_0$ in BK01 is twice the $E_0$ defined above.

When $n=2$, we define $E_0(\tau)$ use analytic continuation following \cite[\S 3]{BK03}. Write $\tau=x+iy$ and define for $s\in \bC$,
\[E_0(\tau,s)=\sum_{(g,\sigma)\in \Gamma'_\infty\backslash \Mp_2(\bZ)}\sigma(\tau)^{-(2+n)}(\rho_L(g,\sigma)^{-1}(y^s\fe_0)),\]
which converges on the upper half plane for $s$ with $\Re s>0$ ($n=2$ here). By \cite[p.~1697]{BK03}, $E_0(\tau,s)$ has meromorphic continuation in $s$ to the entire $\bC$ and it is holomorphic at $s=0$ and we define $E_0(\tau)$ to be the value at $s=0$ of the meromorphic continuation of $E_0(\tau,s)$. Moreover, by \emph{loc.~cit.}, $E_0(\tau)$ is holomorphic and hence lies in $M_{1+\frac{n}{2}}(\rho_L)$ if $\rho_L$ does not contain the trivial representation as a subquotient. In the proof of \Cref{thm_main}(2), we work with $L=L_H$ and this condition for $\rho_L$ is always satisfied as far as the $m$ in the statement of \Cref{thm_main}(2) is not a perfect square, i.e., $\cM$ is not the product of modular curves.

We denote the $q$-expansion of $E_0(\tau)$ as $\sum_{m\geq 0, m\in \bZ+Q(\mu)}q_L(m,\mu)q^m \fe_\mu$ and set $q_L(m):=q_L(m,0)$ for $m\in \bZ_{>0}$. 
\end{para}

\begin{para}\label{notation-Lat}
We fix some notations before we state the explicit formula of $q_L(m)$ given by Bruinier--Kuss.
Given a quadratic lattice $L$ (not necessarily the lattice $L_H,L_S$), we write $\det(L)$ for the determinant of its Gram matrix.
%SAVE: In the terminology in \cite{Han13}, we work with the Hessian matrix. More precisely, $2Q(x)=x^tSx$, where $S$ is the matrix whose determinant we take.
We have $|L^\vee/L|=|\det(L)|$. 

For a rational prime $\ell$, we use $\delta(\ell, L,m)$ to denote the local density of $L$ representing $m$ over $\bZ_\ell$. More precisely, $\delta(\ell, L,m)=\lim_{a\rightarrow \infty} \ell^{a(1-\rk L)}\#\{v\in L/\ell^a L \mid Q(v)\equiv m \bmod \ell^a\}$. \cite[Lem. 5]{BK01} asserts that the limit is stable once $a\geq 1+2v_\ell(2m)$. In particular, if $m$ is representable by $(L\otimes \bZ_\ell, Q)$, then $\delta(\ell,L,m)>0$.

Given $0\neq D\in \bZ$ such that $D\equiv 0,1 \bmod 4$, we use $\chi_D$ to denote the Dirichlet character 
$\chi_D(a)=\left(\frac{D}{a}\right)$, where $\left(\frac{\cdot}{\cdot}\right)$ is the Kronecker symbol. For a Dirichlet character $\chi$, we set $\sigma_s(m,\chi)=\sum_{d|m}\chi(d)d^s$. 
\end{para}

\begin{theorem}[Bruinier--Kuss; see also \cite{Br17}*{Thms.~2.3, 2.4}] \label{Ecoeff}
Consider $L=L_H,L_S$ defined in \S\ref{sec_lat_glo} and $m\in \bZ_{>0}$.
\begin{enumerate}
\item For $L=L_H$, the Fourier coefficient $q_L(m)$ is 
\[-\frac{4\pi^2m\sigma_{-1}(m,\chi_{4\det L})}{\sqrt{|L^\vee/L|}L(2,\chi_{4\det L})}\prod_{\ell \mid 2\det(L)}\delta(\ell, L, m).\]

\item For $L=L_S$, write $m=m_0f^2$, where ${\rm{gcd}}(f,2\det L)=1$ and $v_\ell(m_0)\in \{0,1\}$ for all $\ell\nmid 2\det L$. Then the Fourier coefficient $q_L(m)$ is
\[-\frac{16\sqrt{2}\pi^2m^{3/2}L(2,\chi_{\cD})}{3\sqrt{|L^\vee/L|}\zeta(4)}\left(\sum_{d\mid f}\mu(d)\chi_{\cD}(d)d^{-2}\sigma_{-3}(f/d)\right)\prod_{\ell\mid 2\det L}\Big(\delta(\ell,L,m)/(1-\ell^{-4})\Big),\]
where $\mu$ is the Mobius function and $\cD=-2m_0\det L$.\footnote{Since $\det L_S=2$ and more generally for odd rank quadratic lattice $L$, we have $2\mid \det L$, then $\cD\equiv 0\bmod 4$.}

\end{enumerate}
\end{theorem}
\begin{proof}
When $L=L_S$, this is \cite[Thm.~11]{BK01}. When $L=L_H$, one modifies the proof of \cite[Thm.~11]{BK01} as follows. Using \cite[Prop.~3.1]{BK03} instead of \cite[Prop.~2]{BK01}, we obtain \cite[Prop.~3]{BK01} since Shintani's formula works in general. To express the formula in \cite[Prop.~3]{BK01} as a product of local terms, we use \cite[\S 11.5, p.~196]{Iwa97}. The rest of the proof, which computes the local terms at $\ell\nmid 2\det L$, works in the same way (see also \cite[eqns~(11.71)--(11.74)]{Iwa97}).
\end{proof}

If $Z(m)\neq \emptyset$, then $m$ is representable by $(L,Q)$ and in particular for every $\ell$, $m$ is representable by $(L\otimes \bZ_\ell,Q)$ and hence $\delta(\ell, L,m)>0$.
By \Cref{Ecoeff}, we have $q_L(m)<0$ when $Z(m)\neq \emptyset$.

\subsection{The lattice $L'$ and the Siegel mass formula}\label{sec_Siegel-mass}

 \begin{para}\label{summary_L'}
 For a supersingular point $P\in \cM(k)$, we defined $L''$, the lattice of special endomorphisms, in \Cref{def_L''} and picked $L'\supset L''$ which is maximal at all $\ell\neq p$ and $L'\otimes \bZ_p=L''\otimes \bZ_p$. Though there may be choices for $L'$, the local lattices $L'\otimes \bZ_\ell$ are well-defined up to isometry. More precisely, for $\ell\neq p$, $L'\otimes \bZ_\ell$ is given by \Cref{lem_lat_ell}; and for $\ell=p$, $L'\otimes \bZ_p=L''\otimes \bZ_p$ is computed in \S\S\ref{sec_F_Hil}-\ref{sec_F_Sie}. Note that given $L$, the isometry class of the quadratic lattice $L'\otimes \bZ_p$ only depends on whether $P$ is superspecial or supergeneric; indeed, following the notation in \S\ref{summary_HP}, if $t_P=t_{\max}$ (for instance, when $P$ is supergeneric), then $\Lambda_P$ is a maximal lattice with respect to $pQ'$ and hence its isometry class (and thus the isometry class of $L'\otimes \bZ_p=\Lambda_P^\vee$) is unique; if $t_P=2$, i.e., $P$ is superspecial, then $\Lambda_P^\vee$ is a maximal lattice with respect to $Q'$ and hence is unique up to isometry. 
 
 In order to compute the local intersection number of $Z(m).C$ at $P$, we also need to consider sublattices $L'''$ of $L'$ such that $L'''\otimes \bZ_\ell=L'\otimes \bZ_\ell$ for all $\ell\neq p$ (more precisely, we will take $L'''$ to be the lattices defined in \S\ref{def_smallL'}). In particular, $\det L'''=p^{2a} \det L'$ for some $a\in \bZ_{\geq 0}$.
 
 Let $\theta_{L'''}(q)$ denote the theta series of the positive definite lattice $L'''$, which is a modular form of weight $\rk L'/2$; we decompose $\theta_{L'''}(q)=E_{L'''}(q)+G_{L'''}(q)$, where $E_{L'''}$ is an Eisenstein series and $G_{L'''}$ is a cusp form. Let $q_{L'''}(m)$ denote the $m$-th Fourier coefficients of $E_{L'''}$ (at the cusp $\infty$). The following theorem asserts that $q_{L'''}(m)$ only depends on the genus of $L'''$ and gives explicit formula for $q_{L'''}(m)$. In particular, when we consider the theta series for $L'$, we have that $q_{L'}(m)$ is independent of the choice of $L'$ above and it only depends on $L$ and whether $P$ is superspecial or supergeneric.
 \end{para}

\begin{theorem}[Siegel mass formula]\label{thm_Siegel-mass}
Notation as in \S\ref{summary_L'}. The Eisenstein series $E_{L'''}$ only depends on the genus of $L'''$. Moreover, for $m\in \bZ_{>0}$,
\begin{enumerate}
    \item when $L=L_H$, 
\[q_{L'''}(m)=\frac{4\pi^2m\sigma_{-1}(m,\chi_{4\det L'})}{\sqrt{|L'''^\vee/L'''|}L(2, \chi_{4\det L'})}\prod_{\ell \mid 2 det L'}\delta(\ell, L''',m);\]
\item when $L=L_S$,

\[q_{L'''}(m)=\frac{16\sqrt{2}\pi^2m^{3/2}L(2,\chi_{\cD'})}{3\sqrt{|L'''^\vee/L'''|}\zeta(4)}\left(\sum_{d\mid f}\mu(d)\chi_{\cD'}(d)d^{-2}\sigma_{-3}(f/d)\right)\prod_{\ell\mid 2\det L'}\Big(\delta(\ell,L''',m)/(1-\ell^{-4})\Big),\]
where we write $m=m_0f^2$ with ${\rm{gcd}}(f,2\det L')=1$ and $v_\ell(m_0)\in \{0,1\}$ for all $\ell\nmid 2\det L'$ and $\cD'=-2 m_0 \det L'$
\end{enumerate}
\end{theorem}
\begin{proof}
The first assertion follows from the Siegel mass formula; see for instance \cite[Thm.~20.9, eqn.~(20.121), and pp.~479-480]{IK04}. In order to obtain the formula above, we note that the proof of \cite[Thm.~11]{BK01} using \cite[Thm.~6]{BK01} also applies to $L'''$ and hence we conclude that the formula in \cite[Thms.~2.3, 2.4]{Br17} also applies to $L'''$ and obtain the formulae in the theorem with all $L'$ replaced by $L'''$. Note that by the computations in \S\S\ref{sec_F_Hil}-\ref{sec_F_Sie}, we have $p\mid \det L'$, and hence $\ell \mid 2 \det L'''$ if and only if $\ell \mid 2 \det L'$; also $\chi_{4 \det L'''}=\chi_{4 \det L'}$ and $\chi_{D'}=\chi_{-2m_0 \det L'''}$. Hence using $L'$ (instead of $L'''$) for $\chi, \cD'$ and the product $\ell \mid 2 \det L'$ yields the same formulae.
\end{proof}
%SAVE: I do not have a reference for the vector-valued version. Kudla--Rallis provides an adelic version of Siegel mass formula and their proof works for any $n\geq 1$.

\subsection{The asymptotic of $q_L(m)$}
The discussion of this subsection also applies to $q_{L'''}(m)$ when $m$ is representable by $(L''',Q')$, but we only focus on $q_L(m)$ here.

\begin{para}\label{sec_asymp-qL}
Assume that $m$ is representable by $(L\otimes \bZ_\ell,Q)$ for every prime $\ell$. We will also assume that, as $m$ varies within a specified set $T$, there exists an absolute constant $C>0$ such that for all $\ell\mid 2 \det L$, we have $v_\ell(m)\leq C$.  As we shall see in \S\ref{def_setT}, we will always be in this situation.

For a given $\ell \mid 2 \det L$, as in \cite[proof of Prop.~2.5]{Br17}, by \cite[Lem.~5]{BK01}, we have $\delta(\ell,L,m)=\ell^{a(1-\rk L)}\#\{v\in L/\ell^a L \mid Q(v)\equiv m \bmod \ell^a\}$ with $a=1+2C+2v_\ell(2)$ and hence $\ell^{a(1-\rk L)}\leq \delta(\ell,L,m)\leq \ell^a$.\footnote{When $\rk L\geq 5$, for a fixed $\ell$, it is well known that $\delta(\ell,L,m)\asymp 1$ for all $m$ representable by $(L\otimes \bZ_\ell,Q)$ without imposing any bound on $v_\ell(m)$; see for instance \cite[pp.~198-199]{Iwa97}.}

Therefore, given $(L,Q)$, by \Cref{Ecoeff}, we have that $|q_L(m)|\asymp m \sigma_{-1}(m, \chi_{4\det L})$ and hence $m^{1-\epsilon}\ll_\epsilon |q_L(m)|\ll_\epsilon m^{1+\epsilon}$ for $L=L_H$; and $|q_L(m)|\asymp m^{3/2} L(2, \chi_{\cD}) \sum_{d\mid f} \mu(d) \chi_{\cD}(d) d^{-2} \sigma_{-3}(f/d)$ for $L=L_S$. As in the proof of \cite[Prop.~2.5]{Br17}, we have $\sum_{d\mid f}\mu(d)\chi_{\cD}(d)d^{-2}\sigma_{-3}(f/d)\geq 1/5$ and \[\sum_{d\mid f}\mu(d)\chi_{\cD}(d)d^{-2}\sigma_{-3}(f/d)\leq \sum_{d\mid f}d^{-2}\sigma_{-3}(f/d) <\sum_{d\mid f} d^{-2}\zeta(3)<\zeta(2)\zeta(3);\] moreover, by \emph{loc.~cit.}, $L(2,\chi_D)\geq \zeta(4)/\zeta(2)$ and $L(2,\chi_D)\leq \prod_{p} (1-p^{-2})^{-1}=\zeta(2)$. Hence $|q_L(m)|\asymp m^{3/2}$ when $L=L_S$.
\end{para}

\begin{lemma}\label{lem_glo-bound}
We fix the same assumptions as in \S\ref{sec_asymp-qL}. For $m\gg 1$, we have $Z(m)\neq \emptyset$ and the intersection number $Z(m).C = -q_L(m)(\omega . C)+o(|q_L(m)|)$. More precisely, when $L=L_H$, the error term can be bounded by $O_\epsilon(m^{1/2+\epsilon})$ and when $L=L_S$, the error term can be bounded by $O(m^{5/4})$.
\end{lemma}

\begin{proof}
We follow the discussion in \S\ref{def_Eisenstein}. Let $g(m), m\in \bZ_{>0}$ denote the $m$-th Fourier coefficients of $\fe_0$-component of $G(q)$, which is also a cusp form of weight $1+\frac{n}{2}$ with respect to certain congruence subgroup of $\Mp_2(\bZ)$ depending on $L$. When $L=L_H$, by Deligne's bound (\cite{D73,D74}), we have $|g(m)|\ll m^{1/2}\sigma_0(m)\ll_\epsilon m^{1/2+\epsilon}=o_\epsilon (m^{1-\epsilon})=o(|q_L(m)|)$ for any $0<\epsilon<1/4$. When $L=L_S$, the trivial bound yields $|g(m)|\ll m^{5/4}=o(m^{3/2})$ (see \cite[Prop.~1.3.5]{Sar90}).
%SAVE: the trivial bound $O(m^{wt/2})$ holds for arbitrary level $\Gamma$--not necessarily congruence. In general for bounds of Fourier coefficients $g(Dm^2)$ with $D$ fixed, we may deduce the result for half integer weights from that for integer weights by Shimura lift (may with an extra $\epsilon$); see \cite[\S 4]{Han04}). Another reference between the trivial bound and the Deligne's bound is Sarnak "some applications of modular forms" Props. 1.5.3 (for congruence $\Gamma$), 1.5.5 (for level being $\Gamma_0(N)$). Note that for theta series, we have $\Gamma_0(N)$ level, but I haven't checked the global lattice case yet.
Therefore by \Cref{thm_HMP}, $Z(m).C=-q_L(m)(\omega.C)+o(|q_L(m)|)$; in particular, for $m\gg 1$, $Z(m).C>0$ and hence $Z(m)\neq \emptyset$.
\end{proof}

\begin{para}\label{def_setT}
When $L=L_S$, recall from \S\ref{sec_lat_glo} that the quadratic form is $Q(x)=x_0^2+x_1x_2-x_3x_4$ and hence every $m\in \bZ_{>0}$ is representable by $(L,Q)$. In particular, $Z(m)\neq \emptyset$ and $\delta(\ell,L,m)>0$ for all $\ell$. Moreover, in order to prove \Cref{thm_main}(1) and \Cref{rmk_realquad}, we will work with $m\in T:=\{Dq^2\mid q \text{ prime and }q\neq p \}$, where we take $D=1$ for \Cref{thm_main}(1) and $D$ being the discriminant of the real quadratic field in \Cref{rmk_realquad}; and for \Cref{thm_max}, we work with $m\in T:=\{q\mid q \text{ prime and }q\neq p, q \text{ is a quadratic residue}\bmod p, \text{ and } q\equiv 3 \bmod 4\}$. %SAVE: the quadratic residue condition implies that $p$ is split in $\bQ(\sqrt{q})$ and hence we only have superspecial points.
In particular, for all such $m$, we have $v_\ell(m)\leq 2+v_\ell(D)$ and hence the assumptions in \S\ref{sec_asymp-qL} are satisfied.

When $L=L_H$, since $L$ is maximal and isotropic, we have that the quadratic form on $L\otimes \bZ_\ell$ is given by $xy+Q_1(z)$, where $x,y \in \bZ_\ell, z\in \bZ_\ell^2$ and $Q_1$ is some quadratic form. Then $\delta(\ell, L,m)>0$ for all $\ell$; indeed, by \cite[Def.~3.1, Lem.~3.2]{Han04}, $\delta(\ell,L,m)>0$ if there exists $x,y \in \bZ/\ell^{1+2v_\ell(2)}$ such that $xy\equiv m \bmod \ell^{1+2v_\ell(2)}$ and $x\not\equiv 0\bmod \ell$ (by the terminology in \cite{Han04}, this construct a good type solution (taking $z=0$) for $(L,Q)\bmod \ell^{1+2v_\ell(2)}$, which can be lifted to $\bZ/\ell^k$ for any $k\geq 1+2v_\ell(2)$). Such $x,y$ always exists and hence every $m\in \bZ_{>0}$ is representable by $(L\otimes \bZ_\ell,Q)$ for all $\ell$ and hence by \Cref{lem_glo-bound}, there exists $N\in \bZ_{>0}$ such that for all $m>N$, $m$ is representable by $(L,Q)$. For the proof of \Cref{thm_main}(2), we work with $m$ in \[T:=\{m\in \bZ\mid m>N, p\nmid m, v_\ell(m)\leq C,\forall \ell \mid 2\det L, \text{ and }\exists\, q || m \text{ such that } q \text{ inert in }F\},\]
where $F$ is the real quadratic field attached to the Hilbert modular surface and the constant $C$ is chosen so that this set is non-empty. The existence of $q$ implies that $m\neq \Nm_{F/\bQ}\gamma$ for any $\gamma\in F$ and hence for any $v\in L_H\otimes \bQ$ such that $Q(v)=m$, we have $v^\perp\subset L_H\otimes \bQ$ is anisotropic. Note that if $Z(m)$ is non-compact in $\cM_{\bF_p}$, then $Z(m)$ parametrizes abelian surfaces which are isogenous to the self-product of elliptic curves and then $v^\perp$ is isotropic. Therefore, for any $m\in T$, we have that $Z(m)$ is compact in $\cM_{\bF_p}$. Note that $T\subset \bZ_{>0}$ is of positive density.
\end{para}

\begin{lemma}\label{lem_sum-asymp-H}
For $L=L_H$ and $M>0$, we have $\sum_{1\leq m \leq M, m\in T}|q_L(m)|\asymp M^2$. 
\end{lemma}
%SAVE: I am not sure how to modify the proof to rule out density $\epsilon$ sets given by numbers representable by quadratic forms with large discriminant. Hence, for the proof of the main theorem in the Hilbert case, we also need Keerthi's ordinary lemma.

%SAVE: a precise form for $T$ only defined using certain congruence conditions (this is the old version): Fix $a,D\in \bN$ and $\chi$ a character such that $\chi(\ell)=0$ for all $1<\ell\mid D$. Then \[\sum_{m\leq X, m\equiv a \bmod D}m\sigma_{-1}(m,\chi)=\frac{L(2,\chi)}{2D}X^2+O(X^{3/2}).\]

\begin{proof}
%SAVE: the following argument is the standard hyperbola method used to sum the convolution of two functions.
By \S\S\ref{sec_asymp-qL},\ref{def_setT}, we have for $m\in T$, $|q_L(m)|\asymp m\sigma_{-1}(m,\chi)$, where $\chi=\chi_{4\det L}$. We write
\begin{eqnarray*}
\sum_{1\leq m \leq M, m\in T} m\sigma_{-1}(m,\chi) & =  & \sum_{1\leq m\leq M, m\in T}\sum_{d\mid m}d\cdot \chi(m/d)=  \sum_{1\leq d\leq M, 1\leq f\leq M, df\leq M, df\in T} d\cdot \chi(f)\\
&=& \sum_{1\leq d\leq M^{1/2}}d \sum_{1\leq f\leq M/d, df\in T}\chi(f)+\sum_{1\leq f\leq M^{1/2}}\chi(f)\sum_{1\leq d\leq M/f, df\in T}d\\
& &-(\sum_{1\leq d\leq M^{1/2}}d(\sum_{1\leq f\leq M^{1/2}, df\in T}\chi(f))).
\end{eqnarray*}
Note that \[|\sum_{1\leq d\leq M^{1/2}}d \sum_{1\leq f\leq M/d, df\in T}\chi(f)|\leq \sum_{1\leq d\leq M^{1/2}}d\cdot (M/d)=O(M^{3/2}),\] 
\[|(\sum_{1\leq d\leq M^{1/2}}d(\sum_{1\leq f\leq M^{1/2}, df\in T}\chi(f)))|\leq (\sum_{1\leq d\leq M^{1/2}}d)\cdot (\sum_{1\leq f\leq M^{1/2}}1)=O( M^{3/2}).\] 
The second term is the main term. First let $T':=\{m\in \bZ\mid m>N, p\nmid m, v_\ell(m)\leq C,\forall \ell \mid 2\det L\}$ then 
\[\sum_{1\leq f\leq M^{1/2}}\chi(f)\sum_{1\leq d\leq M/f, df\in T'}d=\sum_{1\leq f\leq M^{1/2}, p\nmid f}\chi(f)\sum_{1\leq d\leq M/f, p\nmid d, v_\ell(d)\leq C,\forall \ell \mid 2 \det L} d,\]
because $v_\ell (df)\leq C \iff v_\ell(d)\leq C, \forall \ell \mid 2 \det L$ when $v_\ell(f)=0, \forall \ell \mid 2 \det L$ and if $v_\ell(f)>0$ for some $\ell \mid 2 \det L$, then $\chi(f)=0$.
Since $\displaystyle\sum_{1\leq d\leq M/f, p\nmid d, v_\ell(d)\leq C,\forall \ell \mid 2 \det L}d=C_1 \frac{M^2}{f^2}+O(M)$, where $C_1$ and the implicit constant only depend on $C,L,p$. Hence \[\sum_{1\leq f\leq M^{1/2}}\chi(f)\sum_{1\leq d\leq M/f, df\in T'}d=C_1M^2\sum_{1\leq f\leq M^{1/2}, p\nmid f}\chi(f)/f^2+O(X^{3/2})\asymp M^2.\]
To finish the proof, we only need to show that $\displaystyle|\sum_{1\leq f\leq M^{1/2}}\chi(f)\sum_{1\leq d\leq M/f, df\in T'-T}d|=o(M^2)$. Since $M/f\geq M^{1/2}$, by definition of $T$, $\#\{d\mid 1\leq d \leq M/f, df\in T'-T\}=o(M/f)$ with implicit constant independent of $f$ and hence we obtain the desired bound. %SAVE: roughly speaking, for an inert prime $q$ of $F$, it introduces a multiple of $1-q^{-1}+q^{-2}$ (if $2\mid v_q(f)$) or $q^{-1}$ (if $2\nmid v_q(f)$) in the density of $T'-T$ and hence when we use all $q$, the density is $0$ no matter which multiple factor we pick.
%SAVE: I do not know how to generalize this step to $T'$ removing a small density set of $m$ representable by quadratic form with large discriminant. Indeed, given $f$, for $df$ to be representable, it might have $\sigma(f)$ many possibilities out of the totally class number of the quadratic form. 
\end{proof}

%SAVE: here we give a possible fix if we do not want to use Serre--Tate for ordinary points in the Hilbert case. Indeed, we obtain finitely many binary quadratic forms for which we want to avoid $m$ representable by these forms. Note that for each quadratic form, we may pick a prime $q_i$, which is inert in the quadratic field of the quadratic form and is not representable (note, if the form is $ax^2+by^2$, then any inert prime $\nmid a, b$ is not representable). Let $N$ denote the product of all $q_i$ (only pick the distinct ones so that $N$ is square free). Then we work with $m$ such that $N\mid m$ and $q_i^2\nmid m$, this is a congruence condition which we could incorporate in the estimate.

\subsection{Local densities at $p$ and the ratios of Fourier coefficients}
We set the same notation as in \S\ref{summary_L'}. \Cref{Ecoeff} and \Cref{thm_Siegel-mass} reduce the comparison between $q_L(m)$ and $q_{L'''}(m)$ to the computation of the local density $\delta(p,L''',m)$, which we now compute following \cite[\S 3]{Han04}. Recall that $p$ is an odd prime and $v_p(m)\leq 1$ for all $m\in T$ defined in \S\ref{def_setT}.
For an arbitrary quadratic lattice $(L,Q)$, let $\alpha(p,L,m):=p^{1-\rk L}\#\{v\in L/pL \mid Q(v)\equiv m \bmod p\}$; if we diagonalize $L\otimes \bZ_p$ such that $Q$ is given by $\sum_{i=1}^{\rk L} a_i x_i^2$ with $a_i\in \bZ_p$, then we define \[\alpha^*(p,L,m):=p^{1-\rk L}\#\{v=(x_1,\dots, x_{\rk L})\in L/pL \mid Q(v)\equiv m, \exists i \text{ such that }v_p(a_i)=0, x_i \not\equiv 0 \bmod p\}.\]
\begin{lemma}[Hanke]\label{lem_Hanke}
If $p\nmid m$, we have \[\delta(p,L''',m)=\alpha(p,L''',m);\] if $v_p(m)=1$, we have \[\delta(p,L''',m)=\alpha^*(p,L''',m)+ p^{1-s_0}\alpha(p, L'''_I, m/p),\] where if we write $(L'''\otimes \bZ_p, Q')$ into diagonal form $\sum_{i=1}^{\rk L'''}a_i x_i^2$ with $a_i\in \bZ_p$, we define $s_0=\#\{a_i \mid v_p(a_i)=0\}$ and $L'''_I$ is the quadratic lattice with quadratic form $\sum_{i=1}^{\rk L'''}a'_i x_i^2$, where $a'_i =p a_i$ if $v_p(a_i)=0$ and $a'_i=p^{-1} a_i$ if $v_p(a_i)\geq 1$.
\end{lemma}
\begin{proof}
If $p\nmid m$, the assertion follows from \cite[Rmk.~3.4.1 (a), Lem.~3.2]{Han04}; If $v_p(m)=1$, then we only have good type and bad type I solutions in the sense of \cite[Def.~3.1, p.~360]{Han04} and the assertion follows from \cite[Lem.~3.2, p.~360, Rmk.~3.4.1 (a)]{Han04}.
\end{proof}

We first compute $\delta(p,L',m)$ by \Cref{lem_Hanke}. We always pick $\epsilon\in \bZ_p^\times \backslash (\bZ_p^\times)^2$ as in \S\ref{summary_HP}.
\begin{para}\label{par_density_H}
Consider $L=L_H$ and recall that $p\nmid m, \forall m \in T$. Let $F$ denote the real quadratic field attached to the Hilbert modular surface defined by $L_H$. 
\begin{enumerate}
    \item Assume that  $p$ is inert in $F$ and $P$ is supergeneric. By \S\ref{Frob-H-inert}, $L'\otimes \bZ_p=\Lambda^\vee=p\Lambda$ and hence $p\mid Q'(v),\forall v\in L'$; in particular, $\delta(p,L',m)=0$.%SAVE: $Q'=pxy+p(z^2-\epsilon w^2)$, where $\epsilon\in \bZ_p$ is a quadratic non-square.
    \item Assume that  $p$ is inert in $F$ and $P$ is superspecial.  By \S\ref{Frob-H-inert}, $Q'(v)=xy+p(z^2-\epsilon w^2)$, where $w_i$ are given right above \eqref{FrobHinert} and $v=xw_3+yw_4+zw_1+ww_2$ with $x,y,z,w\in \bZ_p$. Hence $\delta(p,L',m)=\alpha(p,L',m)=1-1/p$.
    \item Assume that  $p$ is split in $F$; hence $P$ is superspecial.  By \S\ref{parsplit}, $L'\otimes \bZ_p=\Lambda^\vee$ with $Q'(v)= x^2-\epsilon y^2-pz^2+\epsilon pw^2$, where $v=xe_1+ye_2+z(pe_3)+w(pe_4)$ with $x,y,z,w\in \bZ_p$. Hence $\delta(p,L',m)=\alpha(p,L',m)=1+1/p$. 
\end{enumerate}
\end{para}

\begin{para}\label{par_density_S}
Consider $L=L_S$.
\begin{enumerate}
    \item Assume that $P$ is superspecial. By \S\ref{sec_F_Sie}, we have $Q'(v)=xy+\epsilon z^2+pw^2-p\epsilon u^2$, where $v=xw_3+yw_4+zw_5+ww_2+uw_1$ with $x,y,z,w,u\in \bZ_p$ and $w_i$ are given right above \eqref{F-Sie}. Hence if $p\nmid m$, then $\delta(p,L',m)=\alpha(p,L',m)\leq 1+1/p$ by \cite[Table 1]{Han04}. %SAVE: the density is $1\pm 1/p$ depending on whether $m$ is a quadratic residue.
    If $v_p(m)=1$, then the quadratic form of $L'_I$ is $p(xy+\epsilon z^2)+ w^2-\epsilon u^2$ and hence $\delta(p,L',m)=\alpha^*(p,L',m)+p^{-2}\alpha(p,L'_I,m/p)=(1-p^{-2})+p^{-2}(1+p^{-1})=1+p^{-3}$.
    
    \item Assume that $P$ is supergeneric. By \S\ref{sec_F_Sie}, $L'\otimes \bZ_p=\Lambda^\vee$ and hence the quadratic form is $pxy+\epsilon z^2+pw^2-p\epsilon u^2$. If $p\nmid m$, then $\delta(p,L',m)=\alpha(p,L',m)=0 \text{ or } 2$; if $v_p(m)=1$, then the quadratic form of $L'_I$ is $p\epsilon z^2+xy+ w^2-\epsilon u^2$ and hence $\delta(p,L',m)=\alpha^*(p,L',m)+\alpha(p,L'_I,m/p)=0+1+p^{-2}=1+p^{-2}$ by \cite[Table 1]{Han04}.
\end{enumerate}
\end{para}

We now estimate $\delta(p,L''',m)$ for sublattices lattices $L'''$ of $L'$ defined in \S\ref{summary_L'}.

\begin{lemma}\label{lem_den_L'''}
If $p\nmid m$, then $\delta(p,L''',m)\leq 2$.
\end{lemma}
\begin{proof}
By \Cref{lem_Hanke}, $\delta(p,L''',m)=\alpha(p,L''',m)$. Write the quadratic form $Q'$ on $L'''$ into the diagonal form $\sum_{i=1}^{\rk L'''} a_i x_i^2$ with $a_i\in \bZ_p$ and we may assume that there exists $a_i$ such that $p\nmid a_i$; otherwise $\delta(p,L''',m)=0$ then we are done. Now let $\widetilde{L}'''$ denote the quadratic form $\sum_{1\leq i\leq \rk L''', p\nmid a_i} a_i x_i^2$. Then by definition, $\alpha(p,L''',m)=\alpha(p, \widetilde{L}''',m)$.
 Since $p\mid \disc L'$, then $p\mid \disc L'''$ and $\rk \widetilde{L'''}\leq \rk L'''-1\leq 4$. Then by \cite[Table 1]{Han04}, $\alpha(p, \widetilde{L}''',m)\leq 2$ and hence $\delta(p,L''',m)\leq 2$.
\end{proof}

\begin{lemma}\label{lem_den_L'''S}
Assume that $L=L_S$ and $v_p(m)=1$. We have $\delta(p,L''',m)\leq 2+2p$. Moreover, if $P$ is superspecial and $[L':L''']=p$, then $\delta(p,L''',m)\leq 4$.
\end{lemma}
\begin{proof}
By \Cref{lem_Hanke}, $\delta(p,L''',m)=\alpha^*(p,L''',m)+p^{1-s_0}\alpha(p,L'''_I,m/p)\leq \alpha(p,L''',m)+p\alpha(p,L'''_I,m/p)$. By the proof of \Cref{lem_den_L'''}, we have $\alpha(p,L''',m)=\alpha(p,\widetilde{L}''',m)\leq 2$. The same argument implies that $\alpha(p,L'''_I,m/p)\leq 2$ if $\rk(\widetilde{L}''')\leq 4$. If $\rk(\widetilde{L}''')=5$, then it is isotropic and we write the quadratic form as $xy+Q_1(z)$. The equation $xy+Q_1(z)\equiv (m/p) \bmod p$ has $(p-1)p^3$ solutions in $\bF_p^5$ with $x\neq 0$ and has at most $p^4$ solutions with $x=0$. Hence $\alpha(p, L''',m/p)=\alpha(p, \widetilde{L}''',m/p)<2$. Therefore, $\delta(p,L''',m)\leq 2+2p$.

If $P$ is superspecial and $[L':L''']=p$, then $s_0\geq 1$ and hence $\delta(p,L''',m)\leq \alpha^*(p,L''',m)+\alpha(p,L'''_I,m/p)\leq 4$.
\end{proof}

The following lemma, which is the main goal of this subsection, will be used to compare the local intersection number at a supersingular point $P$ with the global intersection number.

\begin{lemma}\label{compare_qL}
Notation as in \S\ref{summary_L'} and consider $m\in T$ (defined in \S\ref{def_setT}).
\begin{enumerate}
    \item If $P$ is superspecial or $L=L_H$, then $\displaystyle\frac{q(m)_{L'}}{-q(m)_L}\leq \frac{1}{p-1}$.
    \item If $L=L_S$ and $P$ is supergeneric, then $\displaystyle\frac{q(m)_{L'}}{-q(m)_L}\leq \frac{2}{p^2-1}$.
    \item If $p\nmid m$, then $\displaystyle \frac{q(m)_{L'''}}{-q(m)_L}\leq \frac{2}{\sqrt{|(L'''\otimes \bZ_p)^\vee/(L'''\otimes \bZ_p)|}(1-p^{-2})}$.
    \item Assumption as in \Cref{lem_den_L'''S}, then $\displaystyle \frac{q(m)_{L'''}}{-q(m)_L}\leq \frac{2p}{\sqrt{|(L'''\otimes \bZ_p)^\vee/(L'''\otimes \bZ_p)|}(1-p^{-1})}$; moreover, if $P$ is superspecial and $[L':L''']=p$, then $\displaystyle \frac{q(m)_{L'''}}{-q(m)_L}\leq \frac{4}{p^2-1}$.
\end{enumerate}

\end{lemma}
\begin{proof}
Recall from \S\ref{summary_L'} that $L'''\otimes \bZ_\ell \cong L\otimes \bZ_\ell, \forall \ell\neq p$; hence for $\ell\neq p$, we have $\delta(p,L''',m)=\delta(p,L,m)$ and $\det L'''=p^k \det L$ for some $k\in \bZ_{\geq 0}$. Since $L$ is self-dual at $p$, then $p\nmid \det L$; by \S\ref{summary_HP}, $\det L'=p^{2b} \det L$ for some $b\in \bZ_{>0}$ (concretely, one may deduce this fact by the explicit formula of $Q'$ in \S\S\ref{par_density_H}-\ref{par_density_S}) and hence $k\in 2 \bZ_{>0}$. Thus $\chi_{4 \det L}(d)=\chi_{4 \det L'}(d)$ and $\chi_{-2 m_0 \det L}(d)=\chi_{-2 m_0 \det L'}(d)$ if $p\nmid d$.

Therefore, by \Cref{Ecoeff} and \Cref{thm_Siegel-mass}, we have that for $L=L_H$, $p\nmid m$
\[\frac{q(m)_{L'''}}{-q(m)_L}=\frac{\delta(p,L''',m)}{\sqrt{|(L'''\otimes \bZ_p)^\vee/(L'''\otimes \bZ_p)|}(1-\chi_{4\det L}(p)p^{-2})}\leq \frac{\delta(p,L''',m)}{\sqrt{|(L'''\otimes \bZ_p)^\vee/(L'''\otimes \bZ_p)|}(1-p^{-2})};\]
for $L=L_S$, $v_p(m)\leq 1$, we observe that $m_0$ remains the same for $L$ and $L'''$ and $p\nmid f$ and hence
\[\frac{q(m)_{L'''}}{-q(m)_L}=\frac{\delta(p,L''',m)(1-\chi_{\cD}(p)p^{-2})}{\sqrt{|(L'''\otimes \bZ_p)^\vee/(L'''\otimes \bZ_p)|}(1-p^{-4})}\leq \frac{\delta(p,L''',m)}{\sqrt{|(L'''\otimes \bZ_p)^\vee/(L'''\otimes \bZ_p)|}(1-p^{-2})}.
\]
Therefore, (1)(2) follow from \S\S\ref{par_density_H}-\ref{par_density_S}; (3) follows from \Cref{lem_den_L'''}; (4) follows from \Cref{lem_den_L'''S}.
\end{proof}
%SAVE: precise form in some cases (1)$L=L_H$, $p$ inert in $F$, $P$ superspecial; the ratio is $(p-1)/(p^2+1)$. (2) $L=L_H$, $p$ inert, $P$ supergeneric; ratio is $0$. (3) $L=L_H$, $p$ is split in $F$. the ratio is $1/(p-1)$. (4) $L=L_S$, $P$ superspecial, $p\nmid m$, $m$ not square mod $p$; then $\delta=1+1/p$, $\chi_{\cD}(p)=\chi_{-4m_0}(p)=1$ if $-1$ is not a square mod $p$, this case gives the upper bound. (5) $L=L_S$, $P$ superspecial, $p||m$. In this case, $p\mid m_0$ and hence $\chi_{\cD}(p)=0$;  Therefore the ratio is $(1+p^{-3})/(p-p^{-3})$.

%%%%%%%%%%%%%%%%%%%%%%%%%%%%%%%%%%%%%%%%%%%Section 4%%%%%%%%%%%%%%%%%%%%%%%%%%%%%%%%%%%%%%%%%%%%%%%%%%%%%%%%%%%%%%%%%%

\section{The decay lemma for supersingular points and its proof in the Hilbert case}\label{sec_decay_Hil}
The goal of this section is to prove that special endomorphisms ``decay rapidly''. More precisely, consider a generically ordinary two-dimensional abelian scheme over $\bar{\bF}_p[[t]]$ whose special fiber is supersingular. We consider the lattice of special endomorphisms of the abelian scheme mod $t^N$ as $N$ varies, and establish bounds for the covolume of these lattices.  These bounds are exactly what we need to bound the local intersection multiplicity $\Spf \bar{\bF}_p[[t]] \cdot Z(m)$ --  see \Cref{oleg}. The precise definitions and results are in \Cref{def_decay} and \Cref{thm_decay}.

Throughout this section, as in \S\ref{sec_lattices}, $k=\bar{\bF}_p$, $W=W(k)$, $K=W[1/p]$. We focus on the behavior of the curve $C$ in \Cref{thm_main,thm_max} in a formal neighborhood of a supersingular point $P$, so we may let $C = \Spf k[[t]]$ denote a generically ordinary formal curve in $\cM_k$ which specializes to $P$.
As in \S\ref{summary_Kisin}, $\sigma$ denote both the Frobenius on $K$ and the Frobenius on the coordinate rings $W[[x,y], W[[x,y,z]]$ of $\widehat{\cM}_{P}$, which is the unique extension of the Frobenius action on $W$ for which $\sigma(x)=x^p, \sigma(y)=y^p, \sigma(z)=z^p$. For a matrix $M$ with entries in $K[[x,y]]$ or $K[[x,y,z]]$, we use $M^{(n)}$ to denote $\sigma^n(M)$. Also recall we set $\lambda \in \bZ_{p^2}^\times$ such that $\sigma(\lambda) = - \lambda$.
We use $\sigma_t$ to denote the Frobenius on $K[[t]]$ which extends $\sigma$ on $K$ and sends $t$ to $t^p$. 

\subsection{Statement of the Decay Lemma and the first reduction step}\label{sec_decay_statement}

The map $C\rightarrow \cM_k$ gives rise to a local ring homomorphism from $k[[x,y]] \rightarrow k[[t]]$ (in the Hilbert case) or $k[[x,y,z]]\rightarrow k[[t]]$ (in the Siegel case), and we denote by $x(t)$, $y(t)$, and $z(t)$ the images of $x$, $y$, and $z$ respectively. Let $v_t$ denote the $t$-adic valuation map on $k[[t]]$. Let $A$ denote the $t$-adic valuation of the local equation defining the non-ordinary locus in \Cref{eqn-non-ord}. More precisely, if $P$ superspecial, then $A=v_t(xy)$ in the Hilbert case and $A=v_t(xy+\frac{z^2}{4\epsilon})$ in the Siegel case.

\begin{defn}\label{def_decay}
Let $w$ denote a special endomorphism of the $p$-divisible group at $P$ (i.e., $w$ is an element in $L'\otimes \bZ_p$; see \Cref{def_spend_pdiv} and \Cref{def_Lcris_H}). 
\begin{enumerate}
\item We say that $w$ \emph{decays rapidly} if $p^n w$ does not lift to an endomorphism modulo $t^{A_n+1}$ for all $n\in \bZ_{\geq 0}$, where $A_n:=[A(p^n+p^{n-1}+\cdots+1+\frac{1}{p})]$; here $[x]$ denote the maximal integer $y$ such that $y\leq x$.
\item We say that a $\bZ_p$-submodule of $L'\otimes \bZ_p$ \emph{decays rapidly} if every primitive vector in the submodule decays rapidly. 
\item We say that $w$ decays \emph{very} rapidly if $p^n w$ does not lift to an endomorphism modulo $t^{A_{n-1} + ap^n+1}$ for some constant $a \leq A/2$, for all $n\in \bZ_{\geq 0}$, where $A_n$ is defined in (1) and we define $A_{-1}=[A/p]$.  
\end{enumerate}
\end{defn}

\begin{theorem}[Decay Lemma]\label{thm_decay}
Assume $P$ is superspecial. There exists a rank $3$ $\bZ_p$-submodule of $L'\otimes \bZ_p$ which decays rapidly and furthermore, there is a primitive vector in this submodule which decays \emph{very} rapidly.  
\end{theorem}
Here we only state the decay lemma for a superspecial point since we do not need to work with supergeneric points to prove \Cref{thm_main,thm_max}. We refer the reader to the appendix for a decay lemma when $P$ is supergeneric.

\begin{proof}[Proof assuming \Cref{prop_decay}]
For $m\in \bZ_{\geq 0}$, let $S_m$ denote $\Spec k[t]/(t^m)$ and let $D_m$ denote the $p$-adic completion of the PD enveloping algebra of the ideal $(t^m,p)$ in $W[[t]]$. Let $\iota_m$ denote the composite map $S_m \rightarrow \Spf k[[t]] \rightarrow \Spf k[[x,y]] \text{ or } \Spf k[[x,y,z]]$. 
Then by \cite[\S 2.3]{dJ95}, there exists a functor from the category of $p$-divisible groups over $S_m$ to the category Dieudonn\'e modules over $D_m$. More precisely, a special endomorphism $\tilde{w}_m$ of the $p$-divisible group over $S_m$ which specializes to $w\in L'\otimes \bZ_p$ gives rise to an endomorphism of the Dieudonn\'e module which specializes to $w$. By functoriality of Dieudonn\'e modules, images of special endomorphisms are horizontal sections of $\iota_m^* \bL_{\cris}(D_m)$ stable under the Frobenius action; here the connection on $\iota_m^* \bL_{\cris}(D_m)$ is the pull-back of the connection on $\bL_{\cris}(W[[x,y]]), \bL_{\cris}(W[[x,y,z]])$ by a ring homomorphism $W[[x,y]]\rightarrow W[[t]]$ or $W[[x,y,z]]\rightarrow W[[t]]$ which lifts $k[[x,y]]\rightarrow k[[t]]$ or $k[[x,y,z]]\rightarrow k[[t]]$ given by $C$\footnote{We may pick a lift $k\rightarrow W$, for instance, the Teichm\"uller lift and hence view $x(t),y(t), z(t)$ as power series in $W[[t]]$.} and the $\sigma_t$-linear Frobenius is given in \cite[\S 4.3.3]{Moonen98}.\footnote{Here we refer to \cite{Moonen98} for the existence of an explicit formula of the $\sigma_t$-linear Frobenius, but we do not need this explicit formula for our purpose. We will always carry out our computation using the $\sigma$-linear Frobenius; see the rest of the proof for the details.} 
%SAVE: we will never use the matrix of this $\sigma_t$-Frobenius. We always use the $\sigma(x)=x^p$ Frobenius; they give the same Frobenius invariant sections since these Frob inv sections are horizontal and being horizontal is compatible with the nature pull-back, i.e., in our case, replace $x,y,z$ in the formula by $x(t), y(t),z(t)$.

The connection on $\bL_{\cris}(W[[x,y]])$ or $\bL_{\cris}(W[[x,y,z]])$ gives rise to a connection on $\bL_{\cris,P}(W)\otimes_W K[[x,y]]\supset \bL_{\cris}(W[[x,y]]$ or $\bL_{\cris,P}(W)\otimes_W K[[x,y,z]]\supset \bL_{\cris}(W[[x,y,z]]$.
Let $\tilde{w}$ denote the horizontal section in $\bL_{\cris,P}(W)\otimes_W K[[x,y]]$ or $\bL_{\cris,P}(W)\otimes_W K[[x,y,z]]$ extending $w\in L'\otimes \bZ_p\subset \bL_{\cris,P}(W)$. Since the image of $\tilde{w}_m$ in $\iota_m^* \bL_{\cris}(D_m)$ is horizontal and the connection on $\iota_m^* \bL_{\cris}(D_m)$ is the pull-back connection, then $\tilde{w}_m=\iota_m^* \tilde{w}$. Therefore, if $w$ lift to a special endomorphism in $S_m$, then $\iota_m^* \tilde{w} \in \iota_m^* \bL_{\cris}(D_m)\subset \bL_{\cris,P}(W)\otimes_W K[[t]]$.

The section $\tilde{w}$ is constructed in \cite[\S 1.5.5]{Kisin} as follows. Recall from \S\S\ref{sec_F_Hil}-\ref{sec_F_Sie}, the Frobenius on $\bL_{\cris}(W[[x,y]]), \bL_{\cris}(W[[x,y,z]])$, with respect to a $\varphi$-invariant basis $\{w_i\}$, is given by $(I+F)\circ \sigma$ for some matrix $F$ with entries in $(x,y)K[x,y]$ or $(x,y,z)K[x,y,z]$.
We define $F_{\infty}$ to be the infinite product $\displaystyle \prod_{i = 0}^{\infty} (1 + F^{(i)})$, where $F^{(i)}$ is the $i$-th $\sigma$-twist of $F$ (recall $\sigma(x)=x^p, \sigma(y)=y^p, \sigma(z)=z^p$). Since $v_t(y), v_t(x), v_t(z)\geq 1$, the product is well-defined and the entries of $\Finf$ are power series valued in $K[[t]]$. The $\bQ_p$-span of the columns of $\Finf$ are  vectors of $\bL_{\cris,P}(W)\otimes K[[x,y]], \bL_{\cris,P}(W)\otimes K[[x,y,z]]$ which are Frobenius stable and horizontal.

Now we are ready to reduce to the proof of the decay lemma to the following proposition. Indeed, by \Cref{prop_decay}, with respect to $\{w_i\}$, there exists a rank $3$ $\bZ_p$-submodule of $L'\otimes \bZ_p$ such that for every primitive $w$ in this submodule, the coefficient of $t^{k_n}$ for some $k_n\leq A(1+p+\cdots+p^{n+1})$ in $p^n \tilde{w}$ does not lie in $(p^{-1}W)^4$; since $p\bL_{\cris,P}(W)\subset L'\otimes W$, with respect to a $W$-basis of $\bL_{\cris,P}(W)$, the coefficient of $t^{k_n}$ in $p^n \tilde{w}$ does not lie in $W^4$. On the other hand, for any $N< p(A_n+1)$, we have $p^{-1}t^N\notin D_{A_n+1}$. Note that $p(A_n+1)>p A(p^n+\cdots + 1/p)=A(p^{n+1}+\cdots +1)\geq k_n$. Hence $p^n \tilde{w}$ does not extend to a special endomorphism over $S_{A_n+1}$. Thus, this rank $3$ submodule decays rapidly. Moreover,
the existence of a vector decaying very rapidly follows by the second assertion of \Cref{prop_decay} via the same argument and the fact that $p(A_{n-1}+ap^n+1)>p( A(p^{n-1}+\cdots + 1/p)+ap^n)=A(p^{n}+\cdots +1)+ap^{n+1}$.
\end{proof}
%SAVE: There are two different integral structures in play here (in the notation of \S\ref{sec_F_Hil}, \ref{sec_F_Sie), one arises from the basis $v_i$ and the other from the basis $w_i$). The $\bF$-vector space $V_{\cris}/p$ has a filtration, with $\Fil^1$  spanned by $v_2$, $\Fil^0$ spanned by $v_2,v_3,v_4$ and $\Fil^{-1} = V_{\cris}/p$. It might be very tempting to argue that $p^nw$ mod $t^m$ gives rise to a mod $t^m$ special endomorphism exactly when $p^nw$ is integral mod $t^m$ (with respect to the basis $v_i$), and the image of $p^m w$ in $V_{\cris}$ mod $p$ lies in $\Fil^0$ and hence we could argue integrality by using the $w_i$-basis. However, if we try to use the Frobenius stable condition of special endomorphism to deduce the compatibility with filtration, see for instance Prop 4.10 (ii) of \cite{MP16}, we note there is a $\sigma^*$, which will give not the desired power $t^m$. In short, the relation between Frobenius and filtration is not going to give new things. The actual saving from applying de Jong's theory and work with PD enveloping algebra.

\begin{proposition}\label{prop_decay}
Assume $P$ is superspecial.
With respect to the $w_i$-basis in \S\S\ref{sec_F_Hil}-\ref{sec_F_Sie}, there exists a rank $3$ $\bZ_p$-submodule of $L'\otimes \bZ_p$ such that for every primitive $w$ in this submodule, the coefficients of $1=t^0, \dots,t^{A(1+p+\cdots+p^n)}$ in the power series $p^n\tilde{w}\in (K[[t]])^4$ do not all lie in $W^4$ for all $n\in \bZ_{\geq 0}$ (property DR); moreover, there exist $a\leq A/2$ and a primitive $w$ in the rank $3$ submodule such that the coefficients of $1,\dots, t^{A(1+p+\cdots+p^{n-1})+ap^n}$ in $p^n\tilde{w}\in (K[[t]])^4$ do not all lie in $W^4$ for all $n\in \bZ_{\geq 0}$ (property DvR).
\end{proposition}
%SAVE: We indeed prove that with respect to the $\varphi$-invariant basis $w_i$, we have $p^n \tilde{w}$ is not integral over $W[t]/t^{A(1+p+\cdots + p^n)+1}$.

By a slight abuse of terminology, if a submodule of $L'\otimes \bZ_p$ satisfies the property DR (with respect to basis $\{w_i\}$), we also say that this submodule \emph{decays rapidly}; if a primitive vector satisfies property DvR, we also say that this vector \emph{decays very rapidly}. By the proof of \Cref{thm_decay} above, property DR (resp. DvR) implies decaying (resp. very) rapidly in the sense of \Cref{def_decay}.

The rest of this section is devoted to prove \Cref{prop_decay} for the Hilbert case and its proof for the Siegel case is given in \S\ref{sec_decay_Sie}. In the following, the split/inert case means that $p$ is split/inert in the real quadratic field attached to the Hilbert modular surface.

In the Hilbert case, by \Cref{eqn-non-ord}, the non-ordinary locus is cut out by the equation $xy = 0$. As in the proof of reducing \Cref{thm_decay} to \Cref{prop_decay}, we pick a lift $W[[x,y]]\rightarrow W[[t]]$ of of the local ring homomorphism $k[[x,y]]\rightarrow k[[t]]$ defined by $C$. Since $C$ is generically ordinary, we have that both $x$ and $y$ map to non-zero power series in $W[[t]]$. Without loss of generality, we assume that $v_t(x) \leq v_t(y)$, and that $x(t) = t^a + \hdots$ and $y(t) = \alpha t^b + \hdots$, where $\alpha \in W^\times$.

\subsection{Decay in the split case}\label{sec_decayHsplit}
Notation as in the proof of \Cref{thm_decay}. We first compute $\displaystyle F_\infty=\prod_{i = 0}^{\infty} (1 + F^{(i)})$, where by \eqref{F-Hil-sp},
\[F=\begin{bmatrix}
\frac{xy}{2p}&-\frac{\lambda xy}{2p}&\frac{x+y}{2p}&\frac{-\lambda(x -y)}{2p}\\
\frac{xy}{2\lambda p}&- \frac{xy}{2p}&\frac{x+y}{2\lambda p}&\frac{-(x-y)}{2p}\\
\frac{x+y}{2}&\frac{-\lambda (x + y)}{2}&0&0\\
\frac{x-y}{2\lambda}&\frac{-(x - y)}{2}&0&0\\
\end{bmatrix}.\]

Let $\Finf(1)$ and $\Finf(2)$ denote the top-left and top-right $2 \times 2$ blocks of $\Finf$ respectively. 
To simplify the notation, define\footnote{These three matrices are the same; however, we use different notations to be consistent with the proof for the Siegel case in \S\ref{sec_decay_Sie}.} 
\[G = \begin{bmatrix}
\frac{1}{2}&\frac{-\lambda}{2}\\
\frac{1}{2\lambda}&\frac{-1}{2}
\end{bmatrix},
H_u = \begin{bmatrix}
\frac{1}{2}&\frac{-\lambda}{2}\\
\frac{1}{2\lambda}&\frac{-1}{2}
\end{bmatrix},
H_l = \begin{bmatrix}
\frac{1}{2}&\frac{-\lambda}{2}\\
\frac{1}{2\lambda}&\frac{-1}{2}
\end{bmatrix},
\] %\yunqing{This change of $H_l$ is to correct a mistake that I made in Dec 2018 when computing the matrices.}
and let $F_t$, $F_u$ and $F_l$ denote the top-left, top-right, and bottom-left $2 \times 2$ blocks of $F$. 
The following elementary lemma picks out the terms in $\Finf(1), \Finf(2)$ with the desired $p$-power on the denominators.
\begin{lemma}\label{niceparamreal}
\begin{enumerate}
%\yunqing{pick a more precise name for "part".}
\item The part of $\Finf(1)$ with $p$-adic valuation $-(n+1)$ consists of sums of products of the form $\displaystyle \prod_{i=0}^{m_1 + 2m_2 } X_i^{(n_i)}$. Here $X_i$ is either
$F_t$, $F_u$ or $F_l$,\footnote{The terms $X_i$ are chosen so that the product makes sense, and has the right size. Note that this would imply that $F_u,F_l$ must occur in consecutive pairs.} $m_1 + 1$ is the number of occurrences of $F_t$, and $m_2$ is the number of occurrences of the pair $F_u,F_l$, $m_1 + m_2 = n $, and $n_i$ is a strictly increasing sequence of non-negative integers. The analogous statement holds for $\Finf(2)$ as well. 

\item Fix values of $m_1,m_2$ as above. Among all the terms in the above sum, the ones with minimal $t$-adic valuation only occur when $n_i = i$, and either when $X_0 = X_1 = \hdots = X_{m_1} = F_t$, or $X_0 = X_2 = \hdots = X_{2m_2-2} = F_u$. The analogous statement holds for $\Finf(2)$ as well. 
%SAVE: in the later more careful analysis, we see that the second case will not happen and hence the rest of the lemma only computes the form for the first case. Indeed, for the second case to happen, we need to have $v_t(xy)$ much smaller than $v_t(x^{p+1}+y^{p+1}$ (smaller $t$-adic valuation term shall have larger twist power); however, having $F_u,F_l$ will give larger twist power, so if $v_t(xy)$ is small, we just need to only work with $F_t$.

\item (for $\Finf(1)$) The product $\displaystyle \prod_{i =0}^{m_1}F_t^{(i)} \prod_{i = 0}^{m_2-1} F_u^{(m_1 + 2i + 1)}F_l^{(m_1 + 2i + 2)}$ (modulo terms with smaller $p$-power in denominators\footnote{We use here that $x^p\pm y^p\equiv (x\pm y)^p \mod p$.}) equals

$$\displaystyle \frac{1}{p^{n+1}}\prod_{i =0}^{m_1}G^{(i)}(xy)^{(i)} \prod_{i = 0}^{m_2 -1} H_u^{(m_1 + 2i + 1)}H_l^{(m_1 + 2i+2)}(x^{1+p} + y^{1+p})^{(m_1 + 2i + 1)}.$$ 

\item (for $\Finf(2)$) The product $\displaystyle \prod_{i =0}^{m_1}F_t^{(i)} \prod_{i = 0}^{m_2-1} F_u^{(m_1  + 2i + 1)}F_l^{(m_1 + 2i + 2)} \cdot F_u^{(m_1 + 2m_2 + 1)}$ (modulo terms with smaller $p$-power in denominators) equals 
$$\displaystyle \frac{1}{p^{n+2}}\prod_{i =0}^{m_1}G^{(i)}(xy)^{(i)} \prod_{i = 0}^{m_2 - 1} H_u^{(m_1 + 2i + 1)}H_l^{(m_1 + 2i+2)}(x^{1+p} + y^{1+p} )^{(m_1 + 2i + 1)} \cdot F_u^{(m_1 + 2m_2 + 1)}$$

\end{enumerate}
\end{lemma}

\begin{para} \textbf{Notations.} We make the following definition to further lighten the notation. 

Let $P(1)_{m_2,n}$ denote the product
$$\displaystyle \prod_{i=0}^{m_1} G^{(i)}\prod_{i=0}^{m_2-1}H_u^{(m_1 + 2i + 1)}H_l^{(m_1 + 2i + 2)}.$$
Recall that $A = a + b$ denotes the $t$-adic valuation $v_t(xy)$ of $xy$ and let $B$ denote $v_t(x^{p+1} + y^{p+1})$. Note that $B \geq a(p+1)$ and the equality holds unless $a = b$.
\end{para}
 In order to prove \Cref{prop_decay}, we will consider the following case-by-case analysis depending on the relation between $a$ and $b$. The following elementary lemmas will be used in the case-by-case analysis.

\begin{lemma}\label{reallineartopleft}
Let $n, e, f$ be in $\bZ_{\geq 0}$. 
\begin{enumerate}
\item The kernel of the $2 \times 2$ matrix $P(1)_{e,n}$ modulo $p$ is defined over $\bF_{p^2}$ but not over $\bF_p$. 
\item The reductions of $P(1)_{e,n}$ and $P(1)_{f,n}$ modulo $p$ are not scalar multiples (over $k$) of each other if $e \not\equiv f \mod 2$. In particular, these reductions are not scalar multiples of each other if $f = e \pm 1$. 
\end{enumerate}
\end{lemma}
\begin{proof}
As the entries of $G$, $H_u$ and $H_l$ are all in $W(\bF_{p^2})[1/p]$, it follows that $G^{(2m)} = G$ and $G^{(2m+1)} = G^{(1)}$ (and the analogous statements hold for $H_u$ and $H_l$). A direct computation shows that $GG^{(1)}G = G$, $H_uH_l^{(1)}H_uH_l^{(1)} = H_uH_l^{(1)}$, and $H_u^{(1)}H_lH_u^{(1)}H_l= H_u^{(1)}H_l$. Therefore, if $n-e$ is odd, then $P(1)_{e,n}$ simplifies to either $GG^{(1)}H_uH_l^{(1)}$, $GG^{(1)}$ or $H_uH_l^{(1)}$; if $n-e$ is even, $P(1)_{e,n}$ simplifies to $G$ or $GH_u^{(1)}H_l$. A direct computation shows that the matrices $GG^{(1)}$, $H_uH_l^{(1)}$ and  $GG^{(1)}H_uH_l^{(1)}$ (resp. $G$ and $GH_u^{(1)}H_l$) are equal to 
\[\begin{bmatrix}
\frac{1}{2}&\frac{\lambda}{2}\\
\frac{1}{2\lambda}&\frac{1}{2}\\
\end{bmatrix}\,\left(\text{resp. }\begin{bmatrix}
\frac{1}{2}&\frac{-\lambda}{2}\\
\frac{1}{2\lambda}&\frac{-1}{2}\\
\end{bmatrix}\right).\]

%In every case, it is easy to see that the product satisfies the property that no non-zero $\bF_p$-linear combination of the columns mod $p$ is zero. We illustrate this in the following two cases. 

In either case, since $\lambda\in W(\bF_{p^2})\backslash \bZ_p$, there is no non-trivial $\bF_p$-linear combination of the columns modulo $p$ which equals zero; this implies part (1). Furthermore, the above matrices are clearly not scalar multiples of each other, whence part (2) follows. 
\end{proof}

\begin{lemma}\label{reallineartopright}
Let $n, e, f$ be in $\bZ_{\geq 0}$. 
\begin{enumerate}
\item The kernel of the $2 \times 2$ matrix $P(1)_{e,n-1}\cdot H_u^{(n+e)}$ modulo $p$ is defined over $\bF_{p^2}$ but not $\bF_p$. 
\item The reductions of $P(1)_{e,n-1} \cdot H_u^{(n+e)}$ and $P(1)_{f,n-1}\cdot H_u^{(n+f)}$ modulo $p$ are not scalar multiples of each other if $e\not\equiv f \mod 2$. In particular, these reductions are not scalar multiples of each other if $f = e \pm 1$. 
\end{enumerate}
\end{lemma}

\begin{proof}
We argue along the lines of the proof of Lemma \ref{reallineartopleft}. Indeed, if $n-e$ is odd (resp. even), we are reduced to the cases of $GG^{(1)}H_uH_l^{(1)}H_u$, $GG^{(1)}H_u$, $H_uH_l^{(1)}H_u$, and $H_u$ (resp. $GH_u^{(1)}H_lH_u^{(1)}$ and $GH_u^{(1)}$). The rest of the argument is similar. %Again, every case has the property that no non-zero $\bF_p$-linear combination of the columns mod $p$ is zero. The second assertion similarly follows. 
\end{proof}

We now prove \Cref{prop_decay} when $p$ is split in the real quadratic field defining the Hilbert modular surface. The proof is a case-by-case study in the following four cases based on the relation of $a=v_t(x)$ and $b=v_t(y)$. The idea is to pick out the term(s) with minimal $t$-adic valuation among all the terms with the same $p$-power denominators given in \Cref{niceparamreal}. Case~4 is the generic case and it is easy to pick out such terms so we give the proof directly. In Cases~1-3, we first state the lemmas on the terms with minimal $t$-adic valuation and then prove the decay lemma. For the convenience of the reader, we summarize the desired vectors which decay rapidly enough at the beginning of each case.

\subsection*{Case 1: $a = b$.}
Recall that $A=v_t(xy)=a+b=2a$. 

We will prove that every vector in $\Span_{\bZ_p}\{w_1,w_2,w_i\}$ decays rapidly, where $w_i = w_4$ if the $t$-adic valuation of $x - y$ is $> a$, and $w_i = w_3$ otherwise. Moreover, $w_i$, $i=3,4$ respectively, decays very rapidly.

\begin{lemma}\label{realdecay}
\begin{enumerate}
\item Among the terms appearing in $\Finf(1)$ described in \Cref{niceparamreal} with denominator $p^{n+1}$, the unique term with minimal $t$-adic valuation is  
$$P(1)_{0,n}(xy)^{1 + p + \hdots + p^{n}}.$$

\item Among the terms appearing in $\Finf(2)$ described in \Cref{niceparamreal} with denominator $p^{n+1}$, the unique term with minimal $t$-adic valuation is 
$$P(1)_{0,n-1}\cdot F_u^{(n )} (xy)^{1 + p + \hdots + p^{n-1}}.$$ 
\end{enumerate}
\end{lemma}
This lemma follows directly from \Cref{niceparamreal} and the assumption that $a = b$.

\begin{proof}[Proof of \Cref{prop_decay} in this case]
We first prove that every primitive vector $w\in \Span_{\bZ_p} \{w_1,w_2\}$ decays rapidly.
Indeed, write $w=cw_1+dw_2$, by \Cref{reallineartopleft}(1) and \Cref{realdecay}(1), there is a unique (non-vanishing) term in $\Finf(1)w$ with denominator $1/p^{n+1}$ and minimal $t$-adic valuation $A(1+p+\cdots + p^n)$ given by $P(1)_{0,n}[c\quad d]^T (xy)^{1+p+\cdots +p^n}$. Hence, modulo $t^{A(1+p+\cdots + p^n)+1}$, the horizontal section $p^n\tilde{w}=\Finf(p^nw)$ does not lie in $W[[t]]$ and hence $w$ decays rapidly.

Secondly, let $i\in \{3,4\}$ be defined as above and we show that $w_i$ decays \emph{very} rapidly. Note that our definition of $w_i$ implies that the first two entries of the $i^{th}$ row of $F$ have $t$-adic valuation equalling $a$. Furthermore, by \Cref{reallineartopleft}(1), $P(1)_{0,n-1} \cdot v \neq 0$ mod $p$, where $v$ is the $n^{th}$ Frobenius twist of either column of $H_u$. Therefore, among the terms in the $i^{th}$ column of $\Finf$ with denominator $p^{n+1}$, the term with minimal $t$-adic valuation has $t$-adic valuation $2a(1 + p + \hdots + p^{n-1}) + ap^n$. Hence $w_i$ decays very rapidly since $a\leq (2a)/2=A/2$.

Finally, we show that every vector in $\Span_{\bZ_p}\{w_1,w_2,w_i\}$ decays rapidly. Let $w_u$ denote a primitive vector in the span of $w_1,w_2$. It suffices to show that every vector which either has the form $p^m w_u + w_i$ or $w_u + p^m w_i$ decays rapidly, where $m \geq 0$. We first prove that every vector which has the form $p^m w_u + w_i$ decays rapidly where $m \geq 0$. Indeed, consider the two-dimensional vector whose entries are the first two entries of $\Finf \cdot p^m w_u$. The $t$-adic valuation of the coefficient of  $1/p^{n+1}$ equals $2a(1 + p + \hdots + p^{n+m})$. Similarly, consider the two-dimensional vector whose entries are the first two entries of $\Finf \cdot w_i$. The $t$-adic valuation of the coefficient of $1/p^{n+1}$ equals $2a(1 + p + \hdots + p^{n-1}) + ap^n$. Regardless of the value of $m$, the latter quantity is always smaller than the former quantity, whence it follows that $p^mw_u + w$ decays rapidly. Now, consider a vector of the form $w_u + p^m w_i$, where $m > 0$. Analogous to the previous case, consider the two-dimensional vector whose entries are the first two entries of $\Finf \cdot w_u$. The $t$-adic valuation of the sum of all terms with denominator $p^{n+1}$ equals $2a(1 + p + \hdots + p^n)$. Similarly, consider the two-dimensional vector whose entries are the first two entries of $\Finf \cdot p^m w_i$. The $t$-adic valuation of the coefficient of $1/p^{n+1}$ equals $2a(1 + p + \hdots + p^{n+m-1}) + ap^{n+m}$. Regardless of the value of $m$ (recall that $m>0$), the latter quantity is always greater than the former quantity, whence it follows that $p^mw_u + w$ decays rapidly.
\end{proof}

\subsection*{Case 2: $b = p^{2e}a$ for some $e\in \bZ_{\geq 1}$}
We will prove that $\Span_{\bZ_p}\{w_1,w_2,w\}$ decays rapidly where $w$ is some primitive vector in $\Span_{\bZ_p}\{w_3,w_4\}$. We will further prove that $w$ decays \emph{very} rapidly. 

\begin{lemma}\label{realdecaytopleft}
\begin{enumerate}
\item Among the terms appearing in $\Finf(1)$ described in \Cref{niceparamreal} with denominator $p^{n+1}$, the unique term with minimal $t$-adic valuation is  
$$P(1)_{e,n}(xy)^{1 + p + \hdots + p^{n-e}}x^{p^{n - e + 1} + p^{n-e+2} + \hdots + p^{n + e}}.$$

\item Among the terms appearing in $\Finf(2)$ described in \Cref{niceparamreal} with denominator $p^{n+1}$, there are exactly two terms with minimal $t$-adic valuation, and they are 
$$P(1)_{e,n-1}\cdot F_u^{(n + e -1)} (xy)^{1 + p + \hdots + p^{n-e-1}}x^{p^{n - e}  + p^{n-e+1} + \hdots + p^{n + e-2}}, \text{ and}$$
$$P(1)_{e+1,n-1}\cdot F_u^{(n + e)} (xy)^{1 + p + \hdots + p^{n-e-2}}x^{p^{n - e-1}  + p^{n-e} + \hdots + p^{n + e-1}}.$$

\end{enumerate}
\end{lemma}
\begin{proof}
In the following, we will prove part (1); part (2) will follow by an identical argument. 

Note that the $t$-adic valuation of all the entries of $F(1)$ is $a+b$, and the $t$-adic valuation of the entries of $F_u$ and $F_l$ is $ a$ . Let $k,l$ be in $\bZ_{\geq 0}$ such that $k + l = n+1$. Consider the following terms of $F_{\infty}(1)$ with denominator exactly $p^{n+1}$: 
\[X_{k,l}\colonequals F(1)\cdot F(1)^{(1)} \hdots \cdot F(1)^{(k-1)}\cdot F_u^{(k)}F_{l}^{(k+1)}\hdots F_u^{(k+2l-2)}F_l^{(k+2l-1)}.\]

Similar to \Cref{niceparamreal}(2), we observe that among all the terms of $F_{\infty}(1)$ with denominator exactly $p^{n+1}$ given in \Cref{niceparamreal}(1), for each other term $X$ not listed above, there exists at least one $X_{k,l}$ (as $k$ and $l$ vary over all non-negative integers constrained by $k + l = n+1$) such that $v_t(X_{k,l})<v_t(X)$.
Therefore, to prove (1), it suffices to show that $v_t(X_{k,l})$ with $k = n - e + 1$ and $l = e$ is less than $v_t(X_{k,l})$ with any other choice of $k,l$.

Since $b = ap^{2e}$ and $k + l = n+1$, then $f(k) \colonequals v_t(X_{k,n}) = a\left((1 + p^{2e})\frac{p^k-1}{p-1} + \frac{p^{2(n-k+1)} -1}{p-1}p^k \right)$, and we need to prove that $k = n-e +1$ minimizes this expression as $k$ ranges over $\bZ \cap [0,n+1]$. Note that if we allow $k$ to take all real values in the interval $[0,n+1]$, a direct computation shows that $f$ is convex (i.e., $f''(k)>0$). Therefore, it suffices to show that $f(n-e+1)<f(n-e)$ and $f(n-e+1) < f(n-e+2)$. These claims can be verified directly and hence we prove (1). %it boils down to $2p^{n+1+e}<p^{n+e}+p^{n+2+e}$.
\end{proof}

\begin{proof}[Proof of \Cref{prop_decay} in this case]
We first prove that $\Span_{\bZ_p}\{w_1,w_2\}$ decays rapidly.
Indeed, let $w'$ be a primitive vector in $\Span_{\bZ_p}\{w_1,w_2\}$. \Cref{reallineartopleft}(1) implies that $P(1)_{e,n} \cdot w'$ mod $p$ is non-zero. This fact taken in conjunction with \Cref{realdecaytopleft}(1) yields that $w'$ decays rapidly. 

Secondly, we prove that there exists a primitive vector $w\in \Span_{\bZ_p}\{w_3,w_4\}$ (independent of $n$) which decays very rapidly.
Set $Y_{e,n}\colonequals P(1)_{e,n-1}\cdot F_u^{(n + e -1)} (xy)^{1 + p + \hdots + p^{n-e-1}}x^{p^{n - e}  + p^{n-e+1} + \hdots + p^{n + e-2}}+ P(1)_{e+1,n-1}\cdot F_u^{(n + e)} (xy)^{1 + p + \hdots + p^{n-e-2}}x^{p^{n - e-1}  + p^{n-e} + \hdots + p^{n + e-1}}$, which is the sum of the two terms with minimal $t$-adic valuation listed in \Cref{realdecaytopleft}(2). The sum $Y_{e,n}$ is non-zero modulo $p$ by \Cref{reallineartopleft}(2). Furthermore, up to Frobenius twists and multiplication by scalars, the matrix $Y_{e,n}\, \bmod p$ is independent of $n$. Therefore, there exists a vector $w\in \Span_{\bZ_p}\{w_3,w_4\}$ which is independent of $n$ and does not lie in the kernel of $Y_{e,n}\,\bmod p$. The very rapid decay of $w$ follows from this observation and \Cref{realdecaytopleft}(2). 

Finally, a valuation-theoretic argument analogous to Case 1 shows that every primitive vector in $\Span_{\bZ_p}\{w_1,w_2,w\}$ decays rapidly, thereby establishing \Cref{prop_decay} in this case. 
\end{proof}

\subsection*{Case 3: $b = p^{2e + 1} a$ for some $e\in \bZ_{\geq 0}$}
We will prove that $\Span_{\bZ_p}\{w_3,w_4,w\}$ decays rapidly where $w$ is some primitive vector in $\Span_{\bZ_p}\{w_1,w_2\}$ and that $\Span_{\bZ_p}\{w_3,w_4\}$ decays \emph{very} rapidly. 

\begin{lemma}\label{realdecaytopright}
\begin{enumerate}
\item Among the terms appearing in $\Finf(2)$ described in \Cref{niceparamreal} with denominator $p^{n+1}$, the unique term with minimal $t$-adic valuation is  
$$P(1)_{e,n-1}\cdot H_u^{(n+e)}(xy)^{1 + p + \hdots + p^{n-e-1}}x^{p^{n - e} + p^{n-e+1} + \hdots + p^{n + e}}.$$

\item Among the terms appearing in $\Finf(1)$ described in \Cref{niceparamreal} with denominator $p^{n+1}$, there are exactly two terms with minimal $t$-adic valuation, and they are 
$$P(1)_{e,n} (xy)^{1 + p + \hdots + p^{n-e-1}}x^{p^{n - e}  + p^{n-e+1} + \hdots + p^{n + e-1}},\text{ and}$$
$$P(1)_{e+1,n} (xy)^{1 + p + \hdots + p^{n-e-2}}x^{p^{n - e-1}  + p^{n-e} + \hdots + p^{n + e}}.$$

\end{enumerate}
%The $1/p^{n+1}$ of $F_{\infty}(2)$ is 
%$$G G^{(1)}  \hdots\cdot G^{(n-e-1)} \cdot H_u^{(n-e})H_l^{(n-e+1)} \hdots \cdot H_l^{(n+e-1)} H_u^{(n + e)} t^{(a + b)(1 + p + \hdots p^{n-e-1})}t^{a(p^{n - e  + \hdots p^{n + e}})}.$$
\end{lemma}
\begin{proof}
The proof of this lemma is identical to that of \Cref{realdecaytopleft}, so we omit the details. 
\end{proof}

\begin{proof}[Proof of \Cref{prop_decay} in this case]
Analogous to Case 2, \Cref{reallineartopright} and \Cref{realdecaytopright}(2) imply the existence of a primitive $w\in \Span_{\bZ_p}\{w_1,w_2\}$ that decays rapidly; and by \Cref{reallineartopright}(1) and \Cref{realdecaytopright}(1), $\Span_{\bZ_p}\{w_3,w_4\}$ decays very rapidly. Finally, a valuation-theoretic argument shows that every primitive vector in $\Span_{\bZ_p}\{w,w_3,w_4\}$ decays rapidly. 
\end{proof}

\subsection*{Case 4: $b \neq a p^e$ for any value of $e$}

\begin{proof}[Proof of \Cref{prop_decay}]
As this is the easiest case, we will be content with merely sketching a proof. Analogous to \Cref{realdecaytopleft,realdecaytopright}, it is easy to see that in this case there are unique terms with minimal $t$-adic valuations with denominator $p^{n+1}$ occurring in both $\Finf(1)$ and $\Finf(2)$. It follows that every primitive vector in $\Span_{\bZ_p}\{w_1,w_2\}$ decays rapidly and every vector in $\Span_{\bZ_p}\{w_3,w_4\}$ decays very rapidly. Finally, a valuation theoretic argument similar to Case 1 shows that every vector in the span of $w_1,w_2,w_3,w_4$ does decay rapidly, finishing the proof of \Cref{prop_decay}. 
\end{proof}

\subsection{Decay in the inert case}
Notation as in the proof of \Cref{thm_decay} and \S\ref{Frob-H-inert}. 
Recall that $P$ is superspecial and we will show that The $\bZ_p$-span of $w_1,w_2,w_3$ decays rapidly, and the vector $w_3$ decays very rapidly. 

\begin{proof}[Proof of \Cref{prop_decay}]
The proof goes along the same lines as the proof of the decay lemma for split Hilbert modular varieties, so we will be content with just outlining the salient points. 

We first compute $\displaystyle F_\infty=\prod_{i = 0}^{\infty} (1 + F^{(i)})$, where by \eqref{FrobHinert}, with respect to the basis $\{w_1,w_2,w_3,w_4 \}$, $F=\begin{pmatrix}
F_t & F_u\\
F_l & 0 \end{pmatrix}$, where 
\[
F_t=\frac{xy}{2p}\begin{pmatrix} -1 & \lambda \\
-1/\lambda & 1 \end{pmatrix}, F_u=\frac 1{2p} \begin{pmatrix} x & y\\
x/\lambda & y/\lambda \end{pmatrix}, F_l=\begin{pmatrix} -y & \lambda y\\
-x & \lambda x \end{pmatrix}.
\]
Recall that the non-ordinary locus is cut out by the equation $xy = 0$ and $a = v_t(x), b = v_t(y)\in \bZ_{>0}$.

Similar to \Cref{niceparamreal}, it is easy to see that the top-left $2\times 2$ block of $\Finf$ with $p$-adic valuation $-(n+1)$ has a term of the form $ F_t F_t^{(1)} \hdots F_t^{(n)}$, and this term is the unique term with minimal $t$-adic valuation (equalling $(a + b)(1 + p + \hdots p^n)$). Similarly, the top-right $2\times 2$ block of $\Finf$ with $p$-adic valuation $-(n+1)$ has a term of the form $ F_t F_t^{(1)} \hdots F_t^{(n-1)} F_u^{(n)}$, and this term is the unique term with minimal $t$-adic valuation (equaling $(a + b)(1 + p + \hdots p^{n-1}) + ap^n$). 

Arguments identical to \Cref{reallineartopleft} and \Cref{reallineartopright} yield that every primitive vector in the $\bZ_p$ span of $w_1,w_2$ (and in the span of $w_3$) decays rapidly (very rapidly, in the case of $w_3$). Further, as the $t$-adic valuation of $ F_t F_t^{(1)} \hdots F_t^{(m)}$ is different from the $t$-adic valuation of $ F_t F_t^{(1)} \hdots F_t^{(n-1)} F_u^{(n)}$ for every pair of integers $n,m$, it follows that $\Span_{\bZ_p}\{w_1,w_2, w_3\}$ also decays rapidly. The argument is elaborated on in the last paragraph of the proof for Case 1 in \S\ref{sec_decayHsplit}.
\end{proof}

%%%%%%%%%%%%%%%%%%%%%%%%%%%%%%%%%%%%%%%%%%%%%Section 5%%%%%%%%%%%%%%%%%%%%%%%%%%%%%%%%%%%%%%%%%%%%%%%%%%%%%%%%%%%%%%%%%%%%%%%

\section{Proof of the decay lemma in the Siegel case}\label{sec_decay_Sie}
In this section, we prove \Cref{prop_decay} and hence \Cref{thm_decay} (for superspecial points) in the Siegel case. We refer the reader to the appendix for a decay lemma for supergeneric points. The main idea of the proof is similar to that of the Hilbert case in \S\ref{sec_decay_Hil}.

\subsection{Preparation of the proof}
We follow the notation in \S\ref{sec_decay_Hil}, $k=\bar{\bF}_p$, $W=W(k)$, $K=W[1/p]$, $\lambda \in \bZ_{p^2}^\times$ such that $\sigma(\lambda) = - \lambda$, and $C = \Spf k[[t]]$ a generically ordinary formal curve in $\cM_k$ which specializes to a superspecial point $P$.  This gives rise to a local ring homomorphism $k[[x,y,z]]\rightarrow k[[t]]$ and we pick a lift $W[[x,y,z]]\rightarrow W[[t]]$ (still a ring homomorphism), and we denote by $x(t),y(t)$ and $z(t)$ the images of $x,y,z$ respectively.

Let $a,b,c$ denote the $t$-adic valuations of $x(t), y(t)$ and $z(t)$ respectively. We adopt the convention that $a,b,c$ may take on the value $\infty$ if the corresponding power series is $0$. As before, $v_t$ denotes the $t$-adic valuation map on $K[[t]]$ or $k[[t]]$.

Also recall that $\sigma$ denotes both the Frobenius on $K$ and the Frobenius on the coordinate rings $W[[x,y,z]]$ with $\sigma(x)=x^p, \sigma(y)=y^p, \sigma(z)=z^p$; and for a matrix $M$ with entries in $K[[x,y,z]]$, $M^{(n)}$ denotes $\sigma^n(M)$.

The preparation lemmas of the Siegel case are very similar to that of the split Hilbert case in the beginning of \S\ref{sec_decayHsplit}.

\begin{para}\textbf{Notations.}
Recall that $\displaystyle F_{\infty}=\prod_{i = 0}^{\infty} (1 + F^{(i)})$, where by \eqref{F-Sie}, with respect to the basis $\{w_1,\cdots, w_5\}$,
\[F=\begin{bmatrix}
\frac{1}{2p}(xy+\frac{z^2}{4\epsilon}) & -\frac{1}{2\lambda p}(xy+\frac{z^2}{4\epsilon}) & \frac{x}{2\lambda p}& \frac{y}{2\lambda p} & \frac{z}{2\lambda p}\\
\frac{\lambda}{2p}(xy+\frac{z^2}{4\epsilon}) & -\frac{1}{2p}(xy+\frac{z^2}{4\epsilon}) & \frac{x}{2p} & \frac{y}{2p} & \frac{z}{2p} \\
\lambda y & -y & 0 & 0 & 0\\
\lambda x & -x& 0 & 0 & 0\\
\frac{\lambda z}{2\epsilon}  & -\frac{ z}{2\epsilon} & 0 & 0 & 0
\end{bmatrix},\]
where $\epsilon=\lambda^2\in \bZ_p^\times$.
We denote by $F_t$, $F_u$, and $F_l$ the top-left $2 \times 2$ block, the top-right $2 \times 3$ block, and the bottom-left $3 \times 2$ block of $F$ respectively. 
Define 

\[G = \begin{bmatrix}
\frac{1}{2}&\frac{-1}{2 \lambda}\\
\frac{\lambda}{2}&\frac{-1}{2}\\
\end{bmatrix}, 
H_u = \begin{bmatrix}
\frac{1}{2\lambda}&\frac{1}{2\lambda}&\frac{1}{2\lambda}\\
\frac{1}{2}&\frac{1}{2}&\frac{1}{2}\\
\end{bmatrix}, \text{ and }
H_l = \begin{bmatrix}
\lambda & -1\\
\lambda & -1\\
\lambda & -1\\
\end{bmatrix}.\]
Let $\Finf(1)$ and $\Finf(2)$ denote the top-left $2 \times 2$ block and top-right $2 \times 3$ of $\Finf$ respectively.

By \Cref{eqn-non-ord}, the non-ordinary locus is cut out by the equation $xy + z^2/(4\epsilon) = 0$.
Let $\eta t^A$ and $\mu t^B $ denote the leading terms of $xy + z^2 /(4\epsilon)$  and $xy^p + x^p y + z^{1 + p}/(2\epsilon)$ respectively. In particular, $A=v_t(xy + z^2/(4 \epsilon))$, and $B=v_t(xy^p + x^p y + z^{1 + p}/(2\epsilon))$. 
\end{para}

The following is analogous to \Cref{niceparamreal}.
\begin{lemma}\label{niceparam}
\begin{enumerate}
\item The part of $\Finf(1)$ with $p$-adic valuation $-(n+1)$ consists of sums of products of the form $\displaystyle \prod_{i=0}^{m_1 + 2m_2 } X_i^{(n_i)}$. Here, $X_i$ is either
 $F_t$, $F_u$ or $F_l$,\footnote{The terms $X_i$ are chosen so that the product makes sense, and has the right size. Note that this would imply that $F_u,F_l$ must occur in consecutive pairs.} $m_1 + 1$ is the number of occurrences of $F_t$, and $m_2$ is the number of occurrences of the pair $F_u,F_l$, $m_1 + m_2 = n $, and $\{n_i\}_{i=0}^{m_1+2m_2}$ is a strictly increasing sequence of non-negative integers. The analogous statement holds for $\Finf(2)$ as well. 

\item Fix values of $m_1,m_2$ as above. Among all the terms in the above sum, the ones with minimal $t$-adic valuation only occur when $n_i = i$ for all $i$, and either when $X_0 = X_1 = \hdots X_{m_1} = F_t$, or $X_0 = X_2 = \hdots = X_{2m_2-2} = F_u$, depending on whether $A \geq B$. The analogous statement holds for $\Finf(2)$ as well. 

\item (for $\Finf(1)$) The product $\displaystyle \prod_{i =0}^{m_1}F_t^{(i)} \prod_{i = 0}^{m_2-1} F_u^{(m_1  + 1 + 2i)}F_l^{(m_1 + 2i + 2)}$ equals

$$\displaystyle \frac{1}{p^{n+1}}\prod_{i =0}^{m_1}G^{(i)}(xy + z^2/2)^{(i)} \prod_{i = 0}^{m_2 -1} \frac{1}{3}H_u^{(m_1 + 2i + 1)}H_l^{(m_1 + 2i+2)}(xy^p + x^py + z^{p+1})^{(m_1 + 2i + 1)}.$$

\item (for $\Finf(2)$) The product $\displaystyle \prod_{i =0}^{m_1}F_t^{(i)} \prod_{i = 0}^{m_2-1} F_u^{(m_1  + 2i + 1)}F_l^{(m_1 + 2i + 2)} \cdot F_u^{(m_1 + 2m_2 + 1)}$ equals 
$$\displaystyle \frac{1}{p^{n+2}}\prod_{i =0}^{m_1}G^{(i)}(xy + z^2/2)^{(i)} \prod_{i = 0}^{m_2-1} \frac{1}{3}H_u^{(m_1 + 2i + 1)}H_l^{(m_1 + 2i+2)}(xy^p + x^py + z^{p+1})^{(m_1 + 2i + 1)} \cdot F_u^{(m_1 + 2m_2 + 1)}$$

\end{enumerate}
\end{lemma}

\begin{para}\textbf{Notation.}
Let $P(1)_{m_2,n}$ denote the product 
$\displaystyle \prod_{i =0}^{m_1}G^{(i)} \prod_{i = 0}^{m_2 -1} \frac{1}{3}H_u^{(m_1 + 2i + 1)}H_l^{(m_1 + 2i+2)}$.
\end{para}

%Fix an integer $n$. Amongst the terms described in \Cref{niceparam}, consider the terms with minimal $t$-adic valuation. Roughly, our strategy to prove the decay lemma will be to exhibit a rank 3 (saturated) $\bZ_p$ submodule of the module of special endomorphisms with the following property: for any primitive vector $w$ of this submodule, the matrix with minimal $t$-adic valuation described above multiplied by $w$ yields a vector with $p$-adic valuation $-(n+1)$. In order to carry out this strategy, we will need to analyze the mod $p$ kernel of matrices of the form $P(1)_{m_2,n}$, which we do in the following lemma: 
The following will play a similar role as \Cref{reallineartopleft}.
\begin{lemma}\label{linear}
The kernel of $P(1)_{g,f+g} \bmod p$ does not contain any non-zero vector defined over $\bF_p.$ Moreover,
if $f$ is odd (resp. even), the kernel of $P(1)_{g,f+g} \bmod p$ does not contain the vector $\begin{bmatrix}
\lambda^{-1}\\
1\\
\end{bmatrix}$ (resp. $\begin{bmatrix}
-\lambda^{-1}\\
1\\
\end{bmatrix}$).
\end{lemma}

\begin{proof}
We prove the assertions by explicit computation as in \Cref{reallineartopleft,reallineartopright}. Note that
\[\frac{1}{3}H_u^{(2m)}H_l^{(2m+1)}=\frac{-1}{2}\begin{bmatrix}
1&\lambda^{-1}\\
\lambda&1\\
\end{bmatrix},\frac{1}{3} H_u^{(2m-1)}H_l^{(2m)}=\frac{1}{2}\begin{bmatrix}
-1& \lambda^{-1}\\
\lambda&-1\\
\end{bmatrix}\]
Both these matrices satisfy the relation $X^2 = -X$ and hence $\displaystyle \prod_{i=0}^{m_2-1}H_u^{(m_1+2i+1)}H_l^{(m_1+2i+2)}$ equals, up to a multiple of $\pm 1$, one of these matrices depending on the parity of $m_1$. 
Similarly, we have
\[G\cdot \cdot \cdot G^{(2m)}=\frac{1}{2}\begin{bmatrix}
1&-\lambda^{-1}\\
\lambda& -1\\
\end{bmatrix},\ G \cdot \cdot \cdot G^{(2m+1)}=\frac{1}{2}\begin{bmatrix}
1&\lambda^{-1}\\
\lambda&1\\
\end{bmatrix}.\]

Therefore, $P(1)_{g,f+g}$ equals 
$\pm \frac{1}{2} \begin{bmatrix}
1&\lambda^{-1}\\
\lambda&1\\
\end{bmatrix}$
if $f$ is odd, and equals  
$\pm \frac{1}{2}\begin{bmatrix}
1&-\lambda^{-1}\\
\lambda&-1\\
\end{bmatrix}$
if $f$ is even. 
The lemma then follows immediately.
\end{proof}

For fixed $n$, among the terms listed in \Cref{niceparam} with denominator $p^{n+1}$, the number of terms with equal minimal $t$-adic valuation depends on certain numerical relation between $A$ and $B$. We then perform the following case-by-case analysis in \S\S\ref{sec_decayS1}-\ref{sec_decayS3} to prove the Decay Lemma. The first case, while technically the easiest, holds the main ideas in general. 

\subsection{Case 1: $A  < B$.}\label{sec_decayS1} \quad

Note that if $a + b \neq 2c$, or more generally, if the leading terms of $xy$ and $z^2/(4\epsilon)$ do not cancel, then $A<B$.

\begin{proof}[Proof of \Cref{prop_decay} in this case]
For the ease of exposition, we assume that $a \leq b \leq c$. Note that this forces $2a \leq A$. Even though the statement of \Cref{prop_decay} is not symmetric in $a,b,c$, an identical argument as the one below suffices to deal with all the other cases. 

We will prove that $\Span_{\bZ_p}\{w_1,w_2,w_3\}$ decays rapidly. 
For a primitive vector $w\in \Span_{\bZ_p}\{w_1,w_2,w_3\}$, 
write $w = \alpha_u w_u + \alpha_l w_3$, where $w_u$ is a primitive vector in $\Span_{\bZ_p}\{w_1,w_2\}$, and $\alpha_u,\alpha_l \in \bZ_p$. Since $w$ is primitive, then either $\alpha_u$ or $\alpha_l$ is a $p$-adic unit. 
We may assume that $\alpha_u$ is a unit -- the other case is entirely analogous to this one. Suppose that the $p$-adic valuation of $\alpha_l$ is $m \geq 0$.  

Consider the terms appearing in $\Finf(1)$ described in \Cref{niceparam} with denominator $p^{n+1}$. As $A < B$, the one with minimal $t$-adic valuation is $P(1)_{0,n}(xy + z^2/(4\epsilon))^{1 + p + \hdots + p^{n}} $, and this is the unique term with this property. Similarly, consider the terms appearing in $\Finf(2)$ with denominator $p^{n+1 + m}$. As $A < B$, the unique term whose first column has minimal $t$-adic valuation is \linebreak $P(1)_{0,n+m-1}\cdot F_u^{(n+m)}(xy + z^2/(4\epsilon))^{1 + p + \hdots + p^{n+m-1}}$. 

Let $P$ denote the $2 \times 3$ matrix whose first two columns equal $P(1)_{0,n}(xy + z^2/(4\epsilon))^{1 + p + \hdots + p^{n}} $ (part of $\Finf(1)$), and whose last column is the first column of $P(1)_{0,n+m-1}\cdot F_u^{(n+m)}(xy + z^2/(4\epsilon))^{1 + p + \hdots + p^{n+m-1}}$ (part of $\Finf(2)$). Since $1\leq a < A$, then for any $m\in \bZ_{\geq 0}$, we have $A(1 + \hdots + p^n)\neq A(1 + \hdots + p^{n+m-1}) + ap^{m+n}$. Therefore, 
regardless of the value of $m$, the $t$-adic valuation of entries of the first two columns of $P$ are different from the $t$-adic valuation of the last column of $P$. 

To prove that $w$ decays rapidly, it suffices to prove that among the monomials in $P w$ with $p$-adic valuation equalling $-(n+1)$, there exists a monomial with $t$-adic valuation $\leq A(1 + \hdots p^n)$. By the proof of \Cref{prop_decay} in Case 1 in \S\ref{sec_decayHsplit}, this in turn reduces to proving the following statement: if $m \geq 1$, then $w_u \bmod p$ is not in the kernel of $ P(1)_{0,n}\bmod p$; and if $m = 0$, the vector $\begin{bmatrix}
(\lambda^{-1})^{(n)}\\
1\\
\end{bmatrix} \bmod p$ is not in the kernel of $P(1)_{0,n-1}\bmod p$. Both statements follow from \Cref{linear}, establishing the decay of the rank $3$ submodule $\Span_{\bZ_p}\{w_1,w_2,w_3\}$.

\Cref{prop_decay} in this case follows from the observation that since $2a\leq A$, then $w_3$ decays very rapidly. 
%Note that the $t$-adic valuations of all the non-zero entries of $P$ are $\leq t^{A(1 + \hdots p^n)}$. 
%By Lemmas \ref{linear} and \ref{caseone}, if $w$ were in the span of $w_1,w_2$, or if $w$ was a multiple of $w_i$, then $w$ decays rapidly. Therefore, we assume that $w = \alpha_u w_u + \alpha_i w_i$ where $w_u$ is a non-zero vector in the span of $w_1,w_2$. 
%It suffices to show that the $t$-adic valuation of the coefficient of $1/p^{n+1}$ of $F_{\infty}(1)$ is always different from the $t$-adic valuation of the coefficient of $1/p^{m+1}$ of $F_{\infty}(i)$ for any choice of $m$ and $n$. As $d < A$, the former is strictly greater than the latter if $n\geq m$, and is strictly smaller if $n \leq m-1$. The result follows.
\end{proof}

\subsection{Case 2: $A \geq B, a \neq b$}\label{sec_decayS2}\quad

Note that if $A \geq B$, then $a + b = 2c$ (as the only way this can happen is if $xy$ has the same $t$-adic valuation as $z^2/(4\epsilon)$). We may therefore assume without loss of generality that $a < b$. It follows then that $a < c < b$. Within this case, we will need to consider the following two subcases.   

\subsubsection*{Subcase $(2.1)_e$: $B (1 + p^{2e-1}) < A(1+p) < B(1 + p^{2e + 1})$ for some $e\in \bZ_{\geq 1}$}

In this subcase, we will prove that $\Span_{\bZ_p}\{w_1,w_2, w_i\}$ decays rapidly, where $i\in\{3,4,5 \}$ will be chosen depending on the values of $a,b$ and $c$. 

The following lemma, in conjunction with \Cref{linear}, implies (as in Case 1) that $\Span_{\bZ_p}\{w_1,w_2\}$ decays rapidly. It can be proved by the same argument as in the proof of \Cref{realdecaytopleft}(1), so we omit its proof.
\begin{lemma}\label{twonetopleft}
Among the terms appearing in $\Finf(1)$ described in \Cref{niceparam} with denominator $p^{n+1}$, the unique term with minimal $t$-adic valuation is 
$$P(1)_{e,n}(xy + z^2/(4\epsilon))^{(1 + \hdots + p^{n-e})}(xy^p + x^py +z^{1+p}/(2\epsilon))^{p^{n-e+ 1} + p^{n-e +3} + \hdots + p^{n + e -1}} .$$
The $t$-adic valuation of this term is $A(1 + \hdots + p^{n-e}) + B(p^{n-e+1} + p^{n-e+3} + \hdots + p^{n+e-1})$.
\end{lemma}

The following lemmas will be used to show that one of $w_3,w_4,w_5$ also decays rapidly. These lemmas imply that among the terms appearing in $\Finf(2)$ with denominator $p^{n+1}$, for at least one of the columns of this matrix, there exists a unique term with minimum $t$-adic valuation.

\begin{lemma}\label{consec}
Given $g\in \bZ_{\geq 1}, n\in \bZ_{\geq 0}$, consider the multiset consisting of numbers of the form $A(1 + \hdots + p^{n - f -1}) + B(p^{n - f} + p^{n-f + 2} + \hdots + p^{n+f -2}) + gp^{n+f}$, as $f$ varies over $\bZ\cap [0, n]$. If the minimal number in this multiset occurs more than once, then it must occur for consecutive values of $f$. 
\end{lemma}
\begin{proof}
For any choice of $f$, let us denote the expression by $v(f)$. It suffices to prove the following statement: for $f_1 < f_2-1$, if $v(f_1) = v(f_2)$, then $v(f_2) > v(f_2 -1)$. 
To that end, suppose that $v(f_1) = v(f_2)$. Then $A(1 + p + \hdots p^{f_2-f_1-1}) = B(p^{f_2-f_1}-1)(p^{f_2 + f_1} +1)/(p^2-1)+gp^{f_2}(p^{f_2}-p^{f_1})$. 

To prove $v(f_2)>v(f_2-1)$,
%We now prove that $(1 + p + \hdots p^{f_2-f_1-1})v(f_2) > (1 + p + \hdots p^{f_2-f_1-1})v(f_2-1)$. Indeed,
note that $p^{-(n-f_2)}(v(f_2) - v(f_2-1)) = B(p^{2f_2-1}+1)/(p+1) + gp^{2f_2-1}(p-1) - A$. Multiplying this by $(1 + p + \hdots + p^{f_2-f_1})$ and applying the relation of $A$ and $B$ above, we have \[\frac{1 + p + \hdots + p^{f_2-f_1-1}}{p^{n-f_2}}(v(f_2) - v(f_2-1))=\frac{B(p^{f_2-f_1} - 1)(p^{2f_2 -1} - p^{f_1 + f_2})}{p^2-1} + g(p^{f_2 -f_1} -1)(p^{2f_2 -1} - p^{f_1 + f_2}),\] which is positive since $f_2 > f_1+1$. The lemma follows.  
\end{proof}

\begin{lemma}\label{onlytwo}
There are at most two numbers $g$ in the set $\{a,b,c\}$ such that there exists an integer $f$ ($f$ is allowed to depend on the choice of $g$ ) with $A(1 + \hdots p^{n-f - 1}) + B(p^{n - f} + p^{n-f + 2} + \hdots p^{n + f - 2}) + g p^{n + f} = A(1 + \hdots p^{n-f }) + B(p^{n - f + 1} + p^{n-f + 1} + \hdots p^{n + f - 3}) + g p^{n + f -1}$.\footnote{Note that if the equation holds, then $f$ is independent of $n$, since the equation is actually independent of $n$; see the proof of \Cref{consec}.}
\end{lemma}
\begin{proof}
Suppose there existed choices of $f\in \bZ_{\geq 0}$ for all three choices of $g$. Let $f_1,f_2,f_3$ be the choices for $f$. Then, by the proof of \Cref{consec}, we have that $ap^{2f_1-1}(p-1) = A - B(1 + p^{2f_1 -1})/(1+p)$, and similarly %expressions for $b,c$ with $f_2,f_3$ taking the place of $f_1$. 
$bp^{2f_2-1}(p-1) = A - B(1 + p^{2f_2 -1})/(1+p),\ cp^{2f_3-1}(p-1) = A - B(1 + p^{2f_3 -1})/(1+p)$.
Substituting these expressions in the equality $a + b = 2c$ yields the equation 
\[(p^{1-2f_1}+p^{1-2f_2}-2p^{1-2f_3})A=\frac{B}{p+1}(p^{1-2f_1}+p^{1-2f_2}-2p^{1-2f_3}).\]
Since $A\geq B\geq p+1$, we have $A\neq B/(p+1)$ and hence $p^{1-2f_1}+p^{1-2f_2}-2p^{1-2f_3}=0$. Since $f_1,f_2,f_3\in \bZ_{\geq 1}$, we must have $f_1=f_2=f_3$ and hence $a=b=c$, which is a contradiction.
\end{proof}

\begin{proof}[Proof of \Cref{prop_decay} in this case]
Let $h \in \{a,b,c\}$ be such that there is no $f$ which satisfies the hypothesis of \Cref{onlytwo} (indeed, the lemma guarantees the existence of such an $h$). 

We first show the existence of a rank $3$ submodule which decays rapidly.
Without loss of generality, we may assume that $h =a$ and we will prove that $\Span_{\bZ_p}\{w_1,w_2,w_3\}$ decays rapidly (if $h = b$ or $c$, the identical proof will show sufficient decay, with $w_4$ or $w_5$ taking the place of $w_3$). 

As in Case 1, \Cref{twonetopleft,linear,consec,onlytwo} imply that $\Span_{\bZ_p}\{w_1,w_2\}$ and $\Span_{\bZ_p}\{w_3\}$ both decay rapidly. Therefore, it suffices to show that $\alpha_u w_u + \alpha_3 w_3$ decays rapidly, where $w_u$ is a primitive vector in the span of $w_1,w_2$, and either $\alpha_u$ or $\alpha_3$ in $\bZ_p$ is a $p$-adic unit. 

By \Cref{twonetopleft}, the $t$-adic valuation of the coefficient of $1/p^{n+1}$ of $\Finf w_u$ is $d(n) = A(1 + \hdots + p^{n-e}) + B(p^{n-e + 1}+ p^{n-e+3} + \hdots + p^{n+e-1})$. Similarly, the $t$-adic valuation of the coefficient of $1/p^{m+1}$ of $ \Finf \cdot w_3$ is $c(m) = A(1 + \hdots + p^{m-f-1}) + B(p^{m-f }+ p^{m-f+2} + \hdots + p^{m+f-2}) + ap^{m + f}$ for some $f\in \bZ\cap [0,n]$. As in Case 1, it suffices to prove that $d(n)$ is never equal to $c(m)$, regardless of the values of $n$ and $m$. 

Let $c(f',m) =  A(1 + \hdots + p^{m-f-1}) + B(p^{m-f' }+ p^{m-f'+2} + \hdots + p^{m+f'-2}) + ap^{m + f'}$, for any value of $f' \leq m$. By the definition of $f$, $c(m) = c(f,m)$, and $f' = f$ minimizes the value of $c(f',m)$. 

If $n \geq m$, since $a<A$, then $d(n) > c(e,m) \geq c(f,m) = c(m)$, as required. On the other hand, if $m > n$, we have $c(m) > A(1 + \hdots + p^{m-f-1}) + B(p^{m-f }+ p^{m-f+2} + \hdots + p^{m+f-2})\geq d(n)$, where the last inequality follows from \Cref{twonetopleft}. 

Finally, we treat the question of very rapid decay. If we may take $h = a$ or $h = c$, the very rapid decay of $w_3$ or $w_5$ is established by the inequality $2 a < 2c \leq A$. Otherwise, $h$ must be $b$ and for both $a,c$, there exist $f_1,f_3$ satisfying the equation in \Cref{onlytwo}. Since $a\neq c$, then $f_1\neq f_3$ and at least one $f_i\geq 2$. By the proof of \Cref{onlytwo}, we have $A-B(1+p^{2f_i-1})/(p+1)>0$ and hence $A\geq 7B >2b$. Thus, $w_4$ decays very rapidly.
%as noted at the beginning of the proof, we may assume that $2B \leq A$, whence the very rapid decay of $w_4$ is established. The theorem follows. 
%Here is the original text: Note that if there exists a choice of $f_i$ for $i = 1,2$ (i.e., for $a,b$), it follows that $2B \leq A$. \yunqing{this is still true, but I obtain it with using that $B=pa+b$ and I got $A\geq 2p\frac{p-1}{p+1}B$. An alternative way is that say for $a,b$ both having $f$, we must have $f_1\neq f_2$ and hence one of $f\geq 2$ and we have $A\geq B(p^3+1)/(p+1)$. }
\end{proof}

\subsubsection*{Subcase $(2.2)_e$: $A(1+p) = B(1 + p^{2e - 1})$ for some $e\in \bZ_{\geq 1}$}

In this subcase, we will prove that $\Span_{\bZ_p}\{w_3,w_4,w_5\}$ decays rapidly. We first need the following lemma.

\begin{lemma}\label{casetwotwo}
Among the terms appearing in $\Finf(2)$ described in \Cref{niceparam} with denominator $p^{n+1}$, the unique term with minimal $t$-adic valuation is 
$$P(1)_{e -1,n-1}F_u^{(n+e -1)}(xy + z^2/(4\epsilon))^{(1 + \hdots + p^{n-e})}(xy^p + x^py +z^{1+p}/(2\epsilon))^{p^{n-e+ 1} + p^{n-e +3} + \hdots + p^{n + e -3}} .$$
The $t$-adic valuation of the $i^th$ column term is $A(1 + \hdots + p^{n-e}) + B(p^{n-e+1} + p^{n-e+3} + \hdots + p^{n+e-3}) + gp^{n+e-1},  $ where $g$ is either $a,b$ or $c$ depending on whether $i$ is $1,2$ or $3$.
\end{lemma}
\begin{proof}
It suffices to prove that choice of $f = e$ minimizes the expression  $A(1 + p + \hdots + p^{n-f}) + B(p^{n-f+1} + p^{n-f + 3} \hdots + p^{n+f-3}) + gp^{n+f-1}$, where $f$ is allowed to range between $0$ and $n$. This can be verified by direct calculation.  
\end{proof}

\begin{proof}[Proof of \Cref{prop_decay} in this case]
It follows from \Cref{casetwotwo,linear} that $w_3$, $w_4$ and $w_5$ individually decay rapidly, and that $w_3$ decays very rapidly. In order to show that $\Span_{\bZ_p}\{w_3,w_4,w_5\}$ decays rapidly, it suffices to show that the $t$-adic valuations of the coefficients $1/p^{l+1},1/p^{m+1}, 1/p^{n+1}$ of $F_{\infty}(w_3),F_{\infty}(w_4),F_{\infty}(w_5)$ are always distinct, regardless of the values of $l,m,n$. By \Cref{casetwotwo}, these quantities equal $A(1 + p + \hdots + p^{l-e}) + B(p^{l-e+1} + p^{l-e + 3} + \hdots + p^{l+e-3}) + ap^{l+e-1}$, $A(1 + p + \hdots + p^{m-e}) + B(p^{m-e+1} + p^{m-e + 3} + \hdots + p^{m+e-3}) + bp^{m+e-1}$ and $A(1 + p + \hdots + p^{n-e}) + B(p^{n-e+1} + p^{n-e + 3} + \hdots + p^{n+e-3}) + cp^{n+e-1}$. 

As $a,b,c$ are all strictly less than $B$, these quantities will all be different unless two of $l,m,n$ are equal. In this case, the quantities still differ, because $a,b,c$ are all distinct integers by assumption. Therefore, $\Span_{\bZ_p}\{w_3,w_4,w_5\}$ decays rapidly. 
%Let $A$ denote the $t$-adic valuation of the coefficient of $1/p^{n+1}$ in $F_{\infty}(1)$, $B$ denote the $t$-adic valuation of the coefficient of $1/p^{m+1}$ in $F_{\infty}(4)$ and $C$ denote the $t$-adic valuation of the coefficient of $1/p^{l+1}$ in $F_{\infty}(5)$. In order to show that any linear combination also decays rapidly, it suffices to show that independent of $n,m,l$, the integers $A$,$B$ and $C$ are always different. 
%Indeed, it is easy to see that if $l \geq m + e + 1$, $C>B$ and if $l \leq m+e,$ $C<B$. Similarly, if $n \geq m + e$, $ A>B$ and if $n \leq m+e-1$, $A<B$. Finally, $A$ and $C$ can't be equal either, and the result follows. 
\end{proof}

\subsection{Case 3: $A \geq B$ and $a = b$}\label{sec_decayS3}
%As $A \geq B$, if any two amongst $a,b,c$ are equal, then 
In this case, $a = b = c$. We may assume that $x(t) = t^a$, $y(t) = \beta t^a + \sum_{i=a+1}^\infty \beta_i t^i$, and $z(t) = \gamma t^a + \sum_{i=a+1}^\infty \gamma_i t^i$. Since $A\geq B$, we have $ \beta + \gamma^2 /(4\epsilon)  =0$. We will break the proof of the Decay Lemma into two subcases and the following lemma will be used in both cases. 

\begin{lemma}\label{aftercancel}
Suppose that $\gamma \in \bF_p$. Let $a' > a$ denote the smallest integer such that either $\beta_{a'}\neq 0$ or $\gamma_{a'}\neq 0$. Then both $\beta_{a'}$ and $\gamma_{a'}$ are non-zero and moreover, $B \geq (p-1)a + 2 a'$.  
\end{lemma}
\begin{proof}
Since $\gamma\in \bF_p$ and $\beta+\gamma^2/(4\epsilon)=0$, then $\beta\in \bF_p$.
Therefore, in $k[[t]]$,
$$xy + z^2/(4\epsilon) = \sum_{i \geq a'} (\beta_i + \gamma \gamma_i/(2\epsilon))t^{i+a} + (4\epsilon)^{-1}\sum_{i,j \geq a'}\gamma_i\gamma_j t^{i+j},$$
$$xy^p + x^p y + z^{1+p}/(2\epsilon) = \sum_{i \geq a'}(\beta_i + \gamma \gamma_i/(2\epsilon))t^{i + pa} + \sum_{i\geq a'}(\beta_i^p + \gamma \gamma_i^p/(2\epsilon)) t^{pi + a} + (2\epsilon)^{-1}\sum_{i,j \geq a'}\gamma_i\gamma_j^p t^{i + jp}.$$

If one of $\beta_{a'}$ and $\gamma_{a'}$ were zero, then $A = a' + a$, whereas $B\geq a'+pa$; this contradicts with the assumption that $A\geq B$. Hence, we obtain the first assertion of the lemma. 

Let $a'' \geq a'$ denote the smallest integer such that $\beta_i + \gamma\gamma_i/(2\epsilon) \neq 0$. Then by applying the Frobenius action, we have $\beta_{a''}^p+\gamma \gamma_{a''}^p/(2\epsilon)\neq 0$, and $B \geq \textrm{min}\{(p+1)a',a'' + pa \}$. % with equality if these two quantities are not equal. 
If $B \geq (p+1)a'$, then the second assertion of the lemma follows. 

Therefore, we assume that $B = a'' + pa < (p+1)a'$. The expansion of $xy + z^2/(4\epsilon)$ above has a non-zero term of the form $(\beta_{a''} + \gamma\gamma_{a''}/(2\epsilon))t^{a + a''}$. As $A\geq B$, the term $(\beta_{a''} + \gamma\gamma_{a''}/(2\epsilon))t^{a + a''}$ has to be cancelled out by a term of the form $(4\epsilon)^{-1}\sum_{i + j = a + a'', i,j \geq a'}\gamma_i\gamma_j t^{i+j}$. Therefore, it follows that $2a' \leq a + a''$ and hence $B = a'' + pa\geq (p-1)a+2a'$. 
\end{proof}

\subsubsection*{Case $(3.1)_e$: $B(1 + p^{2e-1}) < (p+1)A < B(1 + p^{2e + 1})$ for some $e\in \bZ_{\geq 1}$}\quad

The same argument as in Case 2.1 suffices to prove \Cref{prop_decay}, unless $A = B\frac{1 + p^{2e-1}}{1+p} + a(p^{2e} - p^{2e-1})$. Therefore, we will assume that this is the case. 

\begin{lemma}\label{casethreeone}
Among the terms appearing in $\Finf(2)$ described in \Cref{niceparam} with denominator $p^{n+1}$, there are exactly two with minimal $t$-adic valuation. They are:
$$P(1)_{e-1,n-1}F_u^{(n+e-1)} (xy + z^2/(4\epsilon))^{(1 + \hdots + p^{n-e})}(xy^p + x^py + z^{1+p}/(2\epsilon))^{p^{n-e + 1} + p^{n - e + 3 } + \hdots + p^{n + e -3}},$$
$$P(1)_{e,n-1}F_u^{(n+e)} (xy + z^2/(4\epsilon))^{(1 + \hdots + p^{n-e -1})}(xy^p + x^py + z^{1+p}/(2\epsilon))^{p^{n-e } + p^{n - e + 2 } + \hdots + p^{n + e -2}}. $$
Both the terms have $t$-adic valuation $A(1 + \hdots + p^{n-e}) + B(p^{n-e + 1} + p^{n - e + 3} + \hdots + p^{n + e -3}) + ap^{n + e -1} $.
\end{lemma}
\begin{proof}
This lemma follows from a similar argument as \Cref{realdecaytopleft}(2) and the proofs of \Cref{consec,onlytwo}, so we omit the details.
\end{proof}

\begin{proof}[Proof of \Cref{prop_decay} in this case]
We will show that either $w_3$ or $w_5$ decays very rapidly. There are two terms with minimal $t$-adic valuation as in \Cref{casethreeone}, appearing in the coefficient of $1/p^{n+1}$ of $F_{\infty}(w_3)$ and $F_{\infty}(w_5)$. A direct computation yields that the sum of these two terms equals by 
$$\frac{1}{2p^{n+1}}P(1)_{0,n-e-1}(xy + z^2/(4\epsilon))^{1 + p + \hdots + p^{n-e -1}}(X(t)u(t)^{p^{2e}} + Y(t)u(t)^{p^{2e-1}})^{(n-e)},$$
where
\begin{itemize}
\item $u(t)$ stands for either $x(t)$ or $z(t)$, according to whether we work with $w_3$ or $w_5$,

\item $X(t) = pF_u\cdot F_l^{(1)}\cdot pF_u^{(2)} \cdots F_l^{(2e-1)} \cdot [(\lambda^{-1})^{(2e)},1]^{\textrm{T}}$,   and 

\item $Y(t) = pF_t \cdot pF_u^{(1)} \cdot F_l^{(2)} \cdots pF_u^{(2e-3)} \cdot F_l^{(2e -2)} \cdot [(\lambda^{-1})^{(2e-1)},1]^{\textrm{T}}$. The superscript $T$ stands for transpose. 

\end{itemize}
The decay of $w_3$ and $w_5$ is determined by the $t$-adic valuation of the entries of $X(t) u(t)^{p^{2e}} + Y(t)u(t)^{p^{2e-1}}$. For the rest of the proof, it suffices to focus on the second row of $X(t), Y(t)$ and hence we view them as functions. We prove the very rapid decay of $w_3$ or $w_5$ in two cases.

\begin{enumerate}
\item
Both $\beta,\gamma \in \bF_p$. 

In this case, we claim that the $t$-adic valuation of $X(t) u(t)^{p^{2e}} + Y(t)u(t)^{p^{2e-1}}$ is at most $A + B(p + p^3 + \hdots + p^{2e-3}) + a'p^{2e-1}$ for at least one choice of $u(t)$ between $x(t)$ and $z(t)$, where $a'$ is defined in \Cref{aftercancel}.
This claim implies that the $t$-adic valuation of the coefficient of $1/p^{n+1}$ of $F_{\infty}(w_3)$ or $F_{\infty}(w_5)$ is at most $A(1 + \hdots + p^{n-e}) + B(p^{n-e+1} + p^{n-e + 3} + \hdots + p^{n + e -3}) + a'p^{n+e-1}$. This is sufficient to prove the rapid decay of $w_3$ or $w_5$. Indeed, this quantity is strictly less than $A(1 + \hdots + p^{n-f}) + B(p^{n-f+1} + p^{n-f + 3} + \hdots + p^{n + f -3}) + ap^{n+f-1}$ for all values of $f \neq e,e+1$ by \Cref{aftercancel} and hence the sum of the two terms in \Cref{casethreeone} gives the minimal $t$-adic valuation term of the coefficient of $1/p^{n+1}$ in $\Finf(w_3)$ or $\Finf(w_5)$. Moreover, the bounds on $a'$ in \Cref{aftercancel} proves that $w_3$ or $w_5$ decays very rapidly. 

We now prove the claim by contradiction.
Suppose that $X(t)x(t)^{p^{2e}} + Y(t)x(t)^{p^{2e-1}} $ has $t$-adic valuation greater than $A + B(p + p^3 + \hdots + p^{2e-3}) + a'p^{2e-1}$. Since $z(t) = \gamma x(t) + \gamma_{a'}t^{a'} + \hdots $ with $\gamma \in \bF_p, \gamma_{a'}\neq 0$ and we have assumed that $A = B\frac{1 + p^{2e-1}}{1+p} + a(p^{2e} - p^{2e-1})$, it follows that there is a unique monomial in $X(t)z(t)^{p^{2e}} + Y(t)z(t)^{p^{2e-1}} $ with $t$-adic valuation $A + B(p + p^3 + \hdots + p^{2e-3}) + a'p^{2e-1}$, thereby establishing the claim for $u(t)=z(t)$. 

\item
Either $\beta$ or $\gamma$ is not in $\bF_p$.

In this case, as $\beta + \gamma^2/(4\epsilon) = 0$, we may assume that $\gamma \notin \bF_p$. We again consider the function $X(t)u(t)^{p^{2e}} + Y(t)u(t)^{p^{2e-1}}$. Suppose that the leading coefficient of $X(t)$ is $\mu_X$ and that of $Y(t)$ is $\mu_Y$. Then, the terms of minimal equal $t$-adic valuations cancel out in the case when $u(t) = x(t)$ only if $\mu_X + \mu_Y = 0$, otherwise by the same idea as in (1), $w_3$ decays very rapidly. Therefore, we may assume that $\mu_X + \mu_Y = 0$. However in this case, if we pick $u(t) = z(t)$, then the terms terms with minimal equal $t$-adic valuations cancel out only if $\mu_X \gamma^{p^{2e}} + \mu_Y \gamma^{p^{2e-1}} = 0$, which is not possible as $\gamma^{p^{2e}} \neq \gamma^{p^{2e-1}}$. In other words, we show that in this case, $w_5$ decays very rapidly. 
\end{enumerate}

As in Case 2.1, $\Span_{\bZ_p}\{w_1,w_2\}$ decays rapidly, and also every vector that can be written as $\alpha_u w_u + \alpha_iw_i$ with $\alpha_i\in \bZ_p^\times $ ($i = 3,5$ depending on whether $w_3$ or $w_5$ decays) decays very rapidly. The latter statement follows by the same valuation-theoretic argument as in the proof of Case 2.1, which also proves that $\Span_{\bZ_p}\{w_1,w_2, w_i\}$ decays rapidly.
\end{proof}

\subsubsection*{Case $(3.2)_e: A(1+p)=B(1+p^{2e-1})$ for some $e\in \bZ_{\geq 1}$}

\begin{lemma}\label{casethreetwo}
Among the terms appearing in $\Finf(1)$ described in \Cref{niceparam} with denominator $p^{n+1}$, there are exactly two with minimal $t$-adic valuation.
They are:
$$P(1)_{e,n}(xy + z^2/(4\epsilon))^{1 + \hdots + p^{n-e}}(xy^p+ x^py + z^{1+p}/(2\epsilon))^{p^{n-e+1} + p^{n-e+3} + \hdots +p^{n+e-1}}, $$
$$P(1)_{e-1,n}(xy + z^2/(4\epsilon))^{1 + \hdots + p^{n-e+1}}(xy^p+ x^py + z^{1+p}/(2\epsilon))^{p^{n-e+2} + p^{n-e+4} + \hdots +p^{n+e-2}}.$$
 Both these terms have $t$-adic valuation $A(1 + \hdots + p^{n-e}) + B(p^{n-e+1} + p^{n-e+3} \hdots + p^{n+e-1})$.
\end{lemma}
As we have seen many lemmas of this flavor, we omit the proof.

This lemma shows that there are two terms with the same $t$-adic valuation, which could therefore lead to cancellation, and such phenomenon prevents us from proving that $\Span_{\bZ_p}\{w_1,w_2\}$ decays rapidly. Nevertheless, the following lemma shows that there is at least a saturated rank one submodule of $\Span_{\bZ_p}\{w_1,w_2\}$ which decays rapidly. 

\begin{lemma}\label{toptopodd}
There is a vector $w_0$ in $\Span_{\bZ_p}\{w_1,w_2\}$ which decays rapidly. 
\end{lemma}

\begin{proof}
By \Cref{casethreetwo} and the proof of \Cref{linear}, the coefficient (viewed as a power series in $t$) of the sum of the two terms with minimal $t$-adic valuation among the terms with denominator $p^{n+1}$ is of the form $ \mu_1 M_1 +  \mu_2 M_2$, for some $p$-adic units $\mu_i$, where $\displaystyle \{M_1,M_2\}=\left\{\begin{bmatrix}
1&\lambda^{-1}\\
\lambda&1\\
\end{bmatrix}, \begin{bmatrix}
1&-\lambda^{-1}\\
\lambda&-1\\
\end{bmatrix}\right\}$.

As $M_1 \bmod p$ and $M_2 \bmod p$ are not scalar multiples of each other, the linear combination $\mu_1 M_1 + \mu_2 M_2 \bmod p$ is non-zero. Therefore, there exists a vector $\bar{w}_0$ defined over $\bF_p$ which does not lie in $\ker(\mu_1 M_1 + \mu_2 M_2 \bmod p)$. Choosing $w_0\in \Span_{\bZ_p}\{w_1,w_2\}$ which lifts $\bar{w}_0$ finishes the proof of this lemma. 
\end{proof}

We are now ready to prove the last remaining case of \Cref{prop_decay} (and also the Decay Lemma \Cref{thm_decay}). 
\begin{proof}[Proof of \Cref{prop_decay}]
We will first prove that there is a rank $2$ submodule of $\Span_{\bZ_p}\{w_3,w_4,w_5\}$ which decays rapidly. For ease of notation, let $\bar{F_u}$ denote the matrix $\frac{1}{t^a}F_u$ evaluated at $t=0$.

Let $K$ denote $\ker(P(1)_{n-1,e-1}\bar{F_u}^{(n+e-1)} \bmod p)\cap \Span_{\bF_p}\{w_3,w_4,w_5\}$. If $\dim_{\bF_p}K\leq 1$, then lifting two linearly independent $\bF_p$-vectors $\notin K$ gives the desired rank $2$ submodule. Therefore, we assume that $\dim_{\bF_p}K=2$ (note that since $P(1)_{n-1,e-1}\bar{F_u}^{(n+e-1)} \bmod p$ is not the zero matrix, so $\dim_{\bF_p}K\neq 3$). It follows that $\beta,\gamma \in \bF_p$. 

We will prove that $\Span_{\bZ_p}\{w_3,w_4\}$ decays rapidly. First, 
since $K\cap \Span_{\bF_p}\{w_3,w_4\}=\Span_{\bF_p}\{\beta w_3-w_4\}$, then any primitive vector in $\Span_{\bZ_p}\{w_3,w_4\}$ which modulo $p$ is not a multiple of $\beta w_3 - w_4$ must decay rapidly. %Therefore, it suffices to prove that $\beta w_3 - w_4$ decays rapidly. 
Now we consider $\beta w_3 - w_4$. Up to constants, the coefficient of the $1/p^{n+1}$ part of the first entry of $F_{\infty}(\beta w_3 - w_4)$ equals $\beta_{a'} t^{A(1 + \hdots  +p^{n-e}) + B(p^{n-e+1} + p^{n-e+2} \hdots + p^{n+e-3}) + a' p^{n+e -1}}$. 
\Cref{aftercancel} establishes the required decay as follows: firstly, as $a' \leq B \leq A$, we have that the vector $\beta w_3 - w_4$ decays rapidly. Secondly, the exact bound for $a'$ in \Cref{aftercancel} implies (as in the proof in Case 2.1) that $\Span_{\bZ_p}\{w_3,w_4\}$ decays rapidly. Finally, the \emph{very} rapid decay of $w_3$,$w_4$ follows from the bound $2a' \leq B \leq A$. 

Then, \Cref{prop_decay} follows by an argument analogous to that in Case 2.1 with \Cref{toptopodd}. 
\end{proof}

%%%%%%%%%%%%%%%%%%%%%%%%%%%%%%%%%%%%%%%%%%%%%%%%%%%%%Section 6%%%%%%%%%%%%%%%%%%%%%%%%%%%%%%%%%%%%%%%%%%%%%%%%%%%%%%%%%%
\section{The setup of the main proofs}\label{sec_outline}
In this section, we provide the general setup of the proofs of \Cref{thm_main,thm_max}. As mentioned in \S\ref{sec_intro_pf}, the proofs consist of the following parts:
\begin{enumerate}
    \item The sum of the local contributions at supersingular points is at most $11/12$ of the global contribution; and 
    \item the local contribution from non-supersingular points is of smaller magnitude.
\end{enumerate}
\Cref{supersingular} makes (1) precise, and is stated in  \S\ref{subsec_outline}. We will prove \Cref{supersingular} and (2) in \S\ref{sec_pf_Hil} for the Hilbert case and in \S\ref{sec_pf_Sie} for the Siegel case. The idea involved in the statement of \Cref{supersingular} is that we break the global intersection number $C . Z(m)$ into pieces, one for each non-ordinary point on $C$, by using the relation between the Hasse invariant and the Hodge line bundle in \S\ref{sec_decomp}. We also relate the local intersection multiplicity at a point to a lattice-point count.

\subsection{The global contribution and its decomposition}\label{sec_decomp}
Recall that in \S\ref{def_setT}, we list the set $T$ of $m\in \bZ_{>0}$ for which we will study $C.Z(m)$ to prove our main theorems. In order to study the asymptotic behavior, we define $T_M=\{m\in T\mid m\leq M\}$ for $M\in \bZ_{>0}$. Moreover, in \S\S\ref{sec_pf_Hil}-\ref{sec_pf_Sie}, we will construct a subset $S_M\subset T_M$ which consists of bad values of $m$ that we want to rule out. The total global intersection number that we will consider is $\sum_{m\in T_M-S_M} C.Z(m)$. We sum over $m$ instead of working with individual $m$ because geometry-of-numbers techniques which we use to bound the local intersection multiplicity (for cumulative $m$) do not work for individual $m$. The following lemma gives the asymptotics of the global term using results in \S\ref{sec_global}.

\begin{lemma}\label{lem_glo-sum_asymp}
Assume that $\#S_M=O(M^{1-\epsilon})=O(\# T_M^{1-\epsilon})$ for some $\epsilon>0$ if $L=L_H$ and that $\#S_M=o(\# T_M)$ if $L=L_S$. Then
\[\sum_{m\in T_M-S_M} C.Z(m)=(\omega.C)\sum_{m\in T_M-S_M} |q_L(m)|+o(\sum_{m\in T_M-S_M} |q_L(m)|).\]
Moreover, we have, for \Cref{thm_main}(2), $\sum_{m\in T_M-S_M} C.Z(m)\asymp M^2$; for \Cref{thm_main}(1) and \Cref{rmk_realquad}, $\sum_{m\in T_M-S_M} C.Z(m)\asymp M^2/\log M$; for \Cref{thm_max}, $\sum_{m\in T_M-S_M} C.Z(m)\asymp M^{5/2}/\log M$.
\end{lemma}
\begin{proof}
By \S\ref{sec_asymp-qL} and the assumption on $S_M$, we have $\sum_{m\in S_M}|q_L(m)|=o(\sum_{m\in T_M}|q_L(m)|)$. Then
the assertions follow from \S\ref{sec_asymp-qL}, \Cref{lem_glo-bound}, \Cref{lem_sum-asymp-H}, and the prime number theorem.
\end{proof}

For each non-ordinary point $P$ on $C\cap Z(m)$, we introduce the notion of \emph{global intersection number $g_P(m)$ at $P$} using the following (well-known) relation between the non-ordinary locus and the divisor class of the Hodge bundle. Note that in the proof, we will only use the notion $g_P(m)$ for a supersingular point. 

\begin{lemma}
The non-ordinary locus in $\cM_k$ and $\cM^{\rm{tor}}_k$ is cut out by a Hasse-invariant $H$, which is a section of $\omega^{p-1}$, and hence the number of non-ordinary points (counted with multiplicity) on $C$ is given by $(p-1)(C.\omega)$.
\end{lemma}
See for instance \cite[\S\S 1.4-1.5, Theorem 6.2.3]{Boxer} for an explanation of this fact (and we use the fact that the ordinary Newton stratum coincides with the ordinary Ekedahl--Oort stratum). For the last assertion in the lemma, we remark that when $L=L_H$, the boundary $\cM^{\rm{tor}}_k\setminus \cM_k$ is ordinary and hence the intersection of $C'$ (in \S\ref{def_int}) with the non-ordinary locus is the same as the intersection of $C$ with the non-ordinary locus.

\begin{defn}
Let $t$ be the local coordinate at $P$ (i.e., $\widehat{C}_{P}=\Spf k[[t]]$) and let $A=v_t(H)$. We define $g_P(m)=\frac{A}{p-1}|q_L(m)|$.
\end{defn}
Note that by the above lemmas, we have the following decomposition\[\sum_{P\in C \text{ non-ord}}\sum_{m\in T_M-S_M}g_P(m)=\sum_{m\in T_M-S_M}|q_L(m)|(\omega. C)=\sum_{m\in T_M-S_M} C.Z(m) + o(\sum_{m\in T_M-S_M}|q_L(m)|(\omega. C)).\]
%\footnote{For the cycle class of $C$ in the Chow group of $\cM_k$, we take it to be the (reduced) image of $C\rightarrow \cM_k$ if the automorphism group $\Aut(C_{\bar{\eta}})$ of the geometric generic point $\bar{\eta}$ of $C$ is trivial. Otherwise, the cycle class of $C$ is given by $\Aut(C_{\bar{\eta}})$ times the reduced image of $C$. More concretely, for a divisor $Z$ of $\cM_k$, we have $C.Z=\sum_{P\in (C\times_{cM_k}Z)(k)} \frac{\Aut(C_{\bar{\eta}})}{\Aut(P)}\Length(\cO_P)$, where $\cO_P$ denote the \'etale local ring of the intersection of the reduced image of $C$ with $Z$ at $P$. By the projection formula, the local intersection multiplicity $\frac{\Aut(C_{\bar{\eta}})}{\Aut(P)}\Length(\cO_P)$ at $P$ can be computed in the universal deformation ring $R$ of $P$ (equivalently, the local complete ring of a lift of $P$ to a cover of $\cM_k$ with enough level structure) as follows. Pick a lift $\tilde{C}$ of $\widehat{C}_{P}$ to $\Spf R$ and let $\tilde{Z}$ denote the preimage of $Z$ in $\Spf R$, then the intersection multiplicity is the length of the local ring of $\tilde{C}\times_{\Spf R} \tilde{Z}$ at $P$. Therefore, in our convention, the integer $A$ is the local intersection multiplicity of $C$ with the Hasse invariant at $P$.} 

\subsection{The lattices and the outline of the proof}\label{subsec_outline}
 Let $B \rightarrow \Spf k[[t]]$ denote the generically ordinary abelian surface given by fulling back the universal family over $\cM_k$ to $\widehat{C}_{P}=\Spf k[[t]]$ for some point $P\in C$. Recall the notion of special endomorphisms from \S\ref{sec_sp_end} and by a slight abuse of terminology, when $L=L_H$, we will also refer to a special quasi-endomorphism with certain integrality condition in \S\ref{def_spdiv_H} as a special endomorphism. For any $n\in \bZ_{>0}$, the lattice is special endomorphisms of $B \bmod t^n$ is a sublattice of $B\bmod t$, which is equipped with a positive definite quadratic form $Q'$ (see \Cref{def_L''}). 

\begin{lemma}\label{oleg}
The local intersection multiplicity of $C.Z(m)$ at $P$, denoted by $l_P(m)$, equals $$\sum_{n = 1} ^ {\infty} \#\{ \textrm{Special endomorphisms } s \textrm{ of }B \bmod t^n \textrm{ with } Q'(s)=m\}.$$
\end{lemma}
The lemma follows directly from the moduli interpretation of $Z(m)$. Note that as $B$ generically has no special endomorphisms, this infinite sum can actually be be truncated at some finite stage (which will depend on $m$). 

\begin{remark}\label{rmk_rk}
Given $B$, the lattices of special endomorphisms of $B\bmod t^n$ have the same rank for all $n\in \bZ_{>0}$. Indeed, the work of de Jong, Moonen and Kisin cited in the proof of \Cref{thm_decay} applies to any $P$ and for any special endomorphism $w$ of $B\bmod t$, we have the parallel extension $\tilde{w}\in K[[t]]$, which is invariant under the Frobenius on $\bL_{\cris}(W[[t]])$. By de Jong's theory (here we need the fully faithfulness of the Dieudonn\'e functor, see \cite[Cor.~2.4.9]{dJ95}), whether $w$ extends over $\bmod t^n$ depends on the $p$-powers in the denominators of the coefficients of $\tilde{w}$. Therefore, given $n$, there exists $N$ such that $p^Nw$ extends over $\bmod t^n$ and hence these lattices tensor $\bZ_\ell, \ell\neq p$ are all isomorphic and in particular, the rank of the lattices is independent of $n$.
\end{remark}

Motivated by the Decay Lemmas \Cref{thm_decay,thm_decay_sg}, we define the following lattices for supersingular points (note that the notation is slightly different from that in the introduction and we will use the notation in this section for the rest of the paper).

\begin{para}\label{def_smallL'}
Assume $P$ is supersingular and recall that $A=v_t(H)$, where $H$ is the Hasse invariant and in \Cref{def_decay}, $A_n=[A(p^n+p^{n-1}+\cdots+1+\frac{1}{p})]$. 
\begin{itemize}
\item If $B \bmod t$ is superspecial, define $L_{0,1}$, $L_{n,1}, n\in \bZ_{>0}$, and $L_{n,2}, n\in \bZ_{\geq 0}$ to be the lattices of special endomorphisms of $B$ mod $t$, mod $t^{A_{n-1}+1}$, and mod $t^{A_{n-1}+ap^n+1}$ respectively. As in \Cref{def_L''}, we pick a lattice $L'_{n,i}\subset L'$ such that  $L_{n,i}\subset L'_{n,i}$ and for $\ell\neq p$, $L'_{n,i}\otimes \bZ_\ell=L'\otimes \bZ_\ell$ and $L'_{n,i}\otimes \bZ_p=L_{n,i}\otimes \bZ_p$. In particular, $L'_{0,1}=L'$ and by \Cref{thm_decay}, we have $[L'_{n,1}: L'_{n,2}]\geq p$ and  $[L':L'_{n,1}]\geq p^{3n}$.

\item If $B \mod t$ is supergeneric, define $L_0$ and $L_n, n\in \bZ_{> 0}$ to be the lattice of special endomorphisms of $B$ mod $t$ and mod $t^{A_{n-1}+1}$.  Again, as in \Cref{def_L''}, we pick a lattice $L'_n\subset L'$ such that $L_n\subset L'_n$ and for $\ell\neq p$, $L'_{n}\otimes \bZ_\ell=L'\otimes \bZ_\ell$ and $L'_{n}\otimes \bZ_p=L_{n}\otimes \bZ_p$; \Cref{thm_decay_sg} implies that $[L':L'_n]\geq p^{3n}$.
\end{itemize}
Since we assume that $C$ does not admit any global special endomorphisms, we have $\cap_{n=0}^\infty L_{n,i}=\{0\} =\cap_{n=0}^\infty L_n$. By \Cref{rmk_rk}, the difference between $L'_n, L'_{n,i}$ and $L_n,L_{n,i}$ is the same as that between $L_0,L_{0,i}$ and $L'$, we also have $\cap_{n=0}^\infty L'_{n,i}=\{0\} =\cap_{n=0}^\infty L'_n$.
\end{para}

\begin{corollary}\label{lem_count}
 \begin{enumerate}
     \item If $P$ is superspecial, then 
     \[l_P(m)\leq \frac{A(p+2)}{2p}r_{0,1}(m)+ \frac{A}{2}r_{0,2}(m)+ \sum_{n=1}^\infty \frac{Ap^n}{2}(r_{n,1}(m)+r_{n,2}(m)),\]
     where $r_{n,i}(m)=\#\{s\in L'_{n,i}\mid Q'(s)=m\}$.
     \item If $P$ is supergeneric, then 
     \[l_P(m)\leq \frac{A(p+1)}{p}r_{0}(m)+ \sum_{n=1}^\infty Ap^n r_n(m),\]
     where $r_{n}(m)=\#\{s\in L'_{n}\mid Q'(s)=m\}$.
 \end{enumerate}
\end{corollary}
\begin{proof}
By \Cref{oleg} and \S\ref{def_smallL'}, we have that for $P$ superspecial, 
\begin{align*}
    l_P(m) & \leq (A_{-1}+a)r_{0,1}(m)+(A_0-A_{-1}-a) r_{0,2}(m) + \sum_{n=1}^\infty (ap^n r_{n,1}(m)+ (A_n-A_{n-1}-ap^n) r_{n,2}(m))\\
    & =A_{-1}(r_{0,1}(m)-r_{0,2}(m)) + a\sum_{n=0}^\infty p^n(r_{n,1}(m)-r_{n,2}(m))+\sum_{n=0}^\infty A_n (r_{n,2}(m)-r_{n+1,2}(m))\\
    & \leq \frac{A}{p}(r_{0,1}(m)-r_{0,2}(m)) + \frac{A}{2}\sum_{n=0}^\infty p^n(r_{n,1}(m)-r_{n,2}(m))+\sum_{n=0}^\infty (A(p^n+\cdots +1+ p^{-1}) (r_{n,2}(m)-r_{n+1,2}(m)),
\end{align*}
where the last equality follows from the facts that $r_{n,1}(m)\geq r_{n,2}(m), r_{n,2}(m)\geq r_{n+1,2}(m)$ and $a\leq A/2, A_n\leq A(p^n+\cdots + p^{-1})$. We then obtain the assertion in (1) by rearranging the summations. A similar argument gives (2).
\end{proof}

The main task of the next two sections is to prove that
\begin{proposition}\label{supersingular}
Given $C$, there exists $S_M$ satisfying the assumption in \Cref{lem_glo-sum_asymp} such that for every supersingular point $P$ on $C$, we have \[\sum_{m\in T_M-S_M}l_P(m)\leq \frac{11}{12}\sum_{m\in T_M-S_M}g_P(m)+o(\sum_{m\in T_M-S_M}g_P(m)).\]
\end{proposition}

Once we have this proposition, we will prove that the local contribution from non-supersingular points have smaller order of magnitude, whence we conclude that there are infinitely many non-supersingular points on $C$ which lie in the desired special divisors.

\subsection{Ordinary points}
%SAVE: we will also use the ordinary decay lemma below for the Hilbert case so that we do not need to rule out $m$ representable by binary quadratic forms of large discriminant from ordinary points. Indeed, for the asymptotic of $\sum q_L(m)$, we need $\#S_M$ to have a power saving, which does not hold for binary quadratic forms.
In order to bound $l_P(m)$, we need the following decay lemma for ordinary points, which follows directly from Serre--Tate theory. We thank Keerthi Madapusi Pera for pointing this out to us. Let $B \rightarrow \Spf k[[t]]$ denote the abelian surface with ordinary reduction given by pulling back the universal family over $\cM_k$ to $\widehat{C}_{P}=\Spf k[[t]]$ for an ordinary point $P$.

\begin{lemma}\label{decay-ord}
If $w$ is not a special endomorphism for the $p$-divisible group $B[p^\infty] \bmod t^{A+1}$, then $pw$ is not a special endomorphism for $B[p^\infty]\bmod t^{pA+1}$.
\end{lemma}
%SAVE: the original version (from Keerthi's email) is: "Suppose that $f$ is a local equation at an ordinary point (with respect to $q$-coordinates\footnote{This is not necessarily unique since we only talk about the local equation up to unit.}) for the special divisor $Z(m)$. Then $f^p$ is a local equation for the special divisor $Z(p^2m)$." This version is not good enough for us because we need decay for each special endomorphism $w$; the defining equation for $Z(m)$ is the product of all such $w$.
\begin{proof}
Note that an endomorphism of $B[p^\infty] \bmod t^n$ is special if and only if its reduction on $B[p^\infty]\bmod t$ is special. Hence we only need to consider the deformation of endomorphisms.
Let $\{q_i\}$ ($i=1,2$ for the Hilbert case and $i=1,2,3$ for the Siegel case) denote $q$-coordinates of the formal group in the Serre--Tate deformation theory. Then, the formal neighborhood of the Shimura variety is given by $\Spf W[[t_i]]$, where $t_i = q_i-1$. Note that over $W[[t_i]]$, $s$ is an endomorphism of a point $x = (x_1,x_2,x_3)$ \footnote{i.e., $x$ is defined by setting $t_1=x_1,\ t_2=x_2$ and $t_3=x_3$} if and only if, all points $y = (y_1,y_2,y_3)$ satisfying the condition $(y_i+1)^p=x_i+1$ (equivalently $(y_i+1)^p -1 = x_i$) have the property that $ps$ is an endomorphism of $y$. In other words, if $f(t_i) = f(q_i-1)=0$ is the local equation of the locus on $W[[t_i]]$ such that the endomorphism $w$ deforms, then $f(q_i^p-1) = f((t_i+1)^p-1)=0$ is the local equation for $pw$ to deform. Hence over $k[[t]]$, if $f=0$ is the local equation that deforms $w$, then $f^p=0$ is the local equation for $pw$, and this gives the desired assertion.
\end{proof}

%SAVE: the decay at ordinary points may also be computed following the argument in \S\ref{sec_decay_Hil}. For instance, if the $p$-divisible group at $P$ is $(\bQ_p/\bZ_p)^2\times \mu_{p^\infty}^2$ with the product polarization on each $(\bQ_p/\bZ_p)\times \mu_{p^\infty}$. Following Li--Oort and Ananth's notation, we have $\varphi(f_1)=f_1, \varphi(f_2)=pf_2, \varphi(f_3)=pf_3, \varphi(f_4)=f_4$. Then on $\bL_{\cris}$, we have the quadratic form to be $\begin{bmatrix} 0 & I & 0\\ I & 0 & 0 \\ 0 & 0 & 2 \end{bmatrix}$ and $\varphi=diag(1,p^{-1},1, p, 1)$. By Kisin, on the versal deformation space $Frob=\begin{bmatrix} 1 & 0 & 0 & -yp & 0\\ x & 1/p & y & -(xy+z^2/4)p & z \\ 0 & 0 & 1 & -xp & 0\\ 0 & 0 & 0 & p & 0\\ 0 & 0 & 0 & -pz/2 & 1 \end{bmatrix}\circ \sigma$. Then compute $F_\infty$, for instance, the second row of the first column is given by $x^p/p, x^{p^2}/p^2, \cdots, x^{p^n}/p^n, \cdots$. This also gives the desired decay.

\begin{lemma}\label{lem_lat_ord}
Let $L_0, L_n, n\in \bZ_{>0}$ be the lattices of special endomorphisms of $B\bmod t$ and $B\bmod t^{Ap^{n-1}+1}$ respectively where $A\in \bZ_{>0}$. Then
\begin{enumerate}
    \item for any $A$, we have $\rk_{\bZ}L_n\leq 2$ if $L=L_H$ and $\rk_{\bZ}L_n\leq 3$ if $L=L_S$;
    \item there exist a constant $A$ and a $\bZ_p$-lattice $\Lambda$ (depending on $P$) with $\rk_{\bZ_p}\Lambda \leq 1$ when $L=L_H$ and $\rk_{\bZ_p}\Lambda \leq 2$ when $L=L_S$ such that $L_n\subset (\Lambda +p^{n-1} L_1\otimes \bZ_p)\cap L_0$.
\end{enumerate}
In particular, if $\rk_{\bZ}L_n=3$ when $L=L_S$ or $\rk_{\bZ}L_n=2$ when $L=L_H$, then $(\disc L_n)^{1/2}\geq p^{n-1}$.
\end{lemma}

\begin{proof}
Note that $L_n\subset L_n\otimes \bZ_p\subset  L_0\otimes \bZ_p=\bL_{\cris,P}(W)^{\varphi=1}$, where $\bL_{\cris,P}$ is the fiber of the $F$-crystal $\bL_{\cris}$ defined in \Cref{def_Lcris_S} and \Cref{def_Lcris_H} and $\varphi$ is the Frobenius action. Since $P$ is ordinary, then $\varphi$ acts on $\bL_{\cris,P}(W)$ with slope $-1,1,0,0$ (Hilbert case) or $-1,1,0,0,0$ (Siegel case) and hence (1) follows. 

Let $\Lambda'$ be the $\bZ_p$-lattice of special endomorphisms of $B[p^\infty]$.
Since $\widehat{C}_{P}$ is not contained in any special divisor,\footnote{This is the assumption of \Cref{thm_main}(1) and \Cref{thm_max}; and for \Cref{thm_main}(2), we may assume this as otherwise, the conclusion is automatic.} $B[p^\infty]$ admits at most a rank $2$ (resp. rank $1$) module of special endomorphisms when $L=L_S$ (resp. $L=L_H$); indeed, if $\rk_{\bZ_p}\Lambda'=3$ (resp. $2$), then $\Lambda'\otimes\bQ_p=L_0\otimes \bQ_p$, and thus the $B$ admits special endomorphisms.

We now mimic the proof of \cite[Thm.~4.1.1]{Ananth17} using \Cref{decay-ord} instead of \cite[Lem.~4.1.2(2)]{Ananth17}. 
Let $\Lambda\subset L_0\otimes \bZ_p$ be the saturation of $\Lambda'$ in $L_0\otimes \bZ_p$; then there exists $\Lambda_0\subset L_0\otimes \bZ_p$ such that $L_0\otimes\bZ_p=\Lambda\oplus \Lambda_0$. Let $\Lambda_n$ denote $(L_n\otimes\bZ_p+\Lambda)\cap \Lambda_0$; then
$L_n\otimes \bZ_p+\Lambda=\Lambda\oplus \Lambda_n$. 
It suffices to show that there exists $A$ such that $\Lambda_n\subset p\Lambda_{n-1}$ (and this implies that $\Lambda_n\subset p^{n-1}\Lambda_1$).

By definition, none of the elements in $\Lambda_0$ extend to $\Spf k[[t]]$, then there exists $A$ such that $\Lambda_1\subset p\Lambda_0$. For $n\geq 2$, assume for contradiction that there exists $\alpha\in \Lambda_n\backslash p\Lambda_{n-1}$. If $\alpha\in p\Lambda_{n-2}$, then write $\alpha=p\beta$ with $\beta\in \Lambda_{n-2}$. Since $p\beta=\alpha\in \Lambda_n$, then by \Cref{decay-ord}, $\beta\in \Lambda_{n-1}$, which contradicts with the assumption that $\alpha\notin p\Lambda_{n-1}$. Thus we have $\alpha\notin p\Lambda_{n-2}$; by iterating the argument, we have $\alpha\notin p\Lambda_0$. This is a contradiction since $\alpha\in \Lambda_n\subset \Lambda_1\subset p\Lambda_0$.
\end{proof}

%%%%%%%%%%%%%%%%%%%%%%%%%Section 7%%%%%%%%%%%%%%%%%%%%%%%
\section{Proof of \Cref{thm_main}(\ref{item_H})}\label{sec_pf_Hil}
In this section, we use the results proved in \S\S\ref{sec_global}-\ref{sec_decay_Hil} to prove \Cref{supersingular} in the case of Hilbert modular surfaces. This, in conjunction with \Cref{newbert}, yields \Cref{thm_main}(2).

\subsection{The bad set $S_M$ and the local intersection multiplicities at non-supersingular points}
We first construct the set $S_M$; the following lemma only concerns ordinary and superspecial points because we only need to consider such $P$ for the proof of \Cref{thm_main}(2). Indeed, if $P\in Z(m)$, then $P$ is either ordinary or supersingular and if $P\in Z(m), p\nmid m$, then by \S\ref{par_density_H}(1), $P$ is not supergeneric. Therefore for $P\in Z(m), m\in T$, $P$ is either superspecial or ordinary.
\begin{lemma}\label{lem_setSM}
Notation as in \S\ref{sec_decomp},\ref{def_smallL'} and \Cref{lem_lat_ord}. Given a finite set $\{P_i\}\subset (C\cap(\cup_{m\in \bZ_{>0}} Z(m)))(k)$,
there exists $S_M\subset T_M$ with $\#S_M=O(M^{1-\epsilon})$ for some $0<\epsilon<1/6$ such that for all $i$,
\begin{enumerate}
    \item if $P_i$ is superspecial, then $\{s\in L'_{N,1} \mid 0\neq Q'(s)\leq M, Q'(s)\notin S_M\}=\emptyset$ where $N=\frac{1+\epsilon}{3}\log_p M$;
    \item if $P_i$ is ordinary, then $\{s\in L_{N} \mid 0\neq Q'(s)\leq M, Q'(s)\notin S_M\}=\emptyset$ where $N=\epsilon\log_p M$.
\end{enumerate}
\end{lemma}
\begin{proof}
Since the union of finitely many sets with cardinality $O(M^{1-\epsilon})$ still has cardinality to be $O(M^{1-\epsilon})$, it suffices to prove the assertion for each $P_i$ separately. We follow the idea of the proof of \cite[Thm.~4.3.3]{Ananth17}.

If $P_i$ is superspecial, we take $S_M=\{m\in T_M\mid \exists s\in L'_{N,1} \text{ with } Q'(s)=m\}$ and then it satisfies (1) by definition. Note that $\#S_M\leq \#\{s\in L'_{N,1}\mid Q'(s)\leq M\}$. Then by a geometry-of-numbers argument (see for instance \cite[Lem.~4.2.1]{Ananth17}) and \Cref{thm_decay}, we have \[\#\{s\in L'_{N,1}\mid Q'(s)\leq M\}=O(M^2/p^{3N}+M^{3/2}/p^{2N}+M/p^N+ M^{1/2}/d_N),\]
where $d_N$ is the first successive minimum of $L'_{N,1}$ and $d_N\rightarrow \infty$ as $N\rightarrow \infty$ because $\cap L'_{N,1}=\{0\}$. Then $\#S_M=O(M^{1-\epsilon})$ by the definition of $N$.

If $P_i$ is ordinary, then $\rk L_N=2$ by \Cref{lem_lat_ord} and the fact that $\rk L_N=\rk L_0$ is even by the Tate conjecture. Similar to the superspecial case, we take $S_M=\{m\in T_M \mid \exists s\in L_{N} \text{ with } Q'(s)=m\}$ and then by \Cref{lem_lat_ord}, $\#S_M=O(M/p^N+M^{1/2}/d_N)=O(M^{1-\epsilon})$.
\end{proof}

\begin{lemma}\label{newbert}
Notation as in \Cref{lem_setSM}.
For an ordinary point $P=P_i\in C(k)$, we have \[\sum_{m\in T_M-S_M} l_P(m)=O(M^{1+\epsilon})=o(M^2).\]
\end{lemma}
\begin{proof}
By Lemmas \ref{oleg}, \ref{lem_lat_ord} and \ref{lem_setSM}, 
\[\sum_{m\in T_M-S_M}l_P(m)=\sum_{m\in T_M-S_M}A(r_0(m)+\sum_{n=1}^N(p^n-p^{n-1})r_n(m))\leq A\sum_{n=0}^N p^n\sum_{m=1}^M r_n(m),\]
where $r_n(m)=\#\{s\in L_n\mid Q'(s)=m\}$. By a geometry-of-numbers argument and \Cref{lem_lat_ord}, we have $\sum_{m=1}^M r_n(m)=O(M/p^n+M^{1/2}/d_n)$, where $d_n$ is the first successive minimum of $L_n$ and the implicit constant here only depends on $p$. Thus $\sum_{m\in T_M-S_M}l_P(m)=O(NM+p^N M^{1/2})=O(M^{1+\epsilon})$.
\end{proof}

%SAVE: old proof for non-supersingular points in the Hilbert case. Note that this proof requires to rule out a density zero set $S_M$, which may not have a power saving and hence brings difficulty in control the global contribution. Proof: the only points on special divisors are ordinary and supersingular points. For ordinary points, the lattice of special endomorphism is a rank $2$ quadratic form. Since we have assumed that the curve $C$ does not admit any global special endomorphisms, the discriminant of the rank $2$ quadratic form goes to $\infty$. Therefore, as $M\rightarrow \infty$, the density of integers/primes represented by these quadratic goes to $0$. We enlarge $S_M$ in \Cref{sec_pf} to contain these bad numbers. Hence, there exists a constant $c$ (which only depends on the ordinary point, the curve $C$, $p$, and $\epsilon$, the small density of bad numbers; it is independent of $m,M$) such that for any $m\notin S_M$, the local multiplicity is less than $c$. Therefore, the local contribution of the ordinary point (intersecting with all $Z(m)$, $m\in \{1,\dots, M\}-S_M$) will be bounded by $c\sum_{m=1}^M r_0(m)=O(M)$. It is easy to see that the global contribution is, up to multiplying by a constant bounded away from zero, $M^2$.

\subsection{Proof of \Cref{supersingular} in the Hilbert case}\label{sec_pf}
We follow the notation in \Cref{lem_setSM} and $P=P_i$ superspecial. We break $\sum_{m\in T_M-S_M}l_P(m)$ into two parts and are treated in the following lemmas.
%SAVE: the lemmas also hold true for supergeneric points by the same argument.

\begin{lemma}\label{ss_H_tail}
Notation as in \Cref{lem_count}.
For any $\epsilon>0$, there exists $c\in \bZ_{>0}$ which only depends on $P$ and $\epsilon$ such that
\[\sum_{m\in T_M-S_M}\sum_{n=c}^\infty \frac{Ap^n}{2}(r_{n,1}(m)+r_{n,2}(m))\leq \epsilon M^2 + o(M^2).\]
\end{lemma}
\begin{proof}
By \Cref{lem_setSM}, $r_{n,i}(m)=0$ for $n>N=\frac{1+\epsilon}{3}\log_p M$ and hence 
\[\sum_{m\in T_M-S_M}\sum_{n=c}^\infty \frac{Ap^n}{2}(r_{n,1}(m)+r_{n,2}(m))=\sum_{m\in T_M-S_M}\sum_{n=c}^N \frac{Ap^n}{2}(r_{n,1}(m)+r_{n,2}(m))\leq \sum_{n=c}^N\sum_{m=1}^M Ap^n r_{n,1}(m)\]
since $r_{n,1}(m)\geq r_{n,2}(m)$.

By a geometry-of-numbers argument, $\sum_{m=1}^M r_{n,1}(m)\leq c_2(M^2/p^{3n}+M^{3/2}/p^{2n}+M/p^n+M^{1/2}/d_n)$, where $c_2$ is an absolute constant and $d_n$ is the first successive minimum of $L'_{n,1}$. Hence
\[\sum_{n=c}^N Ap^n \sum_{m=1}^M r_{n,1}(m)\leq Ac_2M^2\sum_{n=c}^N 1/p^{2n}+\sum_{n=c}^N Ap^nc_2(M^{3/2}/p^{2n}+M/p^n+M^{1/2}/d_n).\]
Note that $Ac_2\sum_{n=c}^N\leq Ac_2 (p^{2c}(1-p^{-2}))^{-1}$, which goes to $0$ as $c\rightarrow\infty$ and the second term is \[O( M^{3/2})+O((\log M)M)+O(M^{1/2})\sum_{n=c}^N p^n=O(M^{3/2}).\]
Thus we obtain the desired estimate.
\end{proof}

\begin{lemma}\label{ss_H_main}
Notation as in \Cref{lem_count}. For any $c\in \bZ_{>0}$, we have
\[\sum_{m\in T_M-S_M}(\frac{A(p+2)}{2p}r_{0,1}(m)+ \frac{A}{2}r_{0,2}(m)+ \sum_{n=1}^c \frac{Ap^n}{2}(r_{n,1}(m)+r_{n,2}(m)))\leq \alpha \sum_{m\in T_M-S_M} g_P(m)+o(M^2),\]
where $\alpha<11/12$ is an absolute constant. 
\end{lemma}

\begin{proof}
Let $\theta_{n,i}$ denote the theta series attached to the lattice $L'_{n,i}$. We decompose $\theta_{n,i}=E_{n,i}+G_{n,i}$, where $G_{n,i}$ is a cusp form and $E_{n,i}$ is an Eisenstein series as in \S\ref{sec_Siegel-mass} and follow the proof of \Cref{lem_glo-bound}.

Let $E=\frac{A(p+2)}{2p}E_{0,1}+ \frac{A}{2}E_{0,2}+\sum_{n=1}^c \frac{Ap^n}{2}( E_{n,1}+ E_{n,2})$, $G=\frac{A(p+2)}{2p}G_{0,1}+ \frac{A}{2}G_{0,2}+\sum_{n=1}^c \frac{Ap^n}{2}( G_{n,1}+ G_{n,2})$. 

Note that $G$ is a weight $2$ cusp form and by Deligne's Weil bound, we have that its $m$-th Fourier coefficient $q_G(m)=O(m^{1/2+\epsilon})$. Hence the total contribution from the cusp form $G$ is $\sum_{m\in T_M-S_M}q_G(m)=O(M^{3/2+\epsilon})$. 

Let $q_{n,i}(m)$ and $q(m)$ denote the $m$-th Fourier coefficient of $E_{n,i}$ and $E$. Recall that for $p\nmid m$ for $m\in T_M$; by \Cref{compare_qL} and the fact that $|L'^\vee/L'|=p^2$, we have for any $n,i$
\[\frac{q_{n,i}(m)}{|q(m)_L|}\leq \frac{2p}{(p^2- 1)[L':L'_{n,i}]}, \text{ and }\frac{q_{0,1}(m)}{|q(m)_L|}\leq \frac{1}{p-1}.\]
Recall from \S\ref{def_smallL'} that $[L':L'_{n,1}]\geq p^{3n}$ and $[L':L'_{n,1}]\geq p^{3n+1}$; therefore,
\[\frac{q(m)}{|q_L(m)|}\leq \frac{A(p+2)}{2p}\cdot \frac{1}{p-1}+\frac{A}{2}\frac{2p}{(p^2-1)p}+\sum_{n=1}^c\frac{Ap^n}{2}\cdot \frac{2p}{p^2-1}(p^{-3n}+p^{-3n-1})\leq \frac{A}{p-1}\left(\frac{p+2}{2p}+\frac{p}{p^2-1}\right).\]
Take $\alpha=\frac{p+2}{2p}+\frac{p}{p^2-1}$, which is $<11/12$ when $p\geq 5$. We have the left hand side equals
\[\sum_{m\in T_M-S_M}(q(m)+q_G(m))\leq \sum_{m\in T_M-S_M}\frac{\alpha A}{p-1}|q_L(m)|+O(M^{3/2+\epsilon}),\]
which gives the desired estimate by the definition of $g_P(m)$.
\end{proof}

%SAVE: one might get a better decay result by working with $F_\infty(3)$. Indeed, the decay would be $A_n'=[A(p^n+\cdots +1)+a/p]$ (I did not check all the details, but just pretend that the leading term would survive) because the basis vectors in the rows corresponding to $F_\infty(3)$ are primitive in $\bL_{\cris,P}$ so we do not need to correct the integral structure. Then the ratio would be $\frac{p+1}{2p}+\frac{p}{p^2-1}$, which is still $>1$ for $p=3$. However, this ratio could possible be improved to be smaller than $1$ if we improve the trivial bound $2$ for the density of smaller lattices.

\begin{proof}[Proof of \Cref{supersingular} when $L=L_H$]
The set $S_M$ is constructed by \Cref{lem_setSM} and taking $\{P_i\}$ to contain all of (the finitely many) supersingular points in $C\cap(\cup_{p\nmid m}Z(m))$. Then the desired estimate follows from \Cref{ss_H_tail} and \Cref{ss_H_main} by taking $c$ such that $\epsilon<\frac{11}{12}-\alpha$.
\end{proof}

\begin{proof}[Proof of \Cref{thm_main}(2)]If $C$ is contained in $Z(m)$ with $m$ being a perfect square, then by applying suitable Hecke translates, we may assume that $C$ is contained in the product of modular curves and then the assertion is a special case of \cite[Proposition 7.3]{CO06}. Now for the rest of the proof, we may assume that $C$ is contained in some Hilbert modular surface and we will use $Z(m)$ to denote special divisors on the Hilbert modular surface.
Note that any point on $Z(m)$ corresponds to an abelian surface isogenous to the self-product of an elliptic curve. 
Thus we assume for contradiction that there are only finitely many points on $C\cap(\cup_{m\in T}Z(m))$ and take $\{P_i\}$ to be this finite set and apply \Cref{lem_setSM} to construct $S_M$. Since all $Z(m)$ are compact, it makes sense to consider $C.Z(m)$. We deduce a contradiction by \Cref{lem_glo-sum_asymp}, \Cref{supersingular}, and \Cref{newbert}.
\end{proof}

%%%%%%%%%%%%%%%%%%%%%%%%%%%%Section 8%%%%%%%%%%%%%%%$$$$$
\section{Proofs of \Cref{thm_main}(\ref{item_S}) and \Cref{thm_max}}\label{sec_pf_Sie}
In this section, we prove all of \Cref{thm_main} and \Cref{thm_max}. \S\ref{sec_S_prep} consists of results pertaining to squares represented by positive definite quadratic forms.\footnote{Recall that we must prove our curve intersects special divisors of the form $Z(D\ell^2)$ at infinitely many points. This involved dealing with squares represented by quadratic forms, and hence the Geometry-of-numbers arguments are more involved than in the Hilbert case.} In \S\ref{sec_S_pfss}, we prove \Cref{supersingular} by combining results proved in \S\S\ref{sec_global}, \ref{sec_decay_Sie}, and \ref{sec_S_prep}. Finally, we deal with the intersection multiplicities at non-supersingular points in \S\ref{sec_S_nonss} to finish the proof of the main theorem. 

We now set up notation that we will use for \S\ref{sec_pf_Sie}. For supersingular points $P$, recall that we defined $L'_{n,i}$ for superspecial points and $L'_n$ for supergeneric point in \S\ref{def_smallL'}. 
Let $\el_i, i=1,\dots, 5 $ denote the $i^{th}$ successive minimum of the quadratic form $Q'$ restricted to $L'_{n,1}$ or $L'_n$.
Let $P_n$ denote a rank two sublattice of $L'_{n,1}$ or $L'_n$ with minimal discriminant. 
Note that $\el_1\el_2\asymp d_n$, where $d_n$ denotes the root discriminant of $P_n$. Moreover, since $\cap_{n=0}^\infty L'_{n,i}=\{0\}=\cap_{n=0}^\infty L'_n$, we have $\el_1\rightarrow \infty$ as $n\rightarrow \infty$.

\subsection{Preparation}\label{sec_S_prep}
We need the following results to prove \Cref{supersingular}. Although \Cref{betterbound} is stated for the rank $5$ lattices $L'_{n,1},L'_n$, the proof does not use the assumption on rank and hence it holds for the lattices $L_n$ for ordinary points (notation as in \Cref{lem_lat_ord}) when $\rk_{\bZ}L_0=3$; see \S\ref{sec_S_nonss} for details.

\begin{lemma}\label{lem_el}
We have $\el_1\el_2\cdots \el_i\gg p^{(i-2)n}$ for $i\geq 3$.
\end{lemma}
\begin{proof}
Note that if we have two lattices $L_1\supset L_2$, then the successive minima of $L_2$ give upper bounds of that of $L_1$. Thus we may enlarge $L'_{n,i}, L'_n$ and prove the assertion for the enlarged lattices. 

We enlarge $L'_{n,i}, L'_n$ as follows. For $\ell\neq p$, we still require $L'_{n,i}\otimes \bZ_\ell=L'\otimes\bZ_\ell$ and $L'_{n}\otimes \bZ_\ell=L'\otimes\bZ_\ell$; at $p$, let $\Lambda_0$ denote the rank $3$ submodule of $L'\otimes \bZ_p$ which decays rapidly in the Decay Lemmas (\Cref{thm_decay,thm_decay_sg}), then we enlarge $L'_{n,1}, L'_n$ such that $L'_{n,1}\otimes \bZ_p=p^n \Lambda_0+L'\otimes \bZ_p$ and $L'_{n,1}\otimes \bZ_p=p^n \Lambda_0+L'\otimes \bZ_p$.

For the enlarged $L'_{n,1}, L'_n$, we have \[\el_j \ll p^n, j=1,\dots,5, \quad \el_1\el_2 \cdots \el_5 \asymp p^{3n},\] where the implied constants only depend on the lattice $L'$. Thus the assertion follows.
\end{proof}
%SAVE: However, it is unclear whether for enlarged $L'_n$, the total intersection is still $0$. Thus, it is not obvious whether we still have $\el_1\rightarrow \infty$ for enlarged lattices. Therefore, we only do the enlargement here and when we need to compute the theta series for the main part.

\begin{lemma}\label{betterbound}
Suppose that $d_n^2M=o(p^{2n})$ as $n\rightarrow\infty$. Then for any vector $v \in L'_{n,1}$ or $L'_n$ such that $Q(v) \leq M$, we have that $v \in P_n$ for $n\gg 1$. In particular, if $d_n \leq p^{n / 2}$, then for any vector $v\in L'_{n,1}$ or $L'_n$ such that $Q'(v) < p^{n - \epsilon}$ for some absolute constant $\epsilon>0$, we have that $v \in P_n$ for $n\gg 1$. (All the implicit constants here are independent of $n,M$.) 
\end{lemma}
\begin{proof}
Recall that $\el_1 \cdot \el_2 \asymp d_n$.
Thus, by \Cref{lem_el}, we have \[\el_1\el_2\el_3 \gg p^n, \quad \el_3 \gg p^{n}/d_n.\] %SAVE: for $P$ ordinary, we start from here.
In other words, for any vector $v$ linearly independent to $P_n$, we have $Q'(v)\geq \el_3^2\gg p^{2n}/d_n^2$. Then the first assertion follows. The second assertion follows directly from the first assertion by taking $M=p^{n-\epsilon}$.
\end{proof}

\begin{proposition}\label{completesquares}
Fix $D\in \bZ_{>0}$. Recall $r_{n,i}(m), r_n(m)$ from \Cref{lem_count}. Then we have the following two bounds: (here we state the results for $r_n(m)$ and the same results also hold for $r_{n,1}(m)$)
 \begin{enumerate}
\item $\displaystyle\sum_{m=D\ell^2, m\leq M, \ell \text{ prime }} r_n(m) = O_\epsilon\Big(\frac{M^{2+\epsilon}}{p^{2n}} + \frac{M^{3/2 +\epsilon}}{p^{n}} + M^{1 + \epsilon} \Big)$.

\item $\displaystyle\sum_{m=D\ell^2, m\leq M, \ell \text{ prime }} r_n(m)$ and $\displaystyle\sum_{\ell\leq M, \ell \text{ prime }} r_n(\ell)$ are both $O \Big(\frac{M^{5/2}}{p^{3n}} + \frac{M^{2}}{p^{2n}} + \frac{M^{3/2}}{p^{n}} + \frac{M}{d_n} + \frac{M^{1/2}}{\el_1}  \Big)$.
\end{enumerate}
\end{proposition}
\begin{proof}
In the proof, we work with $L'_n$ and everything holds true for $L'_{n,1}$, too.

We note that (2) is a trivial upper bound from a geometry-of-numbers argument. Indeed, both $\displaystyle\sum_{m=D\ell^2, m\leq M, \ell \text{ prime }} r_n(m)$ and $\displaystyle\sum_{\ell\leq M, \ell \text{ prime }} r_n(\ell)$ are no greater than $\displaystyle\sum_{m=1}^M r_n(m)$; we then obtain the desired bound by \cite[Lem.~4.2.1]{Ananth17} and \Cref{lem_el}. 

Now we prove (1). We may assume that there exists a vector $v_0\in L'_0$ such that $Q'(v_0)=D\ell_0^2$ for some prime $\ell_0$. Otherwise $r_n(m)=0$ for all $m=D\ell^2$ for any prime $\ell$. Let $e_1$ denote a primitive vector in $L'_n$ such that $e_1=p^k v_0$ for some $k\in \bZ_{\geq 0}$. By definition, $p^nL'_0\subset L'_n$ and thus $p^nv_0\in L'_n$. Therefore $k\leq n$. %By the definition of $e_1$, we have
%\begin{itemize}
%\item $Q'(e_1) = p^{2k}D\ell_0^2$ (note that $\ell_0$ is independent of $n$) and 
%\item For every $v \in L_n$, $[e_1,v]'=[p^k v_0, v]'\in p^k\bZ$, where $[-,-]'$ is the bilinear form associated to $Q'$.
%\end{itemize}
Since $e_1$ is primitive in $L'_n$, we extend it into a basis $\{e_1,e_2,\hdots e_5\}$ of $L'_n$. Let $\widetilde{L'}_n$ denote the sublattice of $L'_0$ spanned by $f_1:=v_0 = e_1/p^k,e_2,\hdots e_5$; since $\widetilde{L'}_n$ is a sublattice of $L'_0$, then $Q'|_{\widetilde{L'}_n}$ is still $\bZ$-valued. Since $Q'(v_0)=D\ell_0^2=:N$, then there exist vectors $f_2 \hdots f_5 \in (2N)^{-1}\widetilde{L'}_n$ such that $[f_i,f_1]'=0$ for $i\geq 2$, and $\Span_{\bZ}\{f_1,f_2,f_3,f_4,f_5\}\supset \widetilde{L'}_n$. 

Let $\widetilde{Q'}$ denote the restriction of $Q'\otimes \bQ$ to $\Span_{\bZ}\{f_2,f_3,f_4,f_5\}\subset L'_0\otimes \bQ$. By the definition of $f_i$, we have $\widetilde{Q'}$ is a $(2N)^{-2}\bZ$-valued quadratic form. %SAVE: by a direct computation, we can see indeed it is $(2N)^{-1}\bZ$-valued, but we do not need such a careful estimate.
Let $\widetilde{\el}_1,\dots, \widetilde{\el}_4$ denote the successive minima of $\Span_{\bZ}\{f_2,f_3,f_4,f_5\}$. Since $(2N)^{-1}\widetilde{L'}_n=(2N)^{-1}L'_n+(2N)^{-1}p^{-k}\bZ e_1$, then $\widetilde{\el}_1\cdots \widetilde{\el}_i\gg p^{(i-1)n-k}\geq p^{(i-2)n}$ for $i\geq 2$ (note that $k\leq n$ and $N$ is absorbed in the implicit constant as $N$ is independent of $n,k$). Then the standard geometry-of-numbers argument gives \[Y_n:=\#\{y\in \Span_{\bZ}\{f_2,f_3,f_4,f_5\} \mid \widetilde{Q'}(y)\leq M\}=O\Big(\frac{M^2}{p^{2n}} + \frac{M^{3/2}}{p^n} + M \Big).\]

On the other hand, on $\Span_{\bZ}\{f_1,f_2,f_3,f_4,f_5\}$, for $v=xf_1+y_2f_2+\cdots+ y_5f_5$, we have $Q'(v)=D\ell_0^2 x^2 +\widetilde{Q'}(v_y)$, where $v_y=y_2f_2+\cdots+ y_5f_5$. If $Q'(v)=D\ell^2\leq M$, then $Q'(v_y)\leq Q'(v)\leq M$ and $Q'(v_y)=D(\ell-\ell_0x)(\ell+\ell_0x)$. For a given $v_y$ with $Q'(v_y)\leq M$, there are at most $O_\epsilon(M^\epsilon)$ ways to factor $Q'(v_y)/D$ into two factors and thus there are at most $O_\epsilon(M^\epsilon)$ possible $x$ such that for $v=xf_1+v_y$, we have $Q'(v)=D\ell^2\leq M$ for some prime $\ell$. Since $L'_n\subset \Span_{\bZ}\{f_1,f_2,f_3,f_4,f_5\}$, then $\displaystyle\sum_{m=D\ell^2, m\leq M, \ell \text{ prime }} r_n(m)=O_\epsilon(M^\epsilon Y_n)$, which gives the (1) by the above bound for $Y_n$.
\end{proof}

\begin{proposition}\label{grh}
Fix $D\in \bZ_{>0}$.
The proportion of primes $\ell \leq (M/D)^{1/2}$ such that $D\ell^2$ is represented by the quadratic form restricted to $P_n$ goes to zero as $n\rightarrow \infty$. %SAVE: the number of such primes is $O_\epsilon \big( d_n^{\epsilon}(\frac{M^{1/2}}{d_n \log M} + M^{1/4} \log M) \big)$. %here $d^\epsilon$ comes from the number of $2$-torsions in the class group, which can be deduced from genus theory.
\end{proposition}
\begin{proof}
Let $R_n$ denote the imaginary quadratic ring with discriminant $-d_n^2$. 
%SAVE: thus for the rest, we just need to deal with the class group, not the narrow class group.
The class group of $R_n$ is in bijection with equivalence classes of binary quadratic forms of discriminant $-d_n^2$. Let $\fa$ denote the ideal corresponding to $Q'$ restricted to $P_n$. Recall that $\el_1\rightarrow\infty$ as $n\rightarrow \infty$. Thus for $n\gg 1$, we have that $\fa$ is not equivalent to any ideal whose norm is $D$, i.e., $(P_n,Q')$ does not represent $D$. Note that it suffices to deal with primes $\ell$ which are relatively prime to $Dd_n^2$. 

The correspondence between ideal classes and binary quadratic forms yields that the integer $D\ell^2$ is represented by $(P_n,Q')$ if and only if there exists an invertible ideal $\fb$ equivalent to $\fa$ with $\Nm \fb = D\ell^2$. 
%SAVE: the theory of quadratic forms indicates that the quadratic form constructed by an ideal  $\fa$ is given by $\Nm(\alpha)/\Nm(\fa)$ (quadratic form is given by writing $\alpha$ as a $\bZ$-linear combination with respect to a chosen basis of $\fa$. Then the ideal $\alpha \fa (\Nm \fa)^{-1}$ has norm being the integer represented by the quadratic form.
This implies that $\ell = \fc_1 \fc_2$ (i.e. the prime $\ell$ splits in $R_n$), and that $\fb = \fd\fc_1^2$ or $\fb = \fd\fc_2^2$, where $\fd$ is some ideal such that $\Nm \fd=D$ (the case $\fb = \fc_1 \fc_2$ is ruled out by the above discussion that $\fa$ and therefore $\fb$ is not equivalent to any ideal whose norm is $D$). In other words, $Q'$ restricted to $P_n$ represents $D\ell^2$ if and only if there exist some ideals $\fc,\fd$ such that $\Nm \fc=\ell, \Nm \fd=D$ and $\fc^2\fd$ is equivalent to $\fa$.

Let $C$ denote the equivalence classes of ideals $\fc$ such that $\fc^2$ is equivalent to $\fa\fd^{-1}$ for some $\fd$ with $\Nm \fd=D$. Since $D$ is fixed, then $C$ is a finite (independent of $n$) union of torsors for the $2$-torsion of the class group of $R_n$, when $C$ is nonempty. By the Genus theory, the cardinality of the two-torsions of the class group of $R_n$ is bounded by the number of prime divisors of $d_n^2$; thus, $\# C=O_\epsilon(d_n^{\epsilon})$.

 We finish the proof in two cases. 
 \begin{enumerate}
     \item If $d_n \leq (\log M)^2$, it follows by \cite{TZ} that the proportion of primes represented by the quadratic form associated to any ideal class $\fc$ is $1/d_n$ because $d_n\asymp$ the class number of $R_n$. Thus the total proportion of $\ell$ such that $D\ell^2$ is representable is $\#C/d_n=O_\epsilon(d_n^{\epsilon-1})$, which goes to $0$ as $d_n\rightarrow \infty$. 
     
     \item  If $d_n \geq (\log M)^2$, let $f_{\fc}$ denote the binary quadratic form associated to $\fc$. Then as in the proof of \cite[Claim 3.1.9]{Ananth17}, we have 
     \[\#\{\ell \mid \ell<(M/D)^{1/2} \text{ representable by }f_{\fc}\}\leq \#\{m \mid m<(M/D)^{1/2} \text{ representable by }f_{\fc}\}=O(M^{1/2}/d_n+M^{1/4}).\]
     Thus by the above discussion, 
     \[\#\{\ell \mid \exists v\in P_n, Q'(v)=D\ell^2\leq M\}=(\#C)O(M^{1/2}/d_n+M^{1/4})=o(M^{1/2}/\log M),\]
     which finishes the proof.\qedhere
 \end{enumerate}
\end{proof}

The following result gives a bound of Fourier coefficients of the cuspidal part of our theta series in terms of the discriminant of the quadratic lattice.

\begin{proposition}[Duke, Waibel]\label{boundcusp}
Let $S$ be a fixed finite set of primes.
Let $\theta$ be the theta series attached to a positive definite quadratic lattice of rank $5$ with discriminant $D_\theta$ such that all prime factors of $D_\theta$ lie in $S$. Write $\theta=E+G$, where $E$ is an Eisenstein series and $G$ is a cusp form. Then, there exist absolutely bounded positive constants $N_0$ and $C$  such that for all $m\in T$ (the set $T$ defined in \S\ref{def_setT}), the $m$-th Fourier coefficient $q_G(m)$ of $G$ satisfies that $q_G(m) \leq CD_\theta^{N_0} m^{1 + 1/4}$.
\end{proposition}

By \Cref{rmk_rk}, we have that $\disc L'_{n,i}, \disc L'_{n}$ are independent of $n,i$ away from $p$ and hence all the theta series attached to these lattices satisfy the assumption on $D_\theta$.

An analogous result of \Cref{boundcusp} was proved by Duke in the case of ternary quadratic forms. The main steps of his proof carry through in this case too, so we will be content with just sketching his proof.

%SAVE: in the proof below, we use that $G$ is a cusp form of level $\Gamma_0(N)$ and character $\chi$ with $N\mid D_\theta$ and $\chi$ quadratic (see for instance Hanke AWS notes Cor 2.3.2).

\begin{proof}
The proof of \cite[Lem.~1]{Duke} and the discussion on \cite[p.40]{Duke} apply to rank $5$ quadratic forms (with suitable modification of the power of $D_\theta$)
and we have that the Petersson norm of $G$ satisfies $||G||=O(D_\theta^{N_1})$ for some absolute constant $N_1$ (here we use the fact that the level $N_\theta$ of $G$ is $O(D_\theta)$.

Thus to obtain a bound for $q_G(m)$, we only need to bound the Fourier coefficients $a_j(m)$ for an orthonormal basis of the space of cusp forms of weight $5/2$ and level $N_\theta$ (with respect to certain quadratic character determined by $\theta$).
Now we apply \cite[Theorem 1]{Waibel}. %SAVE: the proof in Waibel uses Shimura lift and Deligne bound.
Using the notation there, we have that if $m=\ell$, then $t=\ell, v=1,w=1, (m,N_\theta)=O(1)$; if $m=D\ell^2$, then $t=D, v\asymp 1, w\asymp \ell, (m,N_\theta)=O(1)$. Thus $|a_j(m)|\ll_\epsilon m^{\frac{27}{28}+\epsilon}D_\theta^\epsilon$ for $m=\ell$ and $|a_j(m)|\ll_\epsilon m^{\frac{3}{4}+\epsilon}D_\theta^\epsilon$, which gives the desired bound once we combine with the above estimate of $||G||$.
\end{proof}

\subsection{Proof of \Cref{supersingular} in the Siegel case.}\label{sec_S_pfss}
Notation as in \S\ref{def_smallL'} and 
\Cref{lem_count}. For a supersingular point $P$, we estimate $\sum_{m\in T_M-S_M}r_{n,i}(m), \sum_{m\in T_M-S_M} r_n(m)$ with respect to different ranges of $n$.

\begin{defn}
Given absolute constants $\epsilon_0,\epsilon_1>0$ (we will choose $\epsilon_0,\epsilon_1$ in the proof of \Cref{supersingular}), the ranges of $n$ are defined as follows: 
\begin{itemize}
\item $n$ is \emph{small} if $n \leq \epsilon_0 \log_p M$, where $\epsilon_0$ is a constant independent of $M$.

\item $n$ is in the \emph{lower medium} range if $\epsilon_0 \log_p M < n \leq \frac{3}{4} \log_p M$

\item $n$ is in the \emph{upper medium} range if $\frac{3}{4} \log_p M < n \leq  (1 + \epsilon_1)\log_p M$.

\item $n$ is \emph{large} if $n > (1 + \epsilon_1)\log_p M$. \end{itemize}
\end{defn}

\begin{proof}[Proof \Cref{supersingular} for \Cref{thm_main}(1) and \Cref{rmk_realquad} with $p$ split in $F$]
For $m\in T_M$, we have $m=D\ell^2$, where $D$ is a non-zero quadratic residue mod $p$. Then by \S\ref{par_density_S}(2), any supergeneric point does not lie on $Z(m)$. Hence we will only consider $P$ superspecial.

Recall from \Cref{lem_glo-sum_asymp} that for any $S_M$ such that $\#S_M=o(\# T_M)$, we have \[ \sum_{m\in T_M-S_M}C.Z(m)\asymp \sum_{m\in T_M-S_M}g_P(m)\asymp M^2(\log M)^{-1}.\]
We will first prove that there exists $S_M$ such that $\#S_M=o(\# T_M)$ and the contribution from $n\geq \epsilon_0 \log_p M$ is $o(M^2/\log M)$.\\

\noindent\textbf{The lower medium range.} By \Cref{completesquares}(1),
\begin{align*}
   \sum_{m\in T_M-S_M}\sum_{n=\lceil\epsilon_0 \log_p M\rceil}^{[\frac{3}{4}\log_p M]}\frac{Ap^n}{2}(r_{n,1}(m)+r_{n,2}(m)) & \leq \sum_{m=D\ell^2, m\leq M}\sum_{n=\lceil\epsilon_0 \log_p M\rceil}^{[\frac{3}{4}\log_p M]} Ap^n r_{n,1}(m) \\
   & = A\sum_{n=\lceil\epsilon_0 \log_p M\rceil}^{[\frac{3}{4}\log_p M]} p^n O_\epsilon (M^{2+\epsilon}/p^{2n} + M^{3/2+\epsilon}/p^n + M^{1+\epsilon})\\
   &=O_\epsilon(M^{2+\epsilon-\epsilon_0} + M^{3/2 + \epsilon} \log M +  M^{7/4 + \epsilon}),
\end{align*}
which is $o(M^2/\log M)$ once we take $\epsilon<\min\{\epsilon_0,1/4\}$.

\noindent\textbf{The upper medium range.} 
We treat this part in two ways according to whether $d_{n_0} \leq M^{1/8}$, where $n_0 = \lceil\frac{3}{4}\log_p M\rceil$.

\begin{enumerate}
    \item If $d_{n_0} \geq M^{1/8}$, then we bound this part using geometry-of-numbers.\\ Since $L'_{n,1}\subset L'_{n_0,1}$ for all $n\geq n_0$, then by definition, $d_n\geq d_{n_0}\geq M^{1/8}$. By \Cref{completesquares}(2), we have that 
\begin{align*}
    \sum_{m\in T_M-S_M}\sum_{n=\lceil\frac{3}{4}\log_p M\rceil}^{[(1 + \epsilon_1)\log_p M]}\frac{Ap^n}{2}(r_{n,1}(m)+r_{n,2}(m)) & \leq  \sum_{n=\lceil\frac{3}{4}\log_p M\rceil}^{[(1 + \epsilon_1)\log_p M]}\sum_{m=D\ell^2, m\leq M} Ap^n r_{n,1}(m)\\
    &= O(M + M^{5/4} +  M^{3/2}\log M + M^{15/8 + \epsilon_1} + M^{3/2 + \epsilon_1}),
\end{align*}
    which is $o(M^2/\log M)$ once we take $\epsilon_1<1/8$.
    
    \item If $d_{n_0} < M^{1/8}$, we control this part by putting $m$'s in this range into $S_M$.\\
    More precisely, consider $R_M:=\{m\in T_M \mid \exists v\in L'_{n_0,1}, Q'(v)=m\}$. 
    By our assumption, $d_{n_0}^2M< M^{5/4}=o(p^{2n_0})$ and by \Cref{betterbound}, for $M\gg 1$, if $m\in R_M$, then $m$ is represented by $Q'\mid_{P_{n_0}}$, which is a binary quadratic form. Then by \Cref{grh} (note that $n_0\rightarrow \infty$ as $M\rightarrow \infty$),  $\#R_M=o(M^{1/2}/\log M)=o(\# T_M)$. Thus we may choose $S_M$ such that $S_M\supset R_M$ and then \[ \sum_{m\in T_M-S_M}\sum_{n=\lceil\frac{3}{4}\log_p M\rceil}^{[(1 + \epsilon_1)\log_p M]}\frac{Ap^n}{2}(r_{n,1}(m)+r_{n,2}(m))=0.\]
\end{enumerate}

\noindent\textbf{The large $n$'s.} Let $n_0=\lceil(1+\epsilon_1)\log_p M\rceil$ and let $R_M':=\{m\in T_M \mid \exists v\in L'_{n_0,1}, Q'(v)=m\}$. We will show that $\#R'_M=o(M^{1/2}/\log M)$ and thus we may choose $S_M$ such that $S_M\supset R'_M$ and then \[ \sum_{m\in T_M-S_M}\sum_{n=\lceil(1 + \epsilon_1)\log_p M\rceil}^{\infty}\frac{Ap^n}{2}(r_{n,1}(m)+r_{n,2}(m))=0.\]
We bound the size of $R'_M$ case by case depending on the size of $d_{n_0}, l(n_0)_1$.
\begin{itemize}
    \item Case (1): $d_{n_0} \leq M^{1/2+\epsilon_2}$ for some absolute constant  $\epsilon_2<\epsilon_1/2$.\\
    Then $d_{n_0}\leq M^{1/2+\epsilon_2}< p^{n_0/2}$ and $M<p^{n_0-\epsilon_1}$. By \Cref{betterbound}, for $M\gg 1$, if $m\in R'_M$, then $m$ is represented by $Q'\mid_{P_{n_0}}$. By \Cref{grh},  $\#R'_M=o(M^{1/2}/\log M)$.
    
    \item Case (2): $d_{n_0} > M^{1/2+\epsilon_2}$ for all $\epsilon_2<\epsilon_1/2$ and $l(n_0)_1>M^{\epsilon_3}$ for some absolute constant $\epsilon_3>0$.\\
    We have $\#R'_M\leq \#\{v\in L'_{n_0,1}\mid Q'(v)\in T_M\}$, which is $O(M^{1/2-\epsilon_1}+M^{1/2-\epsilon_2}+M^{1/2}/l(n_0)_1)=o(M^{1/2}/\log M)$ by \Cref{completesquares}(2).
    
    \item Case (3) $d_{n_0} > M^{1/2+\epsilon_2}$ for some $\epsilon_2<\epsilon_1/2$ and $l(n_0)_1\leq M^{\epsilon_3}$ for some $\epsilon_3<\epsilon_2$.\\
    Then $l(n_0)_2=d_{n_0}/l(n_0)_1>M^{1/2}$. In other words, any vector $v\in L'_{n_0,1}$ which is not a scalar multiple of the chosen vector $v_0$ of the minimum length has  $Q'_{n_0}(v)\leq l(n_0)_2^2>M$. Therefore, any $m\in R'_M$ has to be represented by the rank $1$ quadratic form spanned by $v_0$. As $M\rightarrow \infty$, we have $l(n_0)_1\rightarrow \infty$. Thus once $M$ is large enough such that $l(n_0)_1^2>D$, then this rank $1$ quadratic form would represent at most one element in $T_M$ and hence $\#R'_M=o(\# T_M)$.
\end{itemize}
\quad \\
In conclusion, taking $S_M=R_M\cup R'_M$, we have $\#S_M=o(\# T_M)$ and \[\sum_{m\in T_M-S_M}\sum_{n=\lceil\epsilon_0 \log_p M\rceil}^{\infty}\frac{Ap^n}{2}(r_{n,1}(m)+r_{n,2}(m))=o(M^2/\log M).\]

\noindent\textbf{The small $n$'s.} We follow the notation and the idea of the proof in \Cref{ss_H_main}.

We enlarge $L'_{n,1}$ as in the proof of \Cref{lem_el}; also let $w$ be the vector which decays very rapidly in the Decay Lemma for superspecial points, then we enlarge $L'_{n,2}$ such that $L'_{n,2}\otimes\bZ_p=L'_{n,1}\otimes \bZ_p+p^{n+1}\bZ_pw$.  Then $\disc L'_{n,i}\asymp p^{6n}$ with the implicit constant only depending on $P$. Note that \Cref{lem_count} still holds with the new definitions of $L'_{n,i}$.

Let
\[E=\frac{A(p+2)}{2p}E_{0,1}+ \frac{A}{2}E_{0,2}+\sum_{n=1}^{[\epsilon_0 \log_p M]} \frac{Ap^n}{2}( E_{n,1}+ E_{n,2}), G=\frac{A(p+2)}{2p}G_{0,1}+ \frac{A}{2}G_{0,2}+\sum_{n=1}^{[\epsilon_0 \log_p M]} \frac{Ap^n}{2}( G_{n,1}+ G_{n,2}).\]
Note that here the Eisenstein series $E$ and the cusp form $G$ depend on $M$.

Since $\disc L'_{n,i}=O(p^{6\epsilon_0 \log_p M})=O(M^{6\epsilon_0})$ for $n\leq \epsilon_0 \log_p M$, then by \Cref{boundcusp}, 
the $m$-th Fourier coefficient \[q_G(m)\ll  (M^{6\epsilon_0})^{N_0}m^{5/4}\sum_{n=0}^{[\epsilon_0 \log M]}p^n\ll M^{(6N_0+1)\epsilon_0} m^{5/4}\]
and $\sum_{m\in T_M-S_M} q_G(m)=O(M^{(6N_0+1)\epsilon_0+7/4})=o(M^2/\log M)$ once we take $\epsilon_0<(24N_0+4)^{-1}$.

The computation for the Eisenstein part is the same as in the proof of \Cref{ss_H_main}. More precisely, since $p\nmid m$, by \Cref{compare_qL}(1)(3), we have
\[\frac{q(m)}{|q_L(m)|}\leq \frac{A(p+2)}{2p}\cdot \frac{1}{p-1}+\frac{A}{2}\frac{2p}{(p^2-1)p}+\sum_{n=1}^{[\epsilon_0 \log M]}\frac{Ap^n}{2}\cdot \frac{2p}{p^2-1}(p^{-3n}+p^{-3n-1})\leq \frac{11}{12}\cdot\frac{A}{p-1}.\]
Thus we finish the proof by putting all parts together and using \Cref{lem_count}.
\end{proof}

\begin{proof}[Proof of \Cref{supersingular} for \Cref{rmk_realquad} when $p$ inert or ramified in $F$]
Comparing to the previous case when $p$ is split in $F$, the only differenced are (1) $Z(m)$ may contain supergeneric points; and (2) $p\mid m$ when $p$ is ramified and the estimate for $\frac{q(m)}{|q_L(m)|}$ may change. Nevertheless, except the computation of the Fourier coefficients of the Eisenstein series, all the bounds for $n$ in the medium and large range and for the cuspidal contribution for small $n$ remain valid for supergeneric points if we replace $L'_{n,1}$ by $L'_n$ for supergeneric points. Note that we only enlarge $L'_n, L'_{n,i}$ as in the proof of \Cref{lem_el} when we treat small $n$. Thus, to finish the proof, we compute the contribution from the Eisenstein part (for the enlarged lattices).

If $P$ supergeneric, let $\theta_{n}$ denote the theta series attached to the lattice $L'_{n}$. We decompose $\theta_{n}=E_{n}+G_{n}$, where $G_{n}$ is a cusp form and $E_{n}$ is an Eisenstein series. Let $\displaystyle E=\frac{A(p+1)}{p}E_0+ \sum_{n=1}^{[\epsilon_0 \log_p M]} Ap^n E_n$.

If $p$ is inert, i.e., $p\nmid D$, $p\nmid m$, then by \Cref{compare_qL}(2)(3) and the fact that $|(L'_0)^\vee/L'_0|=p^2$ for $P$ supergeneric, we have
\[\frac{q(m)}{|q_L(m)|}\leq \frac{A(p+1)}{p}\cdot \frac{2}{p^2-1}+ \sum_{n=1}^{[\epsilon_0 \log_p M]} Ap^n \frac{2}{(p^2-1)[L'_0:L'_n]}<\frac{A}{p-1}\left(\frac{2}{p}+\frac{2}{(p+1)(p^2-1)}\right)<\frac{11}{12}\cdot \frac{A}{p-1}.\]

If $p$ is ramified, we have $v_p(m)=1$. For $P$ superspecial, by \Cref{compare_qL}(1)(4), we have
\[\frac{q(m)}{|q_L(m)|}\leq \frac{A(p+2)}{2p}\cdot \frac{1}{p-1}+\frac{A}{2}\cdot \frac{4}{p^2-1}+\sum_{n=1}^{[\epsilon_0 \log M]}\frac{Ap^n}{2}\cdot \frac{2p}{p-1}\left(\frac{1}{p^{3n}}+\frac{1}{p^{3n+1}}\right)<\frac{A}{p-1}\left(\frac{p+2}{2p}+\frac{p+3}{p^2-1}\right),\]
which is $ <\frac{11}{12}\cdot\frac{A}{p-1}$ for $p\geq 7$. Similarly, for $P$ supergeneric, by \Cref{compare_qL}(2)(4), we have
\[\frac{q(m)}{|q_L(m)|}\leq \frac{A(p+1)}{p}\cdot \frac{2}{p^2-1}+ \sum_{n=1}^{[\epsilon_0 \log_p M]} Ap^n \frac{2}{(p-1)[L'_0:L'_n]}<\frac{A}{p-1}\left(\frac{2}{p}+\frac{2}{p^2-1}\right)<\frac{11}{12}\cdot \frac{A}{p-1}.\]
Thus we finish the proof.
\end{proof}

\begin{proof}[Proof of \Cref{supersingular} for \Cref{thm_max}]
Since every $m\in T_M$ in this case is a non-zero quadratic residue $\bmod p$, hence by \S\ref{par_density_S}(2), all supersingular points on $Z(m)$ are superspecial. The idea of the proof is similar to the case of \Cref{thm_main} (1).

By \Cref{lem_glo-sum_asymp}, we have $\sum_{m\in T_M-S_M}g_P(m)\asymp M^{5/2}(\log M)^{-1}$. We construct $S_M$ by large $n$. More precisely, we set $S_M=\{m\in T_M \mid \exists v\in L'_{n_0,1}, Q'(v)=m\}$, where $n_0=\lceil(1+\epsilon_1)\log_p M\rceil$. Then
\[\#S_M\leq \#\{v\in L'_{n_0,1}\mid Q'(v)\leq M\}=O(M^{5/2}/p^{3n_0}+M^2/p^{2n_0}+M^{3/2}/p^{n_0}+M/d_{n_0}+M^{1/2})=O(M^{1/2}+M/d_{n_0}),\]
which is $o(M/\log M)=o(\# T_M)$ if there exists an absolute constant $\epsilon>0$ such that $d_{n_0}\gg M^\epsilon$. If not, then by \Cref{betterbound}, we have that for $M\gg 1$, all $m\in S_M$ representable by the binary quadratic form $Q'|_{P_{n_0}}$. Since $d_{n_0}\rightarrow \infty$, the density of primes representable by $Q'|_{P_{n_0}}$ goes to zero, i.e., we still have $\#S_M=o(\# T_M)$. With this choice of $S_M$, we have
\[ \sum_{m\in T_M-S_M}\sum_{n=\lceil(1 + \epsilon_1)\log_p M\rceil}^{\infty}\frac{Ap^n}{2}(r_{n,1}(m)+r_{n,2}(m))=0.\]
For $n$ in the medium range,
\[ \sum_{m\in T_M-S_M}\sum_{n=\lceil \epsilon_0 \log_p M\rceil}^{\lceil(1 + \epsilon_1)\log_p M\rceil}\frac{Ap^n}{2}(r_{n,1}(m)+r_{n,2}(m))\ll \sum_{n=\lceil \epsilon_0 \log_p M\rceil}^{\lceil(1 + \epsilon_1)\log_p M\rceil} p^n \sum_{m\leq M} r_{n,1}(m)=o(M^{5/2}/\log M)\]
since $\sum_{m\leq M} r_{n,1}(m)=O(M^{5/2}/p^{3n}+M^2/p^{2n}+M^{3/2}/p^n+M)$. 
The estimate for small $n$'s is exactly as in the case for \Cref{thm_main}(1) above and thus we finish the proof.
\end{proof}

\subsection{Contribution from non-supersingular points and conclusions}\label{sec_S_nonss}
To finish the proof, we only need to show that $\sum_{m\in T_M-S_M} l_P(m)$ for non-supersingular points $P$ are $o(\sum_{m\in T_M-S_M}C.Z(m))$, which is $o(M^2/\log M)$ for \Cref{thm_main}(1) and \Cref{rmk_realquad} and is $o(M^{5/2}/\log M)$ for \Cref{thm_max}. We still use the notation in \Cref{lem_lat_ord} for ordinary points.

Recall that an abelian surface is ordinary, almost ordinary (i.e., its Newton polygon has slopes $0,1/2,1$), or supersingular. 
\begin{lemma}\label{nonss-rk1}
If $P$ is almost ordinary or if $P$ is ordinary with $\rk_{\bZ} L_0\neq 3$, then \[\sum_{m\in T_M-S_M} l_P(m)=o(\sum_{m\in T_M-S_M}C.Z(m)).\]
\end{lemma}
\begin{proof}
By the classification of endomorphism rings of char $p$ abelian surfaces (see for instance \cite[Thm.~1]{Tate71}), we see that if the abelian surface corresponding to $P$ has almost ordinary reduction, then its lattice of special endomorphisms has rank at most $1$. %Indeed, to have the rank of special endomorphism module to be $3$, $End A$ must be a quaternion algebra (over some field) and then all slopes must have even multiplicity.
On the other hand, if $P$ is ordinary, then $\rk_{\bZ} L_0$ is odd and hence $\rk_{\bZ} L_0=1$. In both cases, let $a_n x^2$ to denote the quadratic form with one variable given by $Q'$ restricted to the lattice of special endomorphisms of the abelian surface mod $t^n$. Since the lattice mod $t^{n+1}$ is a sublattice of the one mod $t^n$, we have $a_n\mid a_{n+1}$.

Since $C$ does not have any global special endomorphisms, we have $a_n\rightarrow\infty$ and hence $a_nx^2$ does not represent any element in $T_M\subset \{D\ell^2\mid \ell \text{ prime }\}$ or $T_M\subset \{\ell\mid \ell \text{ prime }\}$ once $n\gg 1$ (with then implicit constant only depending on $P$). 

Thus $\sum_{m\in T_M-S_M} l_P(m)=\sum_{m\in T_M-S_M} O(M^{1/2})=o(\sum_{m\in T_M-S_M}C.Z(m))$.
\end{proof}

Now it only remains to treat the case when $P$ is ordinary and $\rk_{\bZ}L_0=3$. We first construct $S_M$ for such $P$.

\begin{lemma}\label{ord_large_n}
Given $M$, set $n_0=\lceil(1+\epsilon_0)\log_p M\rceil$ and $S_M=\{m\in T_M\mid \exists v\in L_{n_0} \text{ with } Q'(v)=m\}$. Then $\#S_M=o(\# T_M)$.
\end{lemma}
\begin{proof}
By a geometry-of-numbers argument and \Cref{lem_lat_ord}, we have \[\#S_M\leq \{v\in L_{n_0}\mid Q'(v)\leq M\}=O(M^{3/2}/p^{n_0}+M/b_{n_0}+M^{1/2}/a_{n_0}),\] where $a_{n_0}$ is the minimal length of a non-zero vector in $L_{n_0}$ and $b_{n_0}$ is the minimal root discriminant of a rank $2$ sublattice in $L_{n_0}$. Since $C$ does not have any global special endomorphisms, we have $a_{n_0},b_{n_0}\rightarrow \infty$ as $M\rightarrow \infty$. Fix $0<\epsilon_1<\epsilon_0/4$. We prove the desired estimate by a case-by-case discussion based on the size of $a_{n_0},b_{n_0}$.
\begin{enumerate}
    \item $a_{n_0}<M^{\epsilon_1}$ and $b_{n_0}>M^{1/2+2\epsilon_1}$. Then we conclude as in the proof \Cref{supersingular} for \Cref{thm_main}(1) for large $n$ case (3). More precisely all $v\in L_{n_0}$ with $Q'(v)\leq M$ lie in a rank $1$ sublattice of $L_{n_0}$ and thus the total number of such $v$ is $o(\# T_M)$.
    %Then any vector in $L_{n_0}$ which is not a scalar multiple of the given vector $v_0$ with length $a_{n_0}$ has length at least $b_{n_0}/a_{n_0}> M^{1/2+\epsilon_1}$. Hence for primes squares $\leq M$ represented by $L_{n_0}$, it must be represented by the rank $1$ quadratic form (spanned by $v_0$). Note that for $M$ large enough, $a_{n_0}>1$ and hence no prime square is represented in this case.
    \item $a_{n_0}\geq M^{\epsilon_1}$ and $b_{n_0}>M^{1/2+2\epsilon_1}$.
    Then \[\#S_M=O(M^{3/2}/p^{n_0}+M/b_{n_0}+M^{1/2}/a_{n_0})=O(M^{1/2-\epsilon_1})=o(M^{1/2}/\log M).\]
    \item $b_{n_0}\leq M^{1/2+2\epsilon_1}$. Then $p^{n_0/2}=M^{1/2+\epsilon_0/2}\geq b_{n_0}$ and by \Cref{betterbound} (note the proof of this lemma applies to this case), for $M\gg 1$, if $m\in S_M$, then $m$ is represented by the binary quadratic form given by restricting $Q'$ to the rank $2$ sublattice in $L_{n_0}$ with minimal discriminant(=$b_{n_0}^2$). Since $b_{n_0}\rightarrow \infty$, then we conclude by \Cref{grh} for \Cref{thm_main}(1) and \Cref{rmk_realquad} and by the fact that the density of primes represented by such quadratic forms goes to $0$ for \Cref{thm_max}.\qedhere
\end{enumerate}
\end{proof}

Now we estimate the total local contribution at an ordinary point with $\rk_{\bZ}L_0=3$.
\begin{proposition}\label{ord-main}
Assume $P$ is ordinary with $\rk_{\bZ}L_0=3$.
After possible enlarging $S_M$ in \Cref{ord_large_n} (still with $\#S_M=o(\#T_M)$), we have $\sum_{m\in T_M-S_M} l_P(m)=o(\sum_{m\in T_M-S_M}C.Z(m))$.
\end{proposition}
\begin{proof}
Notation as in \Cref{lem_lat_ord}. By Lemmas \Cref{oleg}, \Cref{lem_lat_ord}, and \Cref{ord_large_n}, we have
\[\sum_{m\in T_M-S_M} l_P(m)=\sum_{m\in T_M-S_M}A(r_0(m)+\sum_{n=1}^{[(1+\epsilon_0)\log_p M]}(p^n-p^{n-1})r_n(m))\ll \sum_{n=0}^{[(1+\epsilon_0)\log_p M]} p^n \sum_{m\in T_M-S_M} r_n(m).\]

Notation as in \Cref{ord_large_n}. We have $\sum_{m=1}^M r_n(m)=O(M^{3/2}/p^n+M/b_n+M^{1/2}/a_n)$.

For \Cref{thm_max}, we have 
\[\sum_{m\in T_M-S_M} l_P(m)\ll \sum_{n=0}^{[(1+\epsilon_0)\log_p M]} p^n(M^{3/2}/p^n+M)=O(M^{2+\epsilon_0})=o(M^{5/2}/\log M),\]
when we take $\epsilon_0<1/2$.

For \Cref{thm_main}(1) and \Cref{rmk_realquad}, set $n_1=\lceil\frac{3}{4}\log_p M\rceil$.
First, \[\sum_{n=0}^{n_1} p^n \sum_{m\in T_M-S_M} r_n(m)\ll \sum_{n=0}^{n_1}p^n(M^{3/2}/p^n+M)=O(M^{7/4})=o(M^2/\log M).\]
Second, for $\sum_{n=n_1}^{[(1+\epsilon_0)\log_p M]} p^n \sum_{m\in T_M-S_M}r_n(m)$, we bound it by studying the following two cases separately.
\begin{enumerate}
    \item $b_{n_1}\geq M^{1/8}$. As in the first part, we have \[\sum_{n=n_1}^{[(1+\epsilon_0)\log_p M]}p^n \sum_{m\in T_M-S_M}r_n(m)\ll \sum_{n=n_1}^{[(1+\epsilon_0)\log_p M]} p^n(M^{3/2}/p^n+M/b_n+M^{1/2}),\]
    which is $O(M^{3/2}\log M + M^{2+\epsilon_0-1/8}+M^{3/2+\epsilon_0})=o(M^2/\log M)$ once we take $\epsilon_0<1/8$.
    \item $b_{n_1} < M^{1/8}$. We are going to enlarge $S_M$ to be $\{m\in T_M\mid \exists v\in L_{n_1} \text{ with } Q'(v)=m\}$. Since $b_{n_1}^2 M < M^{5/4}=o(p^{2n_1})$, then we conclude, as in the upper medium range Case (2) in the proof of \Cref{supersingular} for \Cref{thm_main}(1), by \Cref{betterbound} and \Cref{grh} that $\#S_M=o(\#T_M)$ and $\sum_{n=n_1}^{[(1+\epsilon_0)\log_p M]}p^n \sum_{m\in T_M-S_M}r_n(m)=0$.\qedhere
\end{enumerate}
\end{proof}

\begin{proof}[Proof of \Cref{thm_main}(1), \Cref{rmk_realquad}, and \Cref{thm_max}]
Assume for contradiction that there are only finitely many points on $C\cap(\cup_{m\in T}Z(m))$. Then we construct $S_M$ by taking the union of the $S_M$ in \Cref{supersingular} for supersingular points and that in \Cref{ord_large_n} and \Cref{ord-main} for ordinary points with $\rk_{\bZ}L_0=3$. Since it is a finite union, we still have $\#S_M=o(\#T_M)$. We deduce a contradiction by \Cref{lem_glo-sum_asymp}, \Cref{supersingular}, \Cref{nonss-rk1}, and \Cref{ord-main}.
\end{proof}

\appendix
\section{The decay lemma in the supergeneric case}

%\yunqing{make sure the notation is consistent with \S 3.}

Notation as in \S\ref{sec_decay_Hil}, especially \S\ref{sec_decay_statement}.
The goal of this section is to prove the following theorem:
\begin{theorem}[Decay Lemma]\label{thm_decay_sg}
At a supersingular point $P$ on $C$, there exists a rank $3$ submodule of the lattice $L'\otimes \bZ_p$ of special endomorphisms which decays rapidly (see \Cref{def_decay}). 
\end{theorem}

We will first prove the Siegel case, i.e., $L=L_S$ and then we deduce the Hilbert case ($L=L_H$) from the proof of the Siegel case. Since we have proved the above theorem when $P$ is superspecial in \S\S\ref{sec_decay_Hil}-\ref{sec_decay_Sie}, we will only treat the case when $P$ is supergeneric in this appendix and hence we only need to consider the case when $p$ is inert for the Hilbert case.
\subsection{Siegel case}\label{sec_Sie_sg_prep}
%In this section, we treat supersingular points which are not superspecial. The notion of decaying rapidly is the same as in \Cref{def_decay} with $A$ being the multiplicity of Hasse invariant (although $A$ is used for some other matrices).
First, we follow \S\ref{sec_F_Sie} and let $w_1:=pv_1 + v_3+ (c + \sigma^{-1}(c)) v_4$, $w_2:=\lambda p v_1 -\lambda v_3 + \lambda (c-\sigma^{-1}(c)) v_4 $, $w_3:= v_4$, $w_4:=-\sigma^{-1}(c)pv_1 + pv_2 - cv_3 - \sigma^{-1}(c)cv_4$, and $w_5:=v_5$, where recall that $\lambda\in \bZ_{p^2}^\times$ such that $\sigma(\lambda) = -\lambda$. Then a direct computation implies that \[\Span_{\bZ_p}\{w_1,\dots,w_5\}=\bL_{\cris,P}(W)^{\varphi=1}=L'\otimes \bZ_p.\]

By \S\ref{summary_Kisin} and using the universal unipotent $u$ in \S\ref{sec_F_Sie}, we have that the Frobenius on $\bL_\cris(W[[x,y,z]])$, with respect to the basis $\{w_i\}_{i=1}^5$, is $I + \frac{y}{p}\bbA + \frac{Q}{p}\bbB + x\bbC + z\bbD$, where $Q = xy + \frac{z^2}{4\epsilon}$,

\[{\bbA} = \begin{bmatrix}
0&c\lambda&\frac{1}{2}&\frac{c^2}{2}&0\\
\frac{c}{\lambda}&0&\frac{1}{2\lambda}&\frac{-c^2}{2\lambda}&0\\
-c^2&-\lambda c^2&-c&0&0\\
-1&\lambda&0&c&0\\
0&0&0&0&0\\
\end{bmatrix},\]

\[ \bbB = \begin{bmatrix}
\frac{-1}{2}&\frac{\lambda}{2}&0&\frac{c}{2}&0 \\
\frac{-1}{2\lambda}&\frac{1}{2}&0&\frac{c}{2\lambda}&0\\
c&-c\lambda&0&-c^2&0\\
0&0&0&0&0\\
0&0&0&0&0\\
\end{bmatrix},\]

\[ \bbC = \begin{bmatrix}
0&0&0&\frac{1}{2}&0\\
0&0&0&\frac{1 }{2\lambda}&0\\
-1&\lambda&0&0&0\\
0&0&0&0&0\\
0&0&0&0&0\\
\end{bmatrix},\] %\yunqing{for $C_{34}$, I changed the original $-c$ to $0$}
and 
\[ \bbD = \begin{bmatrix}
0&0&0&0&\frac{1}{2p}\\
0&0&0&0&\frac{1 }{2\lambda p}\\
0&0&0&0&\frac{-c}{p}\\
0&0&0&0&0\\
\frac{-1}{2\epsilon}&\frac{\lambda}{2\epsilon}&0&\frac{c}{2\epsilon}&0\\
\end{bmatrix}.\]
To lighten notation, let $\bbG$ denote the matrix $p \bbD \bbD^{(1)}$. We then have:
\[ \bbG = \begin{bmatrix}
\frac{-1}{4\epsilon}&\frac{-\lambda}{4\epsilon}&0&\frac{c^{(1)}}{4\epsilon}&0\\
\frac{-1}{4\lambda\epsilon}&\frac{-1}{4\epsilon}&0&\frac{c^{(1)}}{4\lambda\epsilon}&0\\
\frac{c}{2\epsilon}&\frac{c\lambda}{2\epsilon}&0&\frac{-cc^{(1)}}{2\epsilon}&0\\
0&0&0&0&0\\
0&0&0&0&\frac{-1}{2\epsilon}\\ %\yunqing{I changed the $G_{55}$ from the original $0$ to $-1/2\epsilon$}

\end{bmatrix}.\]
Here, as in \S\ref{sec_F_Sie}, we assume that $c\in W$ is the Teichmuller lift of its reduction $\bar{c}$ in $k$; since $P$ is supergeneric, we have that $\bar{c}\neq \sigma^2(\bar{c})$ and hence $c\pm \sigma(c)\in W^\times$.

As in \S\ref{sec_decay_statement}, we consider the formal curve $C=\Spf k[[t]]$ which is generically ordinary and specializes to $P$; the map $C\rightarrow \cM_k$
gives rise to a map $k[[x,y,z]] \rightarrow k[[t]]$, and we denote by $x(t),y(t),z(t)$ the images of $x,y,z$. The proof of \Cref{thm_decay_sg} in the case $y(t) = 0$ is much simpler than that of the case $y(t)\neq 0$, so in the rest of this appendix, we will present the proof of \Cref{thm_decay_sg} assuming $y(t)\neq 0$; similar ideas (with simpler case-by-case discussion) yield the proof when $y(t)=0$. Without loss of generality, we assume that $y(t) = t^{v_y}$ for some $v_y\in \bZ_{>0}$ and $z(t)$ has $t$-adic valuation $v_z$, and that the leading coefficient of $z(t)^2/4\epsilon$ (here we also use $\epsilon$ to denote the reduction of $\epsilon\in \bZ_p$ mod $p$) is $\alpha$ (set $v_z = \infty$ and $\alpha = 0$ if $z(t) = 0$). %\yunqing{I changed $z^2/2$ to $z^2/4\epsilon$}
Let $\tilde{\alpha} \in W$ be the Teichmuller lift of $\alpha$, and define $\bbE = \bbA + \tilde{\alpha} \bbB$. 

\begin{defn}
For $n\in \bZ_{>0}$, define $X_n$ to be the product $\prod_{i=0}^{n}X^{(i)}$, where $X$ stands for either $\bbA, \bbB$ or $\bbE$. 
\end{defn}

We record the following lemma which we will make crucial use of in the proof of the Decay Lemma. 
\begin{lemma}\label{lingen}
\begin{enumerate}
    \item The matrices $\bbA_n$ and $\bbA_n \bmod p$ have rank 1. Further, the first four rows of $\bbA_n$ are each scalar multiples of $R^{(n-1)}$ by $p$-adic units, where \[R = \begin{bmatrix} c+c^{(1)}&\lambda(c-c^{(1)})&1&-cc^{(1)}&0 \end{bmatrix}   .\]
    \item $\bbB_n$ and $\bbB_n \bmod p$ have rank 1. Further, the first three rows of $\bbB_n$ are each scalar multiples of $S^{(n-1)}$ by $p$-adic units, where \[S = \begin{bmatrix} 1&\lambda&0&-c^{(1)}&0 \end{bmatrix}   .\]
    \item $\bbE_n$ has rank $\leq 1$. Further, the rows of $\bbE_n$ are contained in the span of $R^{(n-1)}-\alpha^{(n)}S^{(n-1)}$. %\yunqing{originally it was $S^{(n-1)} + \alpha^{(n)}R^{(n-1)} $}. 
    If $n\geq 2$ and $\bbE_n \bmod p$ has rank 1 if and only if $c-c^{(2)}-\alpha^{(1)}$ is a $p$-adic unit. 
\end{enumerate}
\end{lemma}
\begin{proof}
As each of the above cases follow from similar straightforward inductive calculations, we will content ourselves with solely proving part (2). That the rows are spanned by $S^{(n-1)}$ follows directly from the fact that the row span of $\bbB_n$ is contained in the row span of $\bbB^{(n)}$, which is indeed equal to the span of $S^{(n-1)}$. 

We will prove that the first three rows of $\bbB_n$ are scalar multiples of $S^{(n-1)}$ by $p$-adic units by an inductive argument. The case of $n=1$ can be checked by a direct calculation. Observing that $\bbB_n = \bbB_{n-1} \bbB^{(n)}$, it suffices to prove that the $1\times5$ matrix $S^{(n-1)}\cdot \bbB^{(n)}$ has its first, second and fourth entries equalling $p$-adic units. This is also seen by a direct calculation, and the lemma follows.
\end{proof}

For the remainder of our paper, we will replace $c, \bar{c}$ by $\tilde{c},c$ respectively. We will also abuse notation by allowing $\lambda$ to denote both $\lambda$ and its mod $p$ reduction -- it will be clear from context what we refer to.

As in the proof of \Cref{thm_decay} in \S\ref{sec_decay_statement}, we consider the matrix $F_{\infty}$, which has a product expansion $\prod_{m=0}^{\infty} (I + \frac{y(t)}{p}\bbA + \frac{Q(t)}{p}\bbB + x(t)\bbC + z(t)\bbD)^{(m)}$ (note that as in \S\ref{sec_decay_statement}, the twist is with respect to $\sigma(x)=x^p, \sigma(y)=y^p, \sigma(z)=z^p$ and then we specialize to $x=x(t), y=y(t), z=z(t)$). Given any positive integer $n$, the coefficient of $1/p^n$ in $F_{\infty}$ having minimal $t$-adic valuation arise from products of (Frobenius-twisted) powers of $\bbA$, $\bbB$ and $\bbD$, with $\bbC$ contributing only larger order terms. We will therefore ignore $\bbC$ while proving our decay results. 

We will need to consider the following cases: 

\subsection{Case 1: $v_y > 2v_z$}

\begin{proof}[Proof of \Cref{thm_decay_sg}]
We will use part (2) of Lemma \ref{lingen} to prove that  $\Span_{\bZ_p}\{w_1,w_2,w_4\}$ decays rapidly. As $v_y > 2v_z$, the $t$-adic valuation of $Q=xy + z^2/(4\epsilon)$ equals $2v_z$, which is strictly smaller than that of $y$.
Moreover, by \Cref{eqn-non-ord}, the equation of the non-ordinary locus is $(x+a)y+\frac{z^2}{4\epsilon}$ and hence $A=2v_z$, where $A$ in \Cref{def_decay} is the $t$-adic valuation of the equation. Further, the $t$-adic valuation of $z^{1 + p}$ equals $(p+1)v_z$, which is strictly greater than $2v_z=v_t(Q)$. Therefore, the term in the coefficient of $1/p^{n+1}$ with minimum $t$-adic valuation equals $(Q\bbB) \cdots (Q\bbB)^{(n)}$. In order to prove that the property DR (in the statement of Proposition \ref{prop_decay}) holds, by \Cref{lingen}(2), it suffices to prove that the kernel of the matrix $S^{(n-1)} \bmod p$ does not contain any non-zero $\bF_p$-linear combinations of $w_1,w_2,w_4 \bmod p$. Indeed, this follows directly from the fact that $c$ and therefore all its Frobenius twists do not lie in  $\bF_{p^2}$. Thus \Cref{thm_decay_sg} follows from the property DR by the same argument that \Cref{prop_decay} implies \Cref{thm_decay} in \S\ref{sec_decay_statement}.
\end{proof}

\subsection{Case 2: $v_y < 2v_z$}\label{Case2sg}

\begin{proof}[Proof of \Cref{thm_decay_sg}]
It suffices to prove that $\Span_{\bZ_p}\{w_1,w_2,w_3\}$ decays rapidly. As $v_y < 2v_z$, the term in the coefficient of $1/p^{n+1}$ with minimum $t$-adic valuation equals $(y\bbA)\hdots (y\bbA)^{(n)}$. Note also that $A$ in \Cref{def_decay} is $v_t((x+a)y+\frac{z^2}{4\epsilon})=v_y$. As in Case 1, it suffices to prove that the kernel of $ \bbA \cdots \bbA^{(n)} \bmod p$ contains no non-zero $\bF_p$-linear combinations of $w_1,w_2,w_3 \bmod p$. 

To that end, let $w$ be some $\bF_p$-linear combination of these three vectors, which were in the kernel of $ \bbA \cdots \bbA^{(n)}$. By \Cref{lingen}(1), the (first four) rows of $\bbA \cdots \bbA^{(n)}$ are unit-scalar multiples of $R^{(n-1)}$. Therefore, $R^{(n-1)} \cdot w = 0$. As $w$ is Frobenius-invariant, this is equivalent to $R \cdot w = 0$. 

The existence of $w$ implies that there exist $s_1,s_2,s_3 \in \bF_p$ such that $s_1(c + c^{(1)}) + s_2 \lambda(c - c^{(1)})  + s_3 = 0$. If either $s_1 =0$ or $s_2 = 0$, it follows that either  $c +c^{(1)} \in \bF_p$ or $\lambda(c - c^{(1)}) \in \bF_p$. Either case would imply that $c = c^{(2)}$, which is a contradiction as we have assumed that $c \notin \bF_{p^2}$. Therefore, we assume that $s_1 = -1$. We now have 
$$c + c^{(1)} = \lambda s_2(c - c^{(1)}) + s_3,$$
and therefore applying $\sigma$ to the above equation, we also have
$$c^{(1)} + c^{(2)} = \lambda s_2(c^{(2)} - c^{(1)}) + s_3.$$

Subtracting the two equations yields 

$$c - c^{(2)} = \lambda s_2(c - c^{(2)}).$$
This is a contradiction, as $\lambda \notin \bF_p$. 
Therefore, such $w$ could not have been in the kernel of $\bbA \cdots \bbA^{(n)} \bmod p$, whence it follows that every primitive vector in $\Span_{\bZ_p}\{w_1,w_2,w_3\}$ satisfies the condition DR in the statement of Proposition \ref{prop_decay}, as required. Therefore, \Cref{thm_decay_sg} follow in this case.
\end{proof}

\subsection{Case 3: $v_y = 2v_z$}\label{case3sg}
In this case, it follows that $z(t) \neq 0$, and thus $\alpha \neq 0$. 

\subsection*{Subcase 1: $\alpha^{(1)} -c + c^{(2)} \neq 0$}\label{subcase3.1sg}
\begin{proof}[Proof of \Cref{thm_decay_sg}]
We will prove that $\Span_{\bZ_p}\{w_i,w_3,w_5\}$ decays rapidly where $i$ is either $1$ or $2$. In this case, we have $A\geq 2v_z$. 
The conditions imposed on $v_y,v_z$ imply that the term with minimal $t$-adic valuation in the coefficient of $1/p^{n+1}$ is $$\bbE\bbE^{(1)} \cdots \bbE^{(n)}t^{2v_z(1 + p + \cdots +p^n)} + \bbE\bbE^{(1)} \cdots \bbE^{(n-1)}(p\bbD)^{(n)}t^{2v_z(1 + p + \cdots +p^{n-1}) + v_zp^n}.$$ Note that the first term in the sum has its last column equalling zero, and the second term has its first four columns equalling zero. 

For brevity, we will henceforth denote a Frobenius twist with a superscript of $'$. We claim that at least one of $c + c' -\alpha'$ and $\lambda(c - c' - \alpha')$ does not lie in $\bF_p$. Indeed, the first element being in $\bF_p$ implies that $c' + c'' - \alpha'' = c + c' - \alpha'$, consequently $-(\alpha' - \alpha '') = c'' -c$. Similarly, $\lambda(c - c' - \alpha')$ being an element of $\bF_p$ implies that $\alpha' + \alpha'' = c - c''$. Therefore, both elements being in $\bF_p$ implies that $\alpha' + \alpha'' = \alpha'-\alpha''$, which in turn implies that $\alpha = 0$, a contradiction. 

Without loss of generality, we will assume that $c + c' - \alpha' \notin \bF_p$, and prove that $\Span_{\bZ_p}\{w_1,w_3,w_5\}$ decays rapidly (if $\lambda(c-c'-\alpha') \notin \bF_p$, then an identical argument would yield that  $\Span_{\bZ_p}\{w_2,w_3,w_5\}$ decays rapidly). We first prove that every primitive vector in $\Span_{\bZ_p}\{w_1,w_3\}$ decays rapidly. Indeed, let $w$ be any primitive vector. In order to prove that $w$ decays rapidly, it suffices to prove that $w \bmod p$ is not in the kernel of $\bbE\bbE^{(1)} \cdots \bbE^{(n)} \bmod p$. However, as the 4th row of this matrix is a unit-multiple of $R^{(n-1)}-\alpha^{(n)}S^{(n-1)}$ by Lemma \ref{lingen} (3), it suffices to prove that $w \bmod p$ is not in the kernel of the $1 \times 5$ matrix $R^{(n-1)}-\alpha^{(n)}S^{(n-1)}$. As $w \bmod p$ is $\bF_p$-rational, this is equivalent to asking that $w \bmod p$ is not in the kernel of $R-\alpha' S \bmod p$. This follows from the fact that $c + c' - \alpha' \notin \bF_p$. 

We now show that $w_5$ decays rapidly. Indeed, the last column of $\bbE\bbE^{(1)} \cdots \bbE^{(n-1)}(p\bbD)^{(n)}$ has $(R^{(n-2)}-\alpha^{(n-1)}S^{(n-2)}) \cdot v^{(n)}$ as its fourth entry (up to a unit), where $v$ is the last column of $p\bbD$. It suffices to prove that $(R^{(n-2)}-\alpha^{(n-1)}S^{(n-2)}) \cdot v^{(n)} \neq 0 \bmod p$, equivalently $(R-\alpha' S) \cdot v^{(2)} \neq 0 \bmod p$. A direct computation shows that this element equals $-(\alpha' - c + c'')$, which we have assumed to be not zero. Therefore, it follows that $w_5$ decays rapidly. 

Let $w$ denote a primitive vector in the span of $w_1,w_3$. Consider a $\bZ_p$-linear combination $a w + b w_5$, where either $a$ or $b$ is a $p$-adic unit. The only way for $aw + bw_5$ to not decay rapidly is if the $t$-adic valuation of the coefficient of $1/p^{n+1}$ in $F_{\infty}w$ equalled the $t$-adic valuation of the coefficient of $1/p^{m+1}$ in $F_{\infty}w_5$. However, the former equals $t^{2v_z(1 + p + \hdots + p^n)}$ and the latter equals $t^{2v_z(1 + p + \hdots + p^{m-1}) + v_zp^m}$. These two quantities are clearly never equal, thereby establishing the required decay as in Case 1 in \S\ref{sec_decayHsplit}.
\end{proof}

\subsection*{Subcase 2: $\alpha^{(1)}- c + c^{(2)} = 0$}
In this case, we see that $\alpha = c^{(-1)} - c^{(1)}$, and thus by \Cref{eqn-non-ord}, the local equation of the non ordinary locus is given by the equation $Q(t) - \alpha y(t)$. In particular, we see that $H(t) := Q(t) - \alpha y(t) = Q(t) - \alpha t^{v_y}\neq 0$. We will now express $$\Frob = (I + \frac{y(t)}{p}\bbE + \frac{H(t)}{p}\bbB + x(t) \bbC + z(t)\bbD)\circ \sigma,$$ where $\bbE, \bbB,\bbC,\bbD$ are defined in \S\ref{sec_Sie_sg_prep}. We will establish the required decay by considering the fourth row of $F_{\infty}$. As mentioned above, we omit $\bbC$ and powers of $x(t)$ in this analysis, as there are no negative powers of $p$ in the entries of $\bbC$. 

We need the following lemma: 

\begin{lemma}\label{newlin}
Consider all products of the form $W_0W_1W_2 \hdots W_n$, where $W_i$ is the $i^{th}$ Frobenius twist of $\bbE,\bbB$ or $\bbD$. The only products which have a non-zero fourth row are those with the following properties: 
\begin{enumerate}
\item $W_0 = \bbE$. 

\item Suppose that $W_i, W_j \neq $ twists of $\bbD$ but $W_{i+1}, \dots, W_{j-1} = $ twists of $\bbD$ for $1 \leq i < j \leq n$. Then, $j - i$ has to be odd. Equivalently, any maximal consecutive subsequence consisting exclusively of Frobenius twists of $\bbD$ has to have even length, unless the subsequence is terminates with $W_n$. 

\item Apart from $i = 0$, the only possible $j \leq n$ such that $W_j = \bbE^{(j)}$  is $j = n$. 
\end{enumerate}
Further, a product that satisfies the above properties does indeed have a non-zero fourth row. 

Finally, consider the length four vector given by the first four rows of the last column of the product. This vector is nonzero if and only if $W_n = \bbD^{(n)}$ and the number of occurrences of twists of $\bbD$ is odd. If this is the case, the first four columns of the product are all zero. 
\end{lemma}
\begin{proof}
(1)(2) are clear. Part (3) follows from a direct computation. We will illustrate this computation in the particular case
\begin{equation}\label{rowformm}
\bbE\prod_{i = 1}^m \bbB^{(i)} \prod_{j = m+1}^{m + 2e}\bbD^{(j)}\bbE^{(m + 2e + 1)} = \frac{1}{p^{e}}\bbE\prod_{i = 1}^m \bbB^{(i)}\prod_{j=m+1}^{m+e}\bbG^{(2j-m-1)}\bbE^{(m+2e+1)}.
\end{equation}
It will follow from explicitly computing the product that multiplying it by either $\bbE^{(m + 2e + 2)}$, $\bbD^{(m + 2e + 2)}$ or $\bbB^{(m + 2e + 2)}$ yields the zero matrix. The other cases (where the $W_i$ are other choices of $\bbB,\bbD$) are entirely analogous, and the same computation goes through. 

An easy inductive argument shows that, the product $\displaystyle \prod_{i = 1}^m \bbB^{(i)} \prod_{j = m+1}^{m + 2e}\bbD^{(j)} = \frac{1}{p^e}\prod_{i = 1}^m \bbB^{(i)} \prod_{j = m+1}^{m+e} \bbG^{(2j-m-1)}$ equals 

\[ \frac{(-1)^{m+e}}{(2p\epsilon)^e}\begin{bmatrix}
\frac{1}{2}&\frac{(-1)^{m+1}\lambda}{2}&0&\frac{-\tilde{c}^{(2e + m)}}{2}&0\\
\frac{-1}{2\lambda}&\frac{(-1)^m}{2}&0&\frac{ \tilde{c}^{(2e + m)}}{2\lambda}&0\\
-\tilde{c}^{(1)}&(-1)^m \lambda \tilde{c}^{(1)}&0&\tilde{c}^{(1)}\tilde{c}^{(2e + m)}&0\\
0&0&0&0&0\\
0&0&0&0&*\\
\end{bmatrix}.\]

Multiplying this matrix on the left by the fourth row of $\bbE$ and on the right by $\bbE^{(m + 2e + 1)}$ and using the relation $\alpha ^{(1)} = c - c^{(2)}$ yields 
\[ \frac{(-1)^{m+e}}{(2\epsilon p)^e}\begin{bmatrix} 
-\tilde{c}^{(2e + m + 1)} - \tilde{c}^{(2e + m + 2)}&(-1)^{m}\lambda(\tilde{c}^{(m+2e+1)}-\tilde{c}^{(m+2e+2)} )&-1&\tilde{c}^{(m + 2e+ 1)}\tilde{c}^{( m  +2e+2)}&0\\
\end{bmatrix}.\]
Note that the product (not just the fourth row) matrix has rank one, and so every other row must be some multiple of this row. In order to show that the product in \eqref{rowformm} multiplied by $W_{m + 2e + 2}$ (where $W_{m + 2e + 2}$ is either $\bbB^{(m + 2e + 2)},\bbD^{(m + 2e + 2)}$ or $\bbE^{(m + 2e + 2)}$) is zero, it suffices to prove that the fourth row of this product is zero. This can be checked by direct computation. 
\end{proof}

We record the fourth row of the product in \eqref{rowformm} (scaled by $p^e$) for future use: 
\begin{lemma}\label{rowform}
The fourth row of the product $\displaystyle p^e \bbE\prod_{i = 1}^m \bbB^{(i)} \prod_{j = m+1}^{m + 2e}\bbD^{(j)}\bbE^{(m + 2e + 1)}$ equals 
\[ \frac{(-1)^{m+e}}{(2\epsilon)^e}\begin{bmatrix} 
-\tilde{c}^{(2e + m + 1)} - \tilde{c}^{(2e + m + 2)}&(-1)^{m}\lambda(\tilde{c}^{(m+2e+1)}-\tilde{c}^{(m+2e+2)} )&-1&\tilde{c}^{(m + 2e+ 1)}\tilde{c}^{( m  +2e+2)}&0\\
\end{bmatrix}.\]
\end{lemma}

 Let $H(t) = \eta t^{A} + t^{A+1}(z_2(t))$, where $\eta \in k^{\times}$. Let $\tilde{\eta}$ denote the Teichmuller lift of $\eta$. 

We break this final case into two cases. 

\subsection*{Subsubcase 1: $v_z(2p^{2e} - p^{2e -1} + 1) < A < v_z(2p^{2e + 2} - p^{2e + 1} + 1)$}

In this case, we will prove that $\Span_{\bZ_p}\{w_1,w_2,w_3\}$ decays rapidly. 

\begin{lemma}\label{sgtopleftdecay}
The term with minimal $t$-adic valuation in the coefficient of $1/p^{n+1}$ in the fourth row of (the top-left $4\times 4$ block of) $F_{\infty}$ is 
$$p^e\bbE\bbB^{(1)} \cdots \bbB^{(n-e -1)}\bbD^{(n-e)}\cdots \bbD^{(n +e -1)} \bbE^{(n + e)}.$$
\end{lemma}
\begin{proof}
Note that $F_{\infty}$ is composed of sums of products as in \Cref{newlin}, where each $W_i$ is multiplied by: 
\begin{itemize}
\item $y(t)^{(i)} = t^{2v_zp^i}$ if $W_i = \bbE^{(i)}$;

\item $h(t)^{(i)} = \tilde{\eta}^{(i)}t^{Ap^i} + \cdots$ if $W_i = \bbB^{(i)}$;

\item $z(t)^{(i)} = \beta^{(i)}t^{v_zp^i}$ if $W_i = \bbD^{(i)}$, where $\beta$ is the leading coefficient in $z=\beta t^{v_z} +\cdots$.
\end{itemize}

Consider products as in \Cref{newlin}. As we are looking for matrices where the first four columns are not all zero, it follows that the number of occurrences of twists of $\bbD$ must be even. Therefore, consider a product with $n_1$ occurrences of twists of either $\bbE$ or $\bbB$, and $2n_2$ occurrences of twists of $\bbD$. The fourth row of such a product would have a $p$-adic valuation of $-(n_1 + n_2)$, and hence we assume that $n + 1 = n_1 + n_2$. 

It is clear that the $t$-adic valuation of the expression is minimized if the first and last matrices in the product are both twists of $\bbE$. Indeed, the $t$-adic valuation is minimized when $W_i = \bbE^{(i)}$ for as many $i$ as possible, and \Cref{newlin} implies that this happens when the first and last matrices are both twists of $\bbE$. As the $t$-adic valuation of $H(t)$ is strictly greater than that of $z(t)$, it follows that products of the form $\bbE \bbB^{(1)} \cdots \bbB^{(n_1 -2)} \bbD^{n_1 -1} \cdots \bbD^{(n_1 + 2n_2 -2)}\bbE^{(n_1 + 2n_2 -1)}$ contain the term with minimal $t$-adic valuation. 

As in Case 4 in \S\ref{sec_decayHsplit}, a convexity argument yields that the $t$-adic valuation is minimized exactly for the product listed in the statement of this result, thereby concluding the proof.
\end{proof}

\begin{proof}[Proof of \Cref{thm_decay_sg} in this case]
We will prove that $\Span_{\bZ_p}\{w_1,w_2,w_3\}$ decays rapidly.
Let $R_e$ denote the mod $p$ reduction of the row detailed in \Cref{rowform}. It suffices to show that there is no $\bF_p$-linear combination of the first three entries of $R_e$ which evaluates to zero. This is equivalent to prove that the elements $c + c^{(1)}$, $\lambda(c-c^{(1)})$ and $1$ are $\bF_p$-linearly independent. This fact has already been established in Case 2 in \S\ref{Case2sg}. The result follows. 
\end{proof}

\subsection*{Subsubcase 2: $v_z(2p^{2e} - p^{2e -1} + 1) = A$}
\begin{lemma}\label{sgtoprightdecay}
\begin{enumerate}
\item There are two terms with minimal $t$-adic valuation in the coefficient of $1/p^{n+1}$ in the fourth row of (the top-left $4\times 4$ block of) $F_{\infty}$. They are: 
\begin{itemize}
\item $p^e\bbE\bbB^{(1)} \cdots \bbB^{(n-e -1)}\bbD^{(n-e)}\cdots \bbD^{(n +e -1)} \bbE^{(n + e)}\eta^{(1 + p + \cdots + p^{n-e-1})}\beta^{p^{n-e} + \cdots + p^{n+e-1}}$; and

\item $p^{e-1}\bbE\bbB^{(1)} \cdots \bbB^{(n-e)}\bbD^{(n-e + 1)}\cdots \bbD^{(n +e -2)} \bbE^{(n + e -1)}\eta^{(1 + p + \cdots + p^{n-e})}\beta^{p^{n-e+1} + \cdots + p^{n+e-2}}$.
\end{itemize}

\item The term with minimal $t$-adic valuation in the coefficient of $1/p^{n+1}$ in the fourth row of the last column of $F_{\infty}$ is $$ \bbE\bbB^{(1)} \cdots \bbB^{(n-e -1)}\bbD^{(n-e)}\cdots \bbD^{(n +e)}.$$
\end{enumerate}
\end{lemma}
\begin{proof}
The proof goes along the same lines as that of \Cref{sgtopleftdecay}, and so we will not spell out the details.
\end{proof}

\begin{proof}[Proof of the final case of \Cref{thm_decay_sg}]
We will show that there exists a rank $2$ submodule $W\subset \Span_{\bZ_p}\{ w_1,w_2,w_3\}$ such that $W+\bZ_pw_5$ decays rapidly.

The term with minimal $t$-adic valuation in the coefficient of $1/p^{n+1}$ of the fourth row of the top-left $4\times 4$ block of $F_{\infty}$ is the sum of the corresponding parts in the two matrices in \Cref{sgtoprightdecay}(1). The $t$-adic valuation of this sum is $v_z(2 + p^{n-e} + \cdots + p^{n+e-1} + 2p^{n + e}) + A(p + \cdots + p^{n-e-1})$, which is $\leq A(1+\cdots + p^n)$.

In order to prove that $W$ decays rapidly with the above $t$-adic valuation, it suffices to show that the
 kernel of the sum of the fourth rows of the two matrices in Lemma \ref{sgtoprightdecay}(1) mod $p \cap W \bmod p$ is $\{0\}$ for a suitable choice of $W$. To that end, let $\mu = \frac{\beta^{p^{n - e} +  p^{n+e-1}}}{2\epsilon}$ and $\nu = \eta^{p^{n-e}}$. Then, the sum of the two rows with equal minimal $t$-adic valuation mod $p$ is a scalar multiple\footnote{The scalar multiple factor is $(-1)^{n-1}\frac{\eta^{1+p+\hdots p^{n-e-1}}\beta^{p^{n-e+1} + \hdots p^{n+e-2}}}{(2\epsilon)^{e-1}}$.} of
\[ \begin{bmatrix} 
-\alpha_1&(-1)^{n-e-1}\alpha_2&-\alpha_3&\alpha_4&0\\
\end{bmatrix}.\]

\Cref{rowform} yields that, %the multiple is $\beta^{p^{n-e+1}+\cdots +p^{n+e-2}}\eta^{1+\cdots + p^{n-e-1}}$
\begin{itemize}
\item $\alpha_1 = \mu (c^{(n + e)} + c^{(n + e +1)}) + \nu(c^{(n + e-1)} + c^{(n + e)})$;

\item $\alpha_2 = \lambda\big(\mu (c^{(n + e)} - c^{(n + e+1)}) - \nu(c^{(n + e-1)} - c^{(n + e)}) \big)$;

\item $\alpha_3 = \mu + \nu$.

\end{itemize}
We now prove that depending on the values of $\mu$ and $\nu$, we may take $W$ to be either $\Span_{\bZ_p}\{w_1,w_2\}$, $\Span_{\bZ_p}\{w_1,w_3\}$ or $\Span_{\bZ_p}\{w_2,w_3\}$. To further lighten notation, let $\delta_1 = (\mu + \nu)c^{(n+e)} $ and $\delta_2 = \mu c^{(n+e+1)} + \nu c^{(n + e -1)}$. 

\begin{enumerate}
\item Suppose $\nu + \mu = 0$. We will show that $\Span_{\bZ_p}\{w_1,w_2\}$ decays rapidly. It suffices to prove that there is no non-trivial $\bF_p$-linear relation between $\alpha_1 = \mu (c^{(n + e-1)} - c^{(n + e+1)})$ and $\alpha_2 = \lambda\mu(c^{(n + e-1)} - c^{(n + e+1)})$. However, this follows directly from the facts that $\lambda \notin \bF_p$ and that $c^{(n + e+1)} \neq c^{(n + e-1)}$.

\item Suppose that $\delta_1 \neq \pm \delta_2 $. We will show that either $\Span_{\bZ_p}\{w_1,w_2\}$ or $\Span_{\bZ_p}\{w_1,w_3\}$ decays rapidly. Indeed, the former happens when $\alpha_1 = \delta_1 + \delta_2$ and $\alpha_2 = \lambda(\delta_1 - \delta_2)$ are not $\bF_p$ multiples of each other. Therefore, suppose that they were. Then we have that for some $l \in \bF_p$,
$\delta_1 + \delta_2 = l\lambda(\delta_1 - \delta_2).$
%and 
%\begin{equation*}
%\delta'_1 + \delta'_2 = l (\frac{1}{\lambda}(\delta'_1 - \delta'_2))
%\end{equation*}
Note that this yields that $\delta_2 / \delta_1 \in \bF_{p^2}$.
We will prove that $\Span_{\bZ_p}\{w_1,w_3\}$ decays rapidly. If not, then $\alpha_1=s \alpha_3$, where $s\in \bF_p$. That is, $\delta_1+\delta_2=s\delta_1/c^{(n+e)}$. Equivalently, $\delta_2/\delta_1=(s-c^{(n+e)})/c^{(n+e)}$. As $c\notin \bF_{p^2}$, it follows that $\delta_2/\delta_1\notin \bF_{p^2}$, which is a contradiction.

\item Suppose that $\delta_1 = \delta_2$. We will show that $\Span_{\bZ_p}\{w_1,w_3\}$ decays rapidly, by showing that $\alpha_1,\alpha_3$ are $\bF_p$-linearly independent. Indeed, $\alpha_1 = 2(\nu + \mu)c^{(n +e)}$, and $\alpha_3 = \nu + \mu$. As $c \notin \bF_{p^2}$, the two quantities are $\bF_p$-linearly independent as required.

\item Suppose that $\delta_1 = - \delta_2$. The same argument as above works to show that $\Span_{\bZ_p}\{w_2,w_3\}$ decays rapidly. 
\end{enumerate}
Thus we have proved that there exists a rank $2$ submodule $W$ of $\Span_{\bZ_p}\{w_1,w_2,w_3\}$ which decays rapidly and the minimal $t$-valuation in the coefficient of $1/p^{n+1}$ is $v_z(2 + p^{n-e} + \cdots + p^{n+e-1} + 2p^{n + e}) + A(p + \cdots + p^{n-e-1})$.

On the other hand, the term with minimal $t$-adic valuation in the coefficient of $1/p^{m+1}$ of the last column of $F_{\infty}$ is the last column of the matrix in \Cref{sgtoprightdecay}(2) - it is easy to see that this last column is non-zero. Hence $w_5$ decays rapidly. The $t$-adic valuation of this term is $v_z(2 + p^{m-e} + \cdots+ p^{m+e-1} + p^{m + e}) + A(p + \hdots p^{m-e-1}) $, and we notice that this is always different from $v_z(2 + p^{n-e} + \cdots + p^{n+e-1} + 2p^{n + e}) + A(p + \cdots + p^{n-e-1})$ regardless of the values of $m,n$. Therefore, $W+\bZ_p w_5$ decays rapidly by the same argument at the end of Case 3 Subcase 1 above.
\end{proof}

\subsection{The case of inert Hilbert modular surfaces}
That we have a rank $3$ submodule of special endomorphisms follows directly from the Siegel case. Indeed, by \S\ref{Frob-H-inert} and \S\ref{sec_F_Sie}, the $F$-crystal $\bL_{\cris}$ in the setting of inert Hilbert modular surfaces is obtained by setting $z = 0$ in the Siegel case. The required decay follows directly from \S\ref{Case2sg}.

\end{document}